\documentclass{article}

\usepackage[utf8]{inputenc}
\usepackage{amsmath}
\usepackage{amsfonts}
\usepackage{amssymb}
\usepackage{amsthm}
\usepackage[english]{babel}
\usepackage{fontenc}
\usepackage{graphicx}
\usepackage{color}
\usepackage{mathrsfs}
\usepackage[T1]{fontenc} 
\usepackage{mathtools}
\usepackage{graphicx}
\graphicspath{ {./Images/} }
\usepackage[backend=biber]{biblatex}
\usepackage{csquotes}
\usepackage{dsfont}
\usepackage{tikz}
\usepackage{caption}
\usepackage{verbatim}
\usepackage{float}

\usetikzlibrary{calc}

\usepackage[normalem]{ulem}

\newcommand{\norm}[1]{\left\lVert #1 \right\rVert}

\newtheorem{defi}{Definition}

\newtheorem{rem}{Remark}

\newtheorem{prop}{Proposition}

\newtheorem{lemme}{Lemma}
\newtheorem{theo}{Theorem}

\newtheorem{asum}{Assumption}

\newcommand{\R}{\mathbb R}
\newcommand{\C}{\mathbb C}
\newcommand{\N}{\mathbb N}
\newcommand{\Z}{\mathbb Z}

\renewcommand{\leq}{\leqslant}
\renewcommand{\geq}{\geqslant}

\newcommand{\pinf}{+\infty}
\newcommand{\minf}{-\infty}

\newcommand{\dxx}{\dfrac{\textrm{d}^2}{\textrm{d}x^2}}

\newcommand{\equ}[1]{\underset{#1}{\sim}}

\allowdisplaybreaks

\DeclareMathOperator\supp{supp}

\usepackage{hyperref}

\newcommand{\Va}[1]{{\color{green}VA:  #1}}

\AtEveryBibitem{%
  \clearfield{issn} 
  \clearfield{doi} 

  \ifentrytype{online}{}{
    \clearfield{url}
  }
}

\DefineBibliographyStrings{english}{page = {}, pages = {}}

\title{Inverse resonance problem on the line for perturbations of Pöschl-Teller potentials}
\author{Valentin Arrigoni\footnote{Université Marie et Louis Pasteur, CNRS, LmB (UMR 6623), F-25000 Besançon, France, \href{mailto:valentin.arrigoni@math.cnrs.fr}{valentin.arrigoni@math.cnrs.fr}}}

\date{}

\begin{document}

\maketitle

\begin{abstract}
    We study an inverse resonance problem on the line in which we aim at determining a compactly supported and integrable perturbation of a fixed Pöschl-Teller potential. We define the resonances as the poles of the reflection coefficients with a negative imaginary part. Given the zeros and the poles of one of the reflection coefficients, we are able to determine uniquely the perturbation of a Pöschl-Teller potential when its support is to the left or to the right of zero on the line and also in the remaining cases under the addition of other hypothesis or extra spectral data. We also give asymptotics of the resonances and show that they are asymptotically located on two logarithmic branches.
\end{abstract}

\section{Introduction}
An inverse problem consists of determining the causes of a phenomenon based on empirical observations of its effects. In the inverse problem we are interested in, we consider a one-dimensional stationary Schrödinger equation,
\begin{equation}
-y''+V(x)y = z^2y, ~(x,z) \in \R \times \C, 
\label{Premiere}
\end{equation}
with a potential $V \in L^1(\R)$ which is a compactly supported perturbation of a Pöschl-Teller potential:
\begin{equation} \label{Pot}
  V(x) = \frac{\lambda}{\cosh^2(x)} + q(x), \quad \lambda \in \R^*_+, \quad q \in L^1_{comp}(\R).
\end{equation}
We also consider its two specific solutions with prescribed asymptotics denoted $f^\pm$: 
\[
f^\pm(x,z) \equ{x \to \pm \infty} e^{\pm ixz},
\]
called Jost solutions. From these solutions, we can define three Wronskians, $w(z) = [f^-(x,z),f^+(x,z)]$ and $s^\pm(z) = [f^+(x,\mp z),f^-(x,\pm z)]$ where $[f,g]:= fg'-f'g$ denote the Wronskian of functions $f$ and $g$. Under the general assumption $V \in L^1(\R)$, the function $w$ is always defined on the closed upper-half plane and analytic on the open upper-half plane, and its zeros with a positive imaginary part $ik$ for $k>0$ are such that $(ik)^2$ corresponds to \emph{eigenvalues} for the Schrödinger operator of \eqref{Premiere}. Indeed, for such a $z$, $f^+$ and $f^-$ are linearly dependent which ensure they are both in $L^2(\R)$. For our specific potentials $V$ given by (\ref{Pot}), we shall show that the function $w$ can be meromorphically continued to the lower half-plane and its zeros with a negative imaginary part are then called \emph{resonances}. Like the eigenvalues, for $z$ a resonance, $f^+$ and $f^-$ are still linearly dependent, but they are no longer in $L^2(\R)$. One can understand them as "eigenvalues without eigenvectors". 

We are interested here in an inverse resonance problem, that is we consider the question whether a potential $V$ can be uniquely determined from the spectral data consisting in the set of eigenvalues and resonances (the zeros of $w$). Using Hadamard's factorisation theorem and a universal asymptotic (see \eqref{HdamardW} and \eqref{EquivW}  below), we can determine $w$, and so the transmission coefficient $\mathcal{T}(z) = 2iz/w(z)$. To determine the reflection coefficients
\[
\mathcal{R}^\pm = \frac{s^\pm}{w},
\]
it is necessary to know $s^+$ or $s^-$. Since the scattering matrix is unitary, from $w$ one can deduce the product $s^\pm(z)s^\pm(-z)$ (see Lemma \ref{Lemme} below). In the case where the perturbation is even, $s^\pm(z) = s^\pm(-z)$, so $s^\pm$ is determined up to a sign. If $w$ does not vanish at $z=0$, then $s^\pm$ is fully determined as $w(0)=-s^\pm(0)$.
An application of the Gelfand-Levitan-Marchenko theory allows us to determine uniquely the potential in this case. Note that we recover similar results to the ones established by Zworski \cite{Zworski} for compactly supported potentials.
In the case where there is a half-bound state, Zworski showed however that there exists two potentials sharing the same resonances and eigenvalues.
Hence it is in general necessary to have additional spectral data in order to determine uniquely the potential. 

Before describing in more details our results, let us give a quick survey of the known results in inverse scattering on the line. Let us recall that in the Gelfand-Levitan-Marchenko theory, the knowledge of the eigenvalues, associated norming constants and one of the reflection coefficient allows to determine the so-called transformation operator, that in turns, determines uniquely the potential in $L_1^1(\R)$. See \cite[Chapters 1,4]{Levitan} for further explanations. Grébert and Weder \cite{GrebWed} proved that when the potential is a priori known on the half-line, and is in $L_1^1(\R)$, then the potential is uniquely determined by one of the reflection coefficients. To that end, they followed the same framework as Deift and Trubowitz \cite{DeiftTrubo}. 
Few years later, Hitrik \cite{Hitrik} showed that, for compactly supported potentials, the eigenvalues and norming constants can be determined from the reflection coefficient. 
Concerning the real line, Korotyaev \cite{Koro} showed that the eigenvalues, resonances and a sequence of $\pm 1$ allow to determine uniquely a compactly supported potential. Xu and Yang \cite{XuYang} proved uniqueness results for potential with support in $[0,1]$ which is a priori known on $[0,\frac{1}{2}]$ from the knowledge of the eigenvalues and resonances only. Finally Bledsoe \cite{Bledsoe} showed a stability estimate from the small eigenvalues, resonances and zeros of one of the reflection coefficients after proving that the zeros of $w$ and $s^-$ determine uniquely, up to a translation, a compacted supported potential.

Concerning the inverse resonance problem for the Schrödinger equation on the half-line, Korotyaev \cite{Korotyaev} proved that an integrable and compactly supported potential is uniquely determined from the eigenvalues and resonances. This result essentially uses the Marchenko inverse scattering method on the half-line. Another procedure is to use the Borg-Marchenko theory which asserts that the potential is uniquely determined from the Weyl-Titchmarsh function. So the question is then to determine the Weyl-Titchmarsh function from the eigenvalues and resonances. Brown and Weikard \cite{BrownWeikard} proved uniqueness for compactly supported perturbations of algebro-geometric potentials. More recently, Marletta, Shterenberg and Weikard \cite{MSW} established stability results from the small resonances, still for a compactly supported potential in $\R_+$. We finally mention the work of Borthwick, Boussaïd and Daudé \cite{BBD} in which they showed that, in the case of a compactly supported perturbation of an exponentially decreasing potential, eigenvalues and resonances are sufficient to determine uniquely the potential. 
The inverse resonance problem on the half-line was a lot studied.
We refer to the above cited works for other references.


In the present work, the terms inverse resonance problem have the same meaning as in the works of Zworski \cite{Zworski}, Korotyaev \cite{Koro} or Bledsoe \cite{Bledsoe}. Recall from Hitrik \cite{Hitrik} that in Gelfand-Levitan-Marchenko theory the reflection coefficient uniquely determines a compactly supported perturbation. The question we want to address now is whether it is possible to determine the reflection coefficients from the resonances.
In the case of even potentials with no half-bound state, as mentioned above, this is possible. 
In the general case, we pointed out earlier that the knowledge of $w$ only provides the zeros of $s^\pm(z)s^\pm(-z)$, hence the zeros of $s^\pm$ up to a sign.
This indetermination cannot be lifted, so we have to add to the scattering data the zeros of $s^\pm$ to determine this function. This choice, similar to the one of Bledsoe, is not the only one. We could have also chosen to add only the sign of the zeros of $s^\pm$ to the set of scattering data. This is not exactly an inverse resonance problem, but a generalisation, since on top of resonances we need to add another set of data.

We consider compactly supported perturbation of a Pöschl-Teller potential. The reflection coefficient will be determined as the quotient of $s^\pm$ and $w$. We will show that both $w$ and $s^\pm$ can be meromorphically extended to the whole complex plane with possible poles on $-i\N = \{-i,-2i,\dots\}$ for $w$ and $i\Z^* = \{ \pm i, \pm 2i, \dots \}$ for $s^\pm$. In order to apply Hadamard's factorisation theorem, we consider renormalisation $S^\pm$ and $W$ respectively of $s^\pm$ and $w$. These functions are analytic on the whole complex plane, and their zeros are determined by resonances, eigenvalues and zeros of $s^\pm$.
The present work can be understood as a generalisation of all these previous results to the class of compactly supported perturbations of Pöschl-Teller potential on the whole line. The techniques used in the first part of this work are similar to the ones developped by Bledsoe in \cite{Bledsoe} in his first part, and before him Marletta, Shterenberg and Weikard \cite{MSW} for example.

We want to point out that Equation \eqref{Premiere} comes from a geometrical problem. Let $\mathfrak{C}$ be an infinite hyperbolic cylinder 
\[
\mathfrak{C} \simeq (\R)_{x} \times (\mathbb{S}^{d-1})_{\omega}
\label{Hyperbolic}
\]
where $\mathbb{S}^{d-1}$ is the (d-1)-dimensional sphere ($d \geq 2$), equipped with the metric $g:= dx^2+\cosh^2(x)d\omega^2$. We recall that the Gaussian curvature of such a Riemannian surface is $-1$ (see Borthwick \cite{BorthwickBook} for further details on hyperbolic surfaces). Let $\Delta_g$ be the Laplace-Beltrami operator on this infinite hyperbolic cylinder. Then
\begin{equation}
\Delta_g  = \partial_x^2 + (d-1) \tanh(x) \partial_x + \Delta_{\mathbb{S}^{d-1}} \cosh^{-2}(x),
\label{Lap}
\end{equation}
where $\Delta_{\mathbb{S}^{d-1}}$ is the Laplace-Beltrami operator on the sphere $\mathbb{S}^{d-1}$. This operator is, through conjugation by $\cosh^{\frac{d-1}{2}}(x)$, equivalent to
\[
\partial_x^2+ \frac{\Delta_{\mathbb{S}^{d-1}}+(d-1)(d-3)/4}{\cosh^2(x)}-\frac{(d-1)^2}{4},
\]
\textit{i.e.}
\[
\cosh^{\frac{d-1}{2}}(x) \Delta_g \cosh^{-\frac{d-1}{2}}(x) f = \partial_x^2f + \left(\frac{\Delta_{\mathbb{S}^{d-1}}+(d-1)(d-3)/4}{\cosh^2(x)}-\frac{(d-1)^2}{4}\right) f,
\]
for $f \in L^2(\R \times S^{d-1},~ \mathrm{d}\omega_g)$.
We are interested in the Helmholtz equation: 
\begin{equation}
- \Delta_g f = \nu^2 f.
\label{Fourier}
\end{equation}
However, by doing the change of variable $f = \cosh^{-\frac{d-1}{2}}(x) v$ in \eqref{Fourier}, we obtain 
\[
-\partial_x^2 v+ \left(\frac{\Delta_{\mathbb{S}^{d-1}}+(d-1)(3-d)/4}{\cosh^2(x)}+\frac{(d-1)^2}{4}\right) v = \nu^2 v.
\]
We denote by $H_k$ the subspace of spherical harmonics of degree $k$. We recall that
\[
L^2(\mathbb{S}^{d-1}) = \bigoplus_{k \geq 0} H_k,
\]
and for every integer $k$, for every $Y_k \in H_k$,
\[
- \Delta_{\mathbb{S}^{d-1}} Y_k = k(k+d-2) Y_k.
\]
We refer to \cite{daudé2023local} for more details concerning spherical harmonics.
We are looking for solutions $v \in L^2(\R \times S^{d-1}, ~\mathrm{d}x\otimes\mathrm{d}\omega) = \bigoplus_{k \geq 0} L^2(\R) \otimes H_k$. It is enough to consider projections of solution of Equation $\eqref{Fourier}$ on the invariant sub-spaces $L^2(\R) \otimes H_k$ in the form
\[
v_k(x) Y_k(\omega).
\]
where $k \geq 0$ and $Y_k \in H_k$.
Then, functions $v_k$ satisfy the ordinary differential equation 
\[
-v_k''(x) + \frac{k(k+d-2)+(d-1)(3-d)/4}{\cosh^2(x)}v_k(x) = z^2 v_k(x),
\label{Equak}
\]
where $z^2 = \nu^2 - \frac{(d-1)^2}{4}$. This is exactly Equation \eqref{Premiere} without the compact perturbation $q$. Note that Equation \eqref{Premiere} with a potential $V$ being the sum of a Pöschl-Teller potential and a compactly supported perturbation corresponds to an infinite hyperbolic cylinder which is radially compactly perturbed.\footnote{More precisely, if we consider the metric 
$$
  g = dx^2 + a^2(x) \, d\omega^2, \quad a(x) = \cosh(x) + a_0(x), \quad a_0 \in C^2_{0}(\R),
$$
on the hyperbolic cylinder $\mathfrak{C}$ defined in \eqref{Hyperbolic}, then the stationary scattering would be exactly governed by equation (\ref{Premiere}).} We refer to \cite{BorthwickBook,Borthwick,ChristiansenZworski} for further information on the geometry, spectral theory and inverse resonance problem of infinite hyperbolic cylinder. 

At last, we emphasise that we would like to understand how waves propagate on this infinite hyperbolic cylinder and their scattering properties, since this model can be considered as a preliminary toy model in the analysis of black holes spacetime and their gravitational waves. Indeed, the stationary scattering of scalar waves in De Sitter–Schwarzschild black holes is governed by an infinite set of one-dimensional Schrödinger equations 
\[
-v'' + q_k(x) v = z^2 v
\]
similar to Equation \eqref{Equak}, where the potential $q_k$ is exponentially decreasing at $x = \pm \infty$ and possesses a unique non-degenerate maximum. We refer to \cite{JLJ} for more information concerning this toy model.

A second problem consider in this paper is the location of the resonances in the lower half-plane. The first to study the distribution of resonances is Regge \cite{Regge}. From a more mathematical perspective and considering a compactly supported potential, precise asymptotics were obtained by Zworski \cite{ZWORSKI1987277}. Guillopé and Zworski \cite{GuillopeZworski95,GuillopeZworski} provided information on the number of resonances on specific hyperbolic manifolds. Stepin and Tarasov \cite{StepTara07, StepTara09} gave asymptotics for resonances, considering a compactly supported or super exponentially decaying perturbation of the free potential. Similar results have already been obtained in \cite{BBD}, for non-compactly supported potential on the half-line only. We extend their work to our class of potentials (\ref{Pot}) defined on the whole line under the additional hypothesis that the compactly supported perturbation $q$ (or one of its derivatives) has some jump discontinuities at the boundary of its support. When the Pöschl-Teller potential is not perturbed, the resonances can be computed explicitly and lie on two vertical half-lines symmetric with respect to the imaginary negative axis. When we add a perturbation with a jump discontinuity, new resonances appear and asymptotically lie on two logarithmic branches symmetric with respect to the imaginary negative axis. See \cite{BBD} for a similar theoretical result. This finding confirms the numerical observations stated by Jaramillo, Panosso Macedo and Al Sheikh \cite{JLJ}. The techniques used in this section are inspired by the ones developed by Borthwick, Boussaïd and Daudé \cite{BBD}. 

This work is organised as follows. In a first section, we will give some classical definitions of scattering theory, some results on hypergeometric functions and results in the case of the non-perturbed Pöschl-Teller potential. Then in a second section, we state and prove the main theorems of this work, concerning the solution of an inverse resonance problem for a compactly perturbed Pöschl-Teller potential. Eventually, we study how resonances are perturbed by adding a compactly supported potential. This paper includes an appendix on hypergeometric functions, where we recall some standard results we use in the core of our analysis.

\section{Pöschl-Teller potentials}
We start with the following one-dimensional stationary Schrödinger equation 
\begin{equation}
-y''(x)+\frac{\lambda}{\cosh^2(x)}y(x)=z^2y(x),
\label{ES1}
\end{equation}
where $\lambda$ is a positive real fixed parameter, and $z$ a complex spectral parameter. The class of Pöschl-Teller potentials $\frac{\lambda}{\cosh^2(x)}$, was introduced for the first time in \cite{PT} in the study of quantum anharmonic oscillators.
Its Jost solutions denoted $f_0^\pm$ are solutions of \eqref{ES1} such that 
\begin{equation}
    f_0^\pm(x,z) \equ{x \to \pm \infty} e^{ \pm izx}.
    \label{J}
\end{equation}
They can explicitly be expressed in terms of hypergeometric functions: consider the change of variable
$\zeta:= \frac{1}{2}(1 - \tanh{x}) = (1+e^{2x})^{-1}$ and the function $f(x) = \cosh^{iz}(x)y(\zeta)$, where $f$ is a solution to \eqref{ES1}. 
Then we obtain an hypergeometric differential equation
%
\begin{equation}
\label{EH1}
\zeta(1-\zeta)y''+(c-(a+b+1)\zeta)y'-aby=0,
\end{equation}
with
\begin{equation}
a:= \frac{1}{2} - i z + \sqrt{\frac{1}{4}- \lambda},
\label{a}
\end{equation}
\begin{equation}
b:= \frac{1}{2} - i z - \sqrt{\frac{1}{4}- \lambda},
\label{b}
\end{equation}
\begin{equation}
c:= 1 - i z.
\label{c}
\end{equation}
Here, the notation $\sqrt{\cdot}$ refer to a complex determination of the square root. The choice is free, since it simply changes $a$ in $b$ and vice versa.
A solution to $\eqref{EH1}$ is given by an hypergeometric function
\[
\zeta \mapsto F(a,b,c;\zeta):= {}_2F_1(a,b,c;\zeta) = \sum_{n=0}^{\pinf} \frac{(a)_n(b)_n}{(c)_n}\frac{\zeta^n}{n!}.
\]
We refer to Definition \ref{Definition} in the Appendix for details about hypergeometric functions.
Note that $F(a,b,c;\zeta)$ are power series with radius of convergence 1. Furthermore,
\begin{equation}
    \lim_{\zeta \to 0}F(a,b,c;\zeta)=1.
    \label{limite0}
\end{equation}
At last, $F(a,b,c;\zeta)$ are well-defined for all $c \in \C \backslash \{0,-1,-2,\dots\}$, or equivalently for all $z \in \C \backslash (i\N)$
In the following, every hypergeometric function is denoted $F$ instead of ${}_2F_1$.
Under the condition $2c = a +b + 1$, Equation \eqref{EH1} is invariant under the change of variable $\zeta' = 1-\zeta$, so $\zeta \mapsto F(a,b,c;1-\zeta)$ is also a solution to \eqref{EH1}. Since the Wronskian between these two solutions is equal to
\[
[F(a,b,c;\zeta),F(a,b,c;1-\zeta)] = -\frac{\Gamma(c)^2}{\Gamma(a)\Gamma(b)}\zeta^{-c}(1-\zeta)^{-c},
\]
where $[f,g] = fg'-f'g$.
These two solutions are linearly independent if $a$ and $b$ are not non-positive integers.
We refer to the Appendix, especially Lemma \ref{LemmeAppendix2} and Kummer's formulas in Proposition \ref{Kummer} to obtain this Wronskian.
Hence Jost solutions are given by:
\begin{equation}
    f^+_0(x,z) = (2\cosh(x))^{iz} F\left(a,b,c;\frac{1}{1+e^{2x}}\right),
    \label{Jost+}
\end{equation}
\begin{equation}
    f^-_0(x,z) = (2\cosh(x))^{iz} F\left(a,b,c;\frac{e^{2x}}{1+e^{2x}}\right).
    \label{Jost-}
\end{equation}
Note that Equation \eqref{limite0} implies that these two hypergeometric functions satisfy 
\[
F\left(a,b,c;\frac{1}{1+e^{2x}}\right) \underset{x \to +\infty}{\to} 1 ~\text{and}~ F\left(a,b,c;\frac{e^{2x}}{1+e^{2x}}\right) \underset{x \to -\infty}{\to} 1.
\]
Since the Pöschl-Teller potential is even, we also have $f^+_0(-x,z) = f^-_0(x,z)$. At last, we can also write Jost solutions (see Appendix, Proposition \ref{NewExpF}) as:
\[
f^+_0(x,z) = e^{izx} F\left(c-a,c-b,c;\frac{1}{1+e^{2x}}\right),
\]
\[
f^-_0(x,z) = e^{-izx} F\left(c-a,c-b,c;\frac{e^{2x}}{1+e^{2x}}\right).
\]
Since Equation \eqref{ES1} has no first-order terms, then the Wronskian of two solutions is constant. We can so define two specific Wronskians, which are independent of the variable $x$, named Jost functions $w_0$ and $s^\pm_0$, by
\[
    w_0(z):= [f^{-}_0(\cdot,z),f^{+}_0(\cdot,z)]
    \label{w_0}
\]
and
\[
    s^\pm_0(z):= [f^+_0(\cdot,\mp z),f^-_0(\cdot,\pm z)].
    \label{s0}
\]
Relations between hypergeometric functions (details are given in the Appendix, see Proposition \ref{changement} for the derivation of hypergeometric functions, Lemmas \ref{LemmeAppendix1}, \ref{LemmeAppendix2} and Propositions \ref{changement}, \ref{Kummer}) give explicit expressions for $w_0$ and $s^\pm_0$:
\begin{equation}
w_0(z) = 2iz \frac{\Gamma(c-1)\Gamma(c)}{\Gamma(a)\Gamma(b)} = \frac{-2\Gamma(c)^2}{\Gamma(a)\Gamma(b)}, ~ z \in \C \backslash (-i\N),
\label{explicit}
\end{equation}
and
\[
s^\pm_0(z) = 2iz \frac{\Gamma(c)\Gamma(c-a-b)}{\Gamma(c-a)\Gamma(c-b)} = 2 \frac{\Gamma(c)\Gamma(2-c)}{\Gamma(c-a)\Gamma(c-b)}, z \in \C \backslash i\Z^*.
\]
There are three types of zeros of $w_0$. These interpretations still hold for any potential $V$. We will give a short explanation of their differences. Zeros with positive imaginary part $ik$ for $k>0$ are such that $(ik)^2$ are eigenvalues. Indeed, if $w_0(z)=0$, functions $f^+_0$ and $f^-_0$ are linearly dependent, and using Equation \eqref{J}, $f^+_0 \in L^2(\R)$ (and so is $f^-_0$). Zeros with a negative imaginary part are quite different. They are called resonances, and they can be understood as "eigenvalues without eigenvectors", due to the fact that for such $z$, $f^+_0 \notin L^2(\R)$ (and so is $f^-_0 \notin L^2(\R)$). We refer to \cite{DZ,Zworski} for detailed presentations of resonances. Their location is possible, due to the explicit formula \eqref{explicit}. For $\lambda \leq 1/4$, resonances appear for $z = -i\left(n+\frac{1}{2} \pm \sqrt{\frac{1}{4}-\lambda}\right)$, for any $n \geq 0$. While $\lambda>1/4$, they are located at $z = \pm \sqrt{\lambda - \frac{1}{4}} - i (n+\frac{1}{2})$, for any $n \geq 0$. Finally, when $w_0(0) = 0$, we say there is a half-bound state, or a zero resonance. 
We observe that, when $\lambda$ increases, the pair of resonances for the same integer $n$ gets closer along the imaginary axis, then for $\lambda = 1/4$ they are equal and finally split apart when $\lambda$ keeps increasing, See Figure \ref{Dessin}.

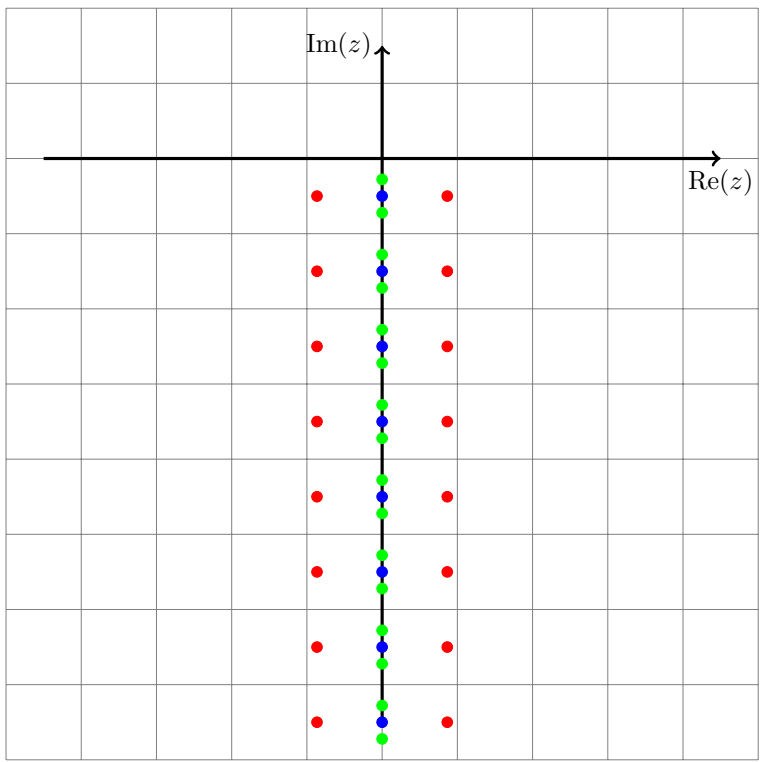
\begin{figure}[H]
\centering
\begin{tikzpicture}[scale=0.85]
    \draw[step=1cm, gray, very thin] (-5, -5) grid (5, 5);
    \draw[very thick, ->] (-4.5, 3) -- (4.5, 3) node[below]{$\Re(z)$};
    \draw[very thick, ->] (0, -4.5) -- (0, 4.5) node[left]{$\Im(z)$};
    \draw (1,3.1) -- (1,2.9) node[below] {1};
    \draw (-0.1,4) -- (0.1,4) node[left] {1};
    \filldraw[green](0,2.722) circle (2pt);
    \filldraw[green](0,2.278) circle (2pt);
    \filldraw[green](0,1.722) circle (2pt);
    \filldraw[green](0,1.278) circle (2pt);
    \filldraw[green](0,0.722) circle (2pt);
    \filldraw[green](0,0.278) circle (2pt);
     \filldraw[green](0,-0.278) circle (2pt);
    \filldraw[green](0,-0.722) circle (2pt);
     \filldraw[green](0,-1.278) circle (2pt);
    \filldraw[green](0,-1.722) circle (2pt);
     \filldraw[green](0,-2.278) circle (2pt);
    \filldraw[green](0,-2.722) circle (2pt);
     \filldraw[green](0,-3.278) circle (2pt);
    \filldraw[green](0,-3.722) circle (2pt);
     \filldraw[green](0,-4.278) circle (2pt);
    \filldraw[green](0,-4.722) circle (2pt);
    \filldraw[blue] (0,2.5) circle (2pt);
    \filldraw[blue] (0,1.5) circle (2pt);
    \filldraw[blue] (0,0.5) circle (2pt);
    \filldraw[blue] (0,-0.5) circle (2pt);
    \filldraw[blue] (0,-1.5) circle (2pt);
    \filldraw[blue] (0,-2.5) circle (2pt);
    \filldraw[blue] (0,-3.5) circle (2pt);
    \filldraw[blue] (0,-4.5) circle (2pt);
    \filldraw[red] (0.866,2.5) circle (2pt);
    \filldraw[red] (-0.866,2.5) circle (2pt);
    \filldraw[red] (0.866,1.5) circle (2pt);
    \filldraw[red] (-0.866,1.5) circle (2pt);
    \filldraw[red] (0.866,0.5) circle (2pt);
    \filldraw[red] (-0.866,0.5) circle (2pt);
    \filldraw[red] (0.866,-0.5) circle (2pt);
    \filldraw[red] (-0.866,-0.5) circle (2pt);
    \filldraw[red] (0.866,-1.5) circle (2pt);
    \filldraw[red] (-0.866,-1.5) circle (2pt);
    \filldraw[red] (0.866,-2.5) circle (2pt);
    \filldraw[red] (-0.866,-2.5) circle (2pt);
    \filldraw[red] (0.866,-3.5) circle (2pt);
    \filldraw[red] (-0.866,-3.5) circle (2pt);
    \filldraw[red] (0.866,-4.5) circle (2pt);
    \filldraw[red] (-0.866,-4.5) circle (2pt);
\end{tikzpicture}
\captionsetup{justification=centering}
\caption{Resonances in the lower half-plane:
\\ $\lambda=1/5$ in green, $\lambda = 1/4$ in blue and $\lambda =1$ in red.}
\label{Dessin}
\end{figure}

A straightforward calculation shows that
$[f^\pm_0(\cdot,z),f^\pm_0(\cdot,-z)]=\mp2iz$, so for every $z \in \C \backslash i\Z$, we have two fundamental systems of solutions for
Equation \eqref{ES1}: $\left(f^\pm_0(\cdot,z), f^\pm_0(\cdot,-z) \right)$. Every solution to Equation \eqref{ES1} can be written in these basis. So
\begin{equation*}
    f^\pm_0(\cdot,z) = \frac{w_0(z)}{2iz}f^\mp_0(\cdot,-z)+\frac{s^\mp_0(z)}{2iz}f^\mp_0(\cdot,z),~z \in \C \backslash i\Z.
\end{equation*}
The transmission coefficient $\mathcal{T}_0$ is defined by $\mathcal{T}_0(z):= 2iz/w_0(z)$ while the right and left reflection coefficients $\mathcal{R}^\pm_0$ are defined by $\mathcal{R}^\pm_0(z):= s^\pm_0(z)/w_0(z)$. They play an important role in the resolution of our inverse problem. Using the Gelfand-Levitan-Marchenko, the knowledge of these coefficients is sufficient to uniquely determine the original potential. More details about that are given in \cite{Bledsoe, Levitan} and in the proof of Theorem \ref{Theorem}.

Recall that the hypergeometric function $F(a,b,c;\zeta)$ are actually not defined at $c \in \{0,-1,-2,\dots\}$ (hence for $z \in \{-i,-2i,-3i,\dots\}$). These poles are created by the term $\Gamma(c)$ due to the presence of the Pochhammer coefficient $(c)_n$ at the denominator. To extend them to the whole complex plane, one can multiply the hypergeometric function by $1/\Gamma(c)$ and consider the analytic continuation of the function
\[
z \mapsto \frac{1}{\Gamma(c)}{F}(a,b,c;\zeta)
\]
to the whole complex plane.
Jost functions \eqref{Jost+} and \eqref{Jost-} are modified accordingly to this change: 
\begin{equation}
    \mathbf{f}^+_0(x,z):= \frac{1}{\Gamma(c)}f^+_0(x,z),
    \label{Jost++}
    \notag
\end{equation}
\begin{equation}
    \mathbf{f}^-_0(x,z):= \frac{1}{\Gamma(c)}f^-_0(x,z),
    \notag
    \label{Jost--}
\end{equation}
and so are the Jost functions $w_0$ and $s^\pm_0$, now defined on the whole complex plane by 
\begin{equation}
W_0(z):= [\mathbf{f}^-_0(x,z),\mathbf{f}^+_0(x,z)] ~\text{and}~ S^\pm_0(z):= [\mathbf{f}^+_0(x,\mp z),\mathbf{f}^-_0(x,\pm z)].
\label{JostFunction0Re}
\end{equation}
Since we will use classical results from complex analysis such as the Hadamard factorisation theorem, we will constantly use the above normalisation $W_0$ and $S^\pm_0$ in our calculations.
\section{Compact perturbations of a Pöschl-Teller potential}
\subsection{Jost solutions and scattering data}

We now consider a perturbation $q$ of the Pöschl-Teller potential.
We assume that this function $q$ is integrable and compactly supported, with its support in $[\alpha,\beta]$. With this new potential, Equation \eqref{Premiere} becomes
\begin{equation}
    -y''(x)+\left( \underbrace{\frac{\lambda}{\cosh^2(x)} + q(x)}_{:=V(x)} \right) y(x) = z^2y(x).
    \label{ES2}
\end{equation}
In terms of operator theory, we are interested in
\begin{equation}
H:= -\dxx + \frac{\lambda}{\cosh^2(x)} + q(x),
\label{H}
\end{equation}
it is a perturbation of
\begin{equation}
H_0:= -\dxx + \frac{\lambda}{\cosh^2(x)}.
\label{H0}
\end{equation}
The solutions of $Hf = z^2f$ are linked with those of $H_0 f=z^2 f$ through the use of transformation operators. We refer to Levitan \cite{Levitan} for more information about transformation operator theory.
\begin{defi}
We define the transformation operator $X^+$ which satisfies:
\begin{equation}
X^+H_0 = HX^+.
\label{Intertwining+}
\end{equation}
Furthermore, we assume that $X^+$ has the following expression:
\[
X^+(f)(x) = f(x) + \int_x^{\pinf} K^+(x,t) f(t) \, \mathrm{d}t,
\]
where $K^+$ is a function such that $X^+$ satisfies the intertwining property \eqref{Intertwining+}.
\label{DefX+}
\end{defi}
Following Levitan \cite[Chapter 1, Section 1.4, Paragraph 1.4.2]{Levitan} (see also \cite[Section 6.4]{Levitan}), this intertwining identity \eqref{Intertwining+} shows that if $f_0^+$ is the Jost solution of Equation \eqref{ES1}, then $Xf_0^+$ is the Jost solution associated with Equation \eqref{ES2}, \textit{i.e.}
\[
f^+(x,z) = f^+_0(x,z) + \int_x^{\pinf} K^+(x,t) f^+_0(t,z) \, \mathrm{d}t.
\]

Similarly, we can also obtain an expression for the Jost solution $f^-$ using the next definition.
\begin{defi} \label{DefX-}
We define the transformation operator $X^-$ which satisfies:
\begin{equation}
X^-H_0 = HX^-.
\label{Intertwining-}
\end{equation}
Furthermore, we assume that $X^-$ has the following expression:
\[
X^-(f)(x) = f(x) + \int_{\minf}^x K^-(x,t) f(t) \, \mathrm{d}t,
\]
where $K^+$ is a function such that $X^+$ satisfies the intertwining property \eqref{Intertwining-}.
\end{defi}
As a consequence, we will be able to write $f^-$ as
\[
f^-(x,z) = f^-_0(x,z) + \int_{\minf}^x K^-(x,t) f^-_0(t,z) \, \mathrm{d}t.
\]
Now, we will give some properties of these kernels $K^\pm$, starting with the existence of these two functions.
Let $E_\beta$ ($E_\alpha$, respectively) the set of continuous and bounded functions on $\{ (x,t) \in \R^2 ~|~ x \leq t \}$ ($\{ (x,t) \in \R^2 ~|~ t \leq x \}$, respectively) and with 
a support in $\{ (x,t) \in \R^2 ~|~ x \leq t \leq 2 \beta - x \}$ ($\{ (x,t) \in \R^2 ~|~ 2 \alpha - x \leq t \leq x \}$, respectively). We endow these two sets with the supremum norm.

Let $\Phi_\beta$ be the function defined on $E_\beta$ by:
\[
\begin{split}
        \Phi_\beta (L) (x,t) &:= \dfrac{1}{2}\int_{\frac{1}{2}(x+t)}^{\beta} q(\tau) \, \mathrm{d}\tau \\
        & + \int_{\frac{t+x}{2}}^{\beta}\int_{0}^{\frac{t-x}{2}} \left( V(\tau - u) - \frac{\lambda}{\cosh^2(\tau + u)} \right) L(\tau-u,\tau+u) \, \mathrm{d}u \, \mathrm{d}\tau.
\end{split}
\]
Similarly, we can define a function $\Phi_\alpha$ on $E_\alpha$: 
\[
\begin{split}
        \Phi_\alpha (L) (x,t) &:= \dfrac{1}{2}\int_{\alpha}^{\frac{x+t}{2}} q(\tau) \, \mathrm{d}\tau \nonumber \\
        + \int_{\alpha}^{\frac{x+t}{2}}&\int_{\frac{t-x}{2}}^0 \left( V(\tau - u) - \frac{\lambda}{\cosh^2(\tau + u)} \right) L(\tau-u,\tau+u) \, \mathrm{d}u \, \mathrm{d}\tau.
\end{split}
\]
Following the proof of \cite[Lemma 1]{BrownWeikard}, we obtain the existence of an integer $n$ such that $(\Phi_{\beta})^n$ is a contraction map, and similarly for $\Phi_{\alpha}$. According to Banach's fixed point theorem, $\Phi_\beta$ and $\Phi_\alpha$ have each a unique fixed point on $E_{\beta}$ and $E_{\alpha}$.
We define $K^+$ as the fixed point
\begin{align}
        K^+(x,t) &= \dfrac{1}{2}\int_{\frac{1}{2}(x+t)}^{\beta} q(\tau) \, \mathrm{d}\tau \nonumber \\
         + \int_{\frac{t+x}{2}}^{\beta}&\int_{0}^{\frac{t-x}{2}} \left( V(\tau - u) - \frac{\lambda}{\cosh^2(\tau + u)} \right) K^+(\tau-u,\tau+u) \, \mathrm{d}u \, \mathrm{d}\tau,
\label{K+equa}
\end{align}
for $x \leq t$. Similarly, we define $K^-$ as fixed point
\begin{align}
    K^-(x,t) & = \dfrac{1}{2}\int_{\alpha}^{\frac{x+t}{2}} q(\tau) \, \mathrm{d}\tau \nonumber \\
        + \int_{\alpha}^{\frac{x+t}{2}}&\int_{\frac{t-x}{2}}^0 \left( V(\tau - u) - \frac{\lambda}{\cosh^2(\tau + u)} \right) K^-(\tau-u,\tau+u) \, \mathrm{d}u \, \mathrm{d}\tau,
\label{K-equa}
\end{align}
for $t \leq x$.
We summarise the main properties of these two kernels in the following proposition:
\begin{prop}
The potential $q$ is an integrable and compactly supported function, with its support satisfying $\supp{q} \subset [\alpha,\beta]$.
The kernels $K^\pm$ are real-valued, and their supports are included in $\{(x,t) \in \R^2 ~|~x\leq t \leq2\beta-x\}$ for $K^+$ and $\{(x,t) \in \R^2 ~|~2 \alpha - x\leq t \leq x\}$ for $K^-$.
$K^\pm$ also satisfy:
\begin{equation}
    K^+(x,2 \beta -x) = 0 ~\text{and}~ K^+(x,x) = \frac{1}{2} \int_x^\beta q(s) \, \mathrm{d}s,
    \label{K+}
\end{equation}
and
\begin{equation}
    K^-(x,2 \alpha -x) = 0 ~\text{and}~ K^-(x,x) = \frac{1}{2} \int_\alpha^x q(s) \, \mathrm{d}s.
    \label{K-}
\end{equation}
Furthermore these functions are continuous and they satisfy the estimates
\begin{equation}
    |K^\pm(x,t)| \leq \frac{1}{2} \norm{q}_1 \exp(\norm{V}_{1,1} + \norm{\frac{\lambda}{\cosh^2(\cdot)}}_{1,1}),
    \label{estimation}
\end{equation}
where $\norm{f}_{1,1}:= \int_\R |t| |f(t)| \, \mathrm{d}t$.
Moreover, in the interior of their supports, $K^\pm$ have first order partial derivatives such that
\begin{align}
    & \left| \partial_u K^\pm(x,t) \pm \frac{1}{4}q\left( \frac{x+t}{2} \right) \right| \notag \\
    & \leq \frac{1}{2} \norm{q}_1 \left( \norm{V}_{1} + 2|\lambda| \right) \exp(\norm{V}_{1,1} + \norm{\frac{\lambda}{\cosh^2(\cdot)}}_{1,1}),
    \label{Deriv}
\end{align}
where $u$ stands for $x$ or $t$.
\label{Noyaux}
\end{prop}
\begin{proof}
    The proofs of \eqref{K+} (\eqref{K-}, respectively) and \eqref{estimation} both use the integral equation \eqref{K+equa} (\eqref{K-equa}, respectively) which defines the kernel $K^+$ ($K^-$, respectively). By taking 
    $t = x$ or $t = 2\beta-x$ ($t = 2\alpha-x$, respectively), we have \eqref{K+} (\eqref{K-}, respectively). For \eqref{estimation}, we also start with \eqref{K+equa} (and then \eqref{K-equa}), and apply Gronwall's inequality. Finally, we use \eqref{estimation} in the expression $\partial_u K^\pm(x,t) \pm \frac{1}{4}q\left( \frac{x+t}{2} \right)$ to prove \eqref{Deriv}. 
\end{proof}

We will need some information on higher derivatives of $K^\pm$. For this, following Zworski \cite{Zworski}, we introduce an additional assumption on the potential $q$. This assumption will play an important role in the study of the localisation of resonances, particularly in Theorems \ref{Theo4}, \ref{Nozero} and \ref{NozeroR-} below.
\begin{asum}
Let $p$ and $r$ be natural numbers.
The potential $q$ is an integrable and compactly supported function, with its support satisfying $\supp{q} \subset [\alpha,\beta]$. It has the following regularity: ${q}\in C^{N}([\alpha,\beta])$ for $N$ a natural integer such that $p,r \leq N$. Furthermore, we assume that 
\[
q(x) \sim C_1(x-\alpha)^{r-1}, ~ \text{as}~ x \to \alpha^+,
\]
and
\[
q(x) \sim C_2(x-\beta)^{p-1}, ~ \text{as}~ x \to \beta^-,
\]
where $C_1$ and $C_2$ are non-zero real constants.
\label{Asum}
\end{asum}
Then, we can prove some properties of the kernels $K^\pm$.
\begin{prop}
    Let $q$ be a real-valued, integrable and compactly supported function, with its support satisfying $\supp{q} \subset [\alpha,\beta]$. 
    \begin{itemize}
    \item The derivatives
    \[
    \partial_{\frac{t+x}{2}}\partial_{\frac{t-x}{2}}K^\pm,\quad \partial_{\frac{t-x}{2}}\partial_{\frac{t+x}{2}}K^\pm
    \]
    exist and are both equal to 
    \[
    \left(\frac{\lambda}{\cosh^2(t)} - V(x) \right) K^\pm(x,t).
    \]
    \item The two kernels satisfy the following wave equation
    \[
    \partial_x^2 K^\pm(x,t) - \partial_t^2 K^\pm(x,t) = \left(V(x) - \frac{\lambda}{\cosh^2(t)} \right) K^\pm(x,t)
    \]
    in a distributional sense.
    Moreover, when $q$ is $C^1(\R)$, the previous wave equation is satisfied in the strong sense.
    
    \item 
    Assume that $q$ satisfies Assumption \ref{Asum}.
    Then $K^+$ is in $C^{N+1}(\Omega^+)$, 
    where $\Omega^+ := \{ (x,t) \in \R^2 ~|~ x \leq t, ~ t \neq 2\beta-x ~\text{and}~ t \neq 2\alpha-x ~ \}$.
    Similarly, $K^-$ is in $C^{N+1}(\Omega^-)$,
    where $\Omega^- = \{ (x,t) \in \R^2 ~|~ t \leq x, ~ t \neq 2\beta-x ~\text{and}~ t \neq 2\alpha-x \}$.
    \item If we assume that $q$ satisfies Assumption \ref{Asum}, we have for $x$ a real number and $k$ an integer in $\{1,2,\dots,p-1\}$:
    \[
    \partial_t^k K^+(x,(2\beta - x)^-) = -\frac{1}{4} q^{(k-1)}(\beta^-) = 0,
    \]
    and
    \begin{equation}
        \partial_t^p K^+(x,(2\beta - x)^-) = -\frac{1}{4} q^{(p-1)}(\beta^-) \neq 0.
        \label{2beta}
    \end{equation}
    We also have, for $x$ a real number and $k$ an integer in $\{1,2,\dots,r-1\}$:
    \[
     \partial_t^r K^-(x,(2\alpha-x)^+) = -\frac{1}{4} q^{(r-1)}(\alpha^+) = 0,
    \]
    and
    \begin{equation}
        \partial_t^r K^-(x,(2\alpha-x)^+) = -\frac{1}{4} q^{(r-1)}(\alpha^+) \neq 0.
        \label{2alpha}
    \end{equation}
    \end{itemize}
    \label{PropNoyau}
\end{prop}
\begin{proof}
A similar proof of this proposition is given in \cite[Lemma 4]{BBD} when $q \in C^k$. For $L \in E_\beta$, and due to the continuity of translation in $L^1(\R)$, the quantity
\[
\int_{0}^{\frac{t-x}{2}} \left( V\left(\frac{t+x}{2} - u\right) - \frac{\lambda}{\cosh^2(\frac{t+x}{2} + u)} \right) L\left(\frac{t+x}{2}-u,\frac{t+x}{2}+u\right) \, \mathrm{d}u
\]
is absolutely continuous in $\frac{t-x}{2}$ and continuous in $\frac{t+x}{2}$, and the quantity 
\[
\int_{\frac{t+x}{2}}^{\beta} \left( V\left(\tau - \frac{t-x}{2}\right) - \frac{\lambda}{\cosh^2(\tau + \frac{t-x}{2})} \right) L\left(\tau-\frac{t-x}{2},\tau+\frac{t-x}{2}\right) \, \mathrm{d}\tau
\]
is absolutely continuous in $\frac{t+x}{2}$, and continuous in $\frac{t-x}{2}$. Then we have, according to the fundamental theorem of calculus,
\begin{align}
& \partial_{\frac{t+x}{2}} \left[ \int_{\frac{t+x}{2}}^{\beta}\int_{0}^{\frac{t-x}{2}} \left( V(\tau - u) - \frac{\lambda}{\cosh^2(\tau + u)} \right) L(\tau-u,\tau+u) \, \mathrm{d}u \, \mathrm{d}\tau \right] \notag \\
& = - \int_{0}^{\frac{t-x}{2}} \left( V\left(\frac{t+x}{2} - u\right) - \frac{\lambda}{\cosh^2(\frac{t+x}{2} + u)} \right) L\left(\frac{t+x}{2}-u,\frac{t+x}{2}+u\right) \, \mathrm{d}u.
\label{t+x}
\end{align}
Using Fubini-Tonelli's theorem, we also have
\begin{align}
    & \partial_{\frac{t-x}{2}} \left[ \int_{\frac{t+x}{2}}^{\beta}\int_{0}^{\frac{t-x}{2}} \left( V(\tau - u) - \frac{\lambda}{\cosh^2(\tau + u)} \right) L(\tau-u,\tau+u) \, \mathrm{d}u \, \mathrm{d}\tau \right] \notag \\
    & = \int_{\frac{t+x}{2}}^{\beta} \left( V\left(\tau - \frac{t-x}{2}\right) - \frac{\lambda}{\cosh^2(\tau + \frac{t-x}{2})} \right) L\left(\tau-\frac{t-x}{2},\tau+\frac{t-x}{2}\right) \, \mathrm{d}\tau.
    \label{t-x}
\end{align}
In the sense of distributions on the set $\{ (x,t) \in \R^2 ~|~ x < t \}$, we have
\begin{align}
    & \partial_{\frac{t+x}{2}} \partial_{\frac{t-x}{2}} \left[ \int_{\frac{t+x}{2}}^{\beta}\int_{0}^{\frac{t-x}{2}} \left( V(\tau - u) - \frac{\lambda}{\cosh^2(\tau + u)} \right) L(\tau-u,\tau+u) \, \mathrm{d}u \, \mathrm{d}\tau \right] \notag \\
    & = \partial_{\frac{t-x}{2}}\partial_{\frac{t+x}{2}}\left[ \int_{\frac{t+x}{2}}^{\beta}\int_{0}^{\frac{t-x}{2}} \left( V(\tau - u) - \frac{\lambda}{\cosh^2(\tau + u)} \right) L(\tau-u,\tau+u) \, \mathrm{d}u \, \mathrm{d}\tau \right] \notag \\
    & = \left( V(x) - \frac{\lambda}{\cosh^2(t)} \right) L(x,t).
    \label{PrevEq}
\end{align}
Note that if $q$ is continuous, then 
\[
- \int_{0}^{\frac{t-x}{2}} \left( V\left(\frac{t+x}{2} - u\right) - \frac{\lambda}{\cosh^2(\frac{t+x}{2} + u)} \right) L\left(\frac{t+x}{2}-u,\frac{t+x}{2}+u\right) \, \mathrm{d}u
\]
is differentiable with respect to $\frac{t-x}{2}$, similarly the quantity \eqref{t-x} is differentiable with respect to $\frac{t+x}{2}$ and \eqref{PrevEq} holds in a strong sense.
It directly follows that $K^+$ satisfies the wave equation
\[
\partial_x^2 K^+(x,t) - \partial_t^2 K^+(x,t) = \left( V(x) - \frac{\lambda}{\cosh^2(t)} \right) K^+(x,t),
\]
again in a distributional sense, and if $q$ is $C^1$, this equation holds in a strong sense.

For the next point, under the assumption that $q \in C^{N}(\R)$, where $N$ is defined in Assumption \ref{Asum}, we prove that $K^+$ is in $C^{N+1}(\Omega^+)$.
We define the application $\Phi$ by
\[
\Phi: L \mapsto \int_{\frac{t+x}{2}}^{\beta}\int_{0}^{\frac{t-x}{2}} \left( V(\tau - u) - \frac{\lambda}{\cosh^2(\tau + u)} \right) L(\tau-u,\tau+u) \, \mathrm{d}u \, \mathrm{d}\tau.
\]
Let $K_0(x,t):= \frac{1}{2}\int_{\frac{t+x}{2}}^{\beta} q(s) \, \mathrm{d}s$ and $K_{m+1}(x,t):= \Phi(K_m)(x,t)$.
For $l \in \N$, $\Phi$ maps $C^{l}(\Omega^+)$ to $C^{\min(l+1,N+1)}(\Omega^+)$. By induction, since $K_0 \in C^{N+1}(\Omega^+)$, it follows that $K_m \in C^{N+1}(\Omega^+)$ for any $m$. 
We can also rewrite $K^+$, using its definition as a fixed point,
\[
\begin{split}
K^+ & = K_0 + \Phi(K^+) \\
& = \sum_{j = 0}^{k} K_j + \Phi^{k+1}(K^+).
\end{split}
\]
Using the support of $K^+$ and the fact that this quantity is continuous, $K^+ \in C^{0}(\Omega^+)$. 
Then, all the terms in this sum are in $C^{N+1}(\Omega^+)$, so is $K^+$.

To conclude on the last item, we remark that if $q \in C^{k-1}_{pw}(\R)$ (piecewise $C^{k-1}$), where $k \in \N$, then the formulas \eqref{t+x} and \eqref{t-x} show that for $L \in C^{l}(\Omega^+)$, for $l \in \N \cup \{0\}$,
\[
\partial_{\frac{t+x}{2}} \left[ \int_{\frac{t+x}{2}}^{\beta}\int_{0}^{\frac{t-x}{2}} \left( V(\tau - u) - \frac{\lambda}{\cosh^2(\tau + u)} \right) L(\tau-u,\tau+u) \, \mathrm{d}u \, \mathrm{d}\tau \right]
\]
and 
\[
\partial_{\frac{t-x}{2}} \left[ \int_{\frac{t+x}{2}}^{\beta}\int_{0}^{\frac{t-x}{2}} \left( V(\tau - u) - \frac{\lambda}{\cosh^2(\tau + u)} \right) L(\tau-u,\tau+u) \, \mathrm{d}u \, \mathrm{d}\tau \right]
\]
are $C^{\min(l,k-1)}$ in $\frac{t+x}{2}$ and $\frac{t-x}{2}$, respectively. If we assume moreover that $q \in C^{k-2}(\R)$, then the quantity $\left( V(x) - \frac{\lambda}{\cosh^2(t)} \right) L(x,t)$ is $C^{\min(l,k-2)}(\Omega^+)$, and so the quantity
\[
\int_{\frac{t+x}{2}}^{\beta}\int_{0}^{\frac{t-x}{2}} \left( V(\tau - u) - \frac{\lambda}{\cosh^2(\tau + u)} \right) L(\tau-u,\tau+u) \, \mathrm{d}u \, \mathrm{d}\tau
\]
is $C^{\min(l+1,k)}(\Omega^+)$. From that, for $L = K^+$, the previous result shows that $K^+ - K_0$ is $C^{k}(\Omega^+)$. Since $q$ satisfies Assumption \ref{Asum}, for $k=p$, and using the fact that $K^+(x,t) = K_0(x,t) = 0$ for $t > 2 \beta -x$, we straightly deduce
\[
\partial_x^p(K^+ - K_0)(x,(2\beta - x)) = 0.
\]
Then, 
\[
\partial_x^p(K^+-K_0)(x,(2\beta - x)^-) = \partial_x^p K^+(x,(2\beta - x)^-) + \frac{1}{4} q^{(p-1)}(\beta^-)
\]
and so
\[
\partial_x^p K^+(x,(2\beta - x)^-) = -\frac{1}{4} q^{(p-1)}(\beta^-).
\]
Considering $K^-$, we can similarly prove that $K^- \in C^{N+1}(\Omega^-)$. Then, the same reasoning gives the other equality
\[
\partial_x^r K^-(x,(2\alpha-x)^+) = -\frac{1}{4} q^{(r-1)}(\alpha^+).
\qedhere
\]
\end{proof}
\begin{rem}
The second point of Proposition \ref{PropNoyau} shows that $X^+$ from Definition \ref{DefX+} and $X^-$ from Definition \ref{DefX-} satisfy respectively \eqref{Intertwining+} and \eqref{Intertwining-}.
\end{rem}
Using this transformation operators theory, we can finally define
\begin{align}
    f^+(x,z) &= f^+_0(x,z) + \int_x^{2\beta-x} K^+(x,t)  f^+_0(t,z) \, \mathrm{d}t~
    \label{Eq:Jost+}\\
    f^-(x,z) &= f^-_0(x,z) + \int_{2\alpha - x}^{x} K^-(x,t)  f^-_0(t,z) \, \mathrm{d}t~
    .\label{Eq:Jost-}
\end{align}

These two functions $f^\pm$ still satisfy \eqref{J} and \eqref{ES2}, so they are Jost solutions associated with Equation \eqref{ES2}.
As we did in the previous part, we can define Jost functions $w$ and $s^\pm$ as the two following Wronskians
\begin{equation*}
    w(z):= [f^-(\cdot,z),f^+(\cdot,z)]
\end{equation*}
and
\begin{equation*}
    s^\pm(z):= [f^+(\cdot,\mp z),f^-(\cdot,\pm z)].
\end{equation*}
A straightforward calculation for a well-chosen $x$ gives us that $[f^\pm(\cdot,z),f^\pm(\cdot,-z)] = \mp 2iz$, for $z \in \C \backslash i\Z$, then the two couples of Jost solutions $(f^\pm(\cdot,z),f^\pm(\cdot,-z))$ are two fundamental systems of solution for Equation \eqref{ES2}. So
\[
f^\pm(\cdot,z) = \frac{w(z)}{2iz}f^\mp(\cdot,-z)+\frac{s^\mp(z)}{2iz}f^\mp(\cdot,z),~z \in \C \backslash i\Z.
\]
We define for $z \in \C \backslash i\Z$
the transmission coefficient $\mathcal{T}(z):= 2iz/w(z)$, the right and left reflection coefficients $\mathcal{R}^\pm(z):= s^\pm(z)/w(z)$ and finally the scattering matrix 
\[
\mathcal{S}(z):= \begin{pmatrix}
\mathcal{T}(z) & \mathcal{R}^-(z) \\
\mathcal{R}^+(z) & \mathcal{T}(z) 
\end{pmatrix}.
\]
Again, due to the use of the hypergeometric functions, $f^\pm$ may have poles at $z \in i\Z$. To extend them to the whole complex plane, one have to consider the normalised Jost solutions
\begin{align}
    \mathbf{f}^+(x,z) &:= \frac{1}{\Gamma(c)}f^+(x,z),
    \label{fb+}
\end{align}
and
\begin{align}
    \mathbf{f}^-(x,z) &:= \frac{1}{\Gamma(c)}f^-(x,z).
    \label{fb-}
\end{align}
The normalised Jost functions $W$ and $S^\pm$ are defined similarly to $W_0$ and $S_0^\pm$ (see \eqref{JostFunction0Re}):
\[
W(z):= [\mathbf{f}^-(x,z),\mathbf{f}^+(x,z)] ~\text{and}~ S^\pm(z) := [\mathbf{f}^+(x,\mp z),\mathbf{f}^-(x,\pm z)].
\]
When $z \in \C \backslash (-i\N)$ for $W$  and $z \in \C \backslash i\Z^*$ for $S^\pm$, we have 
\begin{equation}
W(z) = \frac{w(z)}{\Gamma(1-iz)^2} ~\text{and}~ S^\pm(z) = \frac{s^\pm(z)}{\Gamma(1-iz)\Gamma(1+iz)}.
\label{WS}
\end{equation}
With this normalisation, we have $[\mathbf{f}^\pm(x,z),\mathbf{f}^\pm(x,-z)] = \frac{\mp 2iz}{\Gamma(1-iz)\Gamma(1+iz)}$. So, the expressions of the normalised Jost solutions in the basis $(\mathbf{f}^\pm(\cdot,-z),\mathbf{f}^\pm(\cdot,z))$ are now:
\[
\mathbf{f}^\pm(\cdot,z) = \frac{W(z)}{\frac{2iz}{\Gamma(1-iz)\Gamma(1+iz)}} \mathbf{f}^{\mp}(\cdot,-z) + \frac{S^\mp(z)}{\frac{2iz}{\Gamma(1-iz)\Gamma(1+iz)}}\mathbf{f}^\mp(\cdot,z), ~z \in \C \backslash i\Z.
\]
Coefficients $A$, $B$, $C$ and $D$ are defined as the coordinates in these basis, \textit{i.e.} 
\[
A(z):= \frac{W(z)}{\frac{2iz}{\Gamma(1-iz)\Gamma(1+iz)}} ~\text{and}~ B(z):= \frac{S^-(z)}{\frac{2iz}{\Gamma(1-iz)\Gamma(1+iz)}},
\]
\[D(z):= \frac{W(z)}{\frac{2iz}{\Gamma(1-iz)\Gamma(1+iz)}} ~\text{and}~ C(z):= \frac{S^+(z)}{\frac{2iz}{\Gamma(1-iz)\Gamma(1+iz)}}.
\]
Finally, we can define, always for $z \in \C \backslash i\Z$, the normalised transmission coefficient $T(z):= \frac{1}{A(z)} = \frac{2iz}{W(z)\Gamma(1-iz)\Gamma(1+iz)}$, the two normalised reflection coefficients $R^-(z):= \frac{B(z)}{A(z)} = \frac{S^-(z)}{W(z)}$ and $R^+(z):= \frac{C(z)}{A(z)} = \frac{S^+(z)}{W(z)}$, and finally the normalised scattering matrix $S(z)$. We remark here that $T(z) \Gamma(1-iz)^2 = \mathcal{T}(z)$ and 
\begin{equation}
R^\pm(z) \frac{\Gamma(1+iz)}{\Gamma
(1-iz)}= \mathcal{R}^\pm(z).
\label{R+}
\end{equation}
\subsection{An inverse resonance problem}
In this Section, we solve an inverse resonance problem for the class of compactly supported perturbation of a Pöschl-Teller potential introduced in Section 3.
\underline{From now, we assume that $q$ is real.}
We recall that $q$ is an integrable and compactly supported function, with its support satisfying $\supp{q} \subset [\alpha,\beta]$.
First, we start with a lemma on normalised Jost functions.
\begin{lemme}
    The functions $W$ and $S^\pm$ are entire and their growth order is at most one. Moreover, they satisfy:
    \begin{enumerate}
        \item $\overline{W(z)} = W(-\overline{z})$, $\overline{S^{\pm}(z)} = S^{\pm}(-\overline{z})$;
        \item $S^{-}(z) = S^{+}(-z)$;
        \item For $z \in \R^*$: $|A(z)|^2 = 1 + |B(z)|^2$;
        \item $W(z)W(-z) - \dfrac{4z^2}{\Gamma(1-iz)^2\Gamma(1+iz)^2} = S^{\pm}(z)S^{\pm}(-z)$;
         \item $W(0) = -S^{\pm}(0)$.
    \end{enumerate}
    \label{Lemme}
\end{lemme}
\begin{proof}
We start this proof by saying that $\frac{1}{\Gamma(c)}F(c-a,c-b,c;\zeta)$ for $\zeta \in (0,1)$ is an entire function of growth order at most one with respect to the complex parameter $z$ and uniformly with respect to the variable $\zeta$. We refer to Lemma \ref{LemmeEntireFunction} in the Appendix for a proof of this result.
Then, using \eqref{fb+} and \eqref{fb-}, and estimates \eqref{estimation} and \eqref{Deriv}, we prove that the function $\mathbf{f}^\pm$ and $(\mathbf{f}^\pm)'$ are entire and their growth order is at most one.
The proof of the other items of this lemma is similar to \cite[Lemma 2.1]{Bledsoe}, except for the third point. So we give a short proof. For a such $z$, a straightforward calculation shows:
    \[
    \begin{split}
        -& \frac{2iz}{\Gamma(1-iz)\Gamma(1+iz)} \\
        & = [\mathbf{f}^+(x,z),\mathbf{f}^+(x,-z)] \\
        & = [A(z)\mathbf{f}^-(x,-z)+B(z)\mathbf{f}^-(x,z),A(-z)\mathbf{f}^-(x,z)+B(-z)\mathbf{f}^-(x,-z)] \\
        & = A(z)A(-z)[\mathbf{f}^-(x,-z),\mathbf{f}^-(x,z)]+B(z)B(-z)[\mathbf{f}^-(x,z),\mathbf{f}^-(x,-z)] \\
        & = -\frac{2iz}{\Gamma(1-iz)\Gamma(1+iz)} \left( |A(z)|^2-|B(z)|^2 \right),
    \end{split}
    \]
since $A(-z) = \overline{A(z)}$, $B(-z) = \overline{B(z)}$ are consequences of the first two items.
\end{proof}
The first item of this Lemma \ref{Lemme} shows that the resonances must be symmetric with respect to the imaginary axis. If the potential were not real-valued, this property would not be true. 
The first, the fourth and the fifth points of Lemma \ref{Lemme} show that 0 is at most a simple zero of $W$. Furthermore, the second and the third points justify that $\mathcal{S}$ is a unitary matrix.
Since these functions are entire of growth order at most one, Hadamard's factorisation theorem gives the next expressions. Consider $\{ W_n \in \C ~|~ 0<|W_1| \leq |W_2| \leq \dots \}$ the set of non-vanishing zeros of $W$ listed according to multiplicity. We would like to characterise the elements of this set:
\begin{lemme}
    The zeros of $W$ are exactly the zeros of $w$ and the complex numbers $-ik_0$, for $k_0 \in \N$ and such that $(ik_0)^2$ is an eigenvalue.
    \label{ZeroW}
\end{lemme}
\begin{proof}
    We know from Formula \eqref{WS} that, for $z \in \C \backslash (-i\N)$,
    \[
    W(z) = \frac{w(z)}{\Gamma(1-iz)^2}.
    \]
    We straightly deduce that, on $\C \backslash (-i\N)$, the zeros of $W$ are the zeros of $w$. 
    Assume now that 
    $z_0 \in -i\N$ and $W(z_0)=0$. Note that for $z \in \C$, $$[\mathbf{f}^\pm(x,z),\mathbf{f}^\pm(x,-z)] = \frac{\mp 2iz}{\Gamma(1-iz)\Gamma(1+iz)}.$$ 
    Therefore the two normalised Jost solutions are colinear for $z = z_0$. This implies the existence of a complex constant $C$ such 
    that $W(z_0) = C W(-z_0)$. Then $-z_0$
    is also a zero of $W$ in $\C \backslash (-i\N)$, and thus of $w$. Therefore, $z_0^2$ is an eigenvalue.
    Conversely, if $(ik_0)^2$ is an eigenvalue, $k_0 \in \N$,
    the previous relation gives that $-ik_0$ is also a zero of $W$. Finally, we proved our result.
\end{proof}
If $-ik_0$ for $k_0 \in \N$ is a zero of $W$, then $ik_0$ is a zero of both $w$ and $W$, and reciprocally.
The knowledge of the eigenvalues and resonances is equivalent to the knowledge of the zeros of $W$. Then, by Hadamard's factorisation theorem we have
\begin{equation}
W(z) = z^m e^{a_0+a_1z} \prod_{n=1}^{\pinf} \left( 1 - \frac{z}{W_n} \right) e^{z/W_n},
\label{HdamardW}
\end{equation}
where $a_0,a_1 \in \C$ and $m \in \{0,1\}$.
\begin{lemme}
    The function $w$ satisfies the following universal asymptotic:
    \[
    w(z) \sim 2iz, ~z \to \infty, ~\Im(z) \geq 0.
    \]
    Similarly, we have also a universal asymptotic for $W$:
    \begin{equation}
    W(z) \sim \frac{2iz}{\Gamma(1-iz)^2},~ z \to \infty,~ \Im(z) \geq 0.
    \label{EquivW}
\end{equation}
\label{LemmeEquivW}
\end{lemme}
\begin{proof}
Since the poles of $w$ are in the lower half-plane, it is enough to consider an asymptotic for $w$. We are going to give another expression for $w$. Using the definition of $w$, \eqref{Eq:Jost+} and \eqref{Eq:Jost-}, we can prove that, for $x\geq \beta$ and $z \in \C \backslash i\Z^*$,
\begin{align}
w(z) & = 2iz\frac{\Gamma(c-1)\Gamma(c)}{\Gamma(a)\Gamma(b)} + F(c-a,c-b,c;1/(1+e^{2x})) \nonumber \\
& \times \Bigg[ \int_{2\alpha-x}^x  e^{-iz(t-x)} \left(\left( \partial_t - \partial_x \right) K^-(x,t)\right) F(c-a,c-b,c;e^{2t}/(1+e^{2t}))\, \mathrm{d}t \nonumber \\
& + \int_{2\alpha-x}^x  e^{-iz(t-x)} K^-(x,t) \partial_t F(c-a,c-b,c;e^{2t}/(1+e^{2t}))\, \mathrm{d}t \nonumber  \\
&  -2K^-(x,x)F(c-a,c-b,c;e^{2x}/(1+e^{2x})) \Bigg] \nonumber \\
& + \int_{2\alpha - x}^{x} K^-(x,t)  e^{-iz(t-x)} F(c-a,c-b,c;e^{2t}/(1+e^{2t}))  \, \mathrm{d}t \nonumber \\
& \times -\frac{2e^{2x}(c-a)(c-b)}{(1+e^{2x})^2c} F(c-a+1,c-b+1,c+1;1/(1+e^{2x})).
\label{NewW}
\end{align}
Using Stirling's formula:
\[
\Gamma(z) \sim (2\pi z)^{\frac{1}{2}} e^{-z} z^z, |\Im(z)| \leq \pi - \varepsilon,~ \varepsilon>0,~ z \to \infty,
\]
we have that $\frac{\Gamma(c-1)\Gamma(c)}{\Gamma(a)\Gamma(b)}$ tends to 1 as $z$ tends to infinity in the upper half-plane.
Now it is enough to prove that the four other terms of the sum are bounded as $|z|$ tends to infinity in the upper half-plane in order to conclude. For the integral of the fourth term, we have:
\[
\begin{split}
    & \int_{2\alpha - x}^{x} K^-(x,t)  e^{-iz(t-x)} F(c-a,c-b,c;e^{2t}/(1+e^{2t})) \, \mathrm{d}t \\
\end{split}
\]
Due to the fact that here $\Im(z) \geq 0$, we can apply Proposition \ref{Prop1} of the Appendix to obtain an upper-bound for the hypergeometric function which is uniform with respect to the parameter $\frac{e^{2t}}{1+e^{2t}}$. 
Hence using Estimates \eqref{estimation}, Proposition \ref{Prop1} and the triangle inequality prove that the above integral
is bounded as $z$ tends to $\pinf$ in the upper-half plane.
Hence, the last term in \eqref{NewW} is bounded
Note that
\begin{align*}
& \partial_t F(c-a,c-b,c;e^{2t}/(1+e^{2t})) \\
& = \frac{2e^{2t}}{(1+e^{2t})^2} \frac{(c-a)(c-b)}{c}F(c-a+1,c-b+1,c+1;e^{2t}/(1+e^{2t})),
\end{align*}
and using \eqref{t-x}, with $L = K^-$, we have that $\left( \partial_t - \partial_x \right) K^-$ is bounded, so the same proof holds for the remaining terms. 
\end{proof}
We now use the universal asymptotics given in Lemma \ref{LemmeEquivW} to show that the Jost function $W$ is uniquely determined by its zeros, or equivalently the poles of one of the reflection coefficients $R^+$ or $R^-$.
\begin{prop}
    The map between the set of the normalised Jost functions $W$ associated with a potential $V(x) = \frac{\lambda}{\cosh^2(x)}+q(x)$ such that $q$ is a real-valued, integrable and compactly supported function and the poles of the reflection coefficients is one-to-one.
    \label{PropW}
\end{prop}
\begin{proof}
    Let $\Tilde{V}$ be a potential defined by $\Tilde{V}(x) = \frac{\lambda}{\cosh^2(x)}+\Tilde{q}(x)$ and $\tilde{W}$ the normalised Jost function associated to this potential. We assume that $W$ and $\tilde{W}$ share the same zeros. Furthermore, as we did for $W$ in \eqref{HdamardW}, we can use the Hadamard's factorisation theorem to obtain the following expression for $\Tilde{W}$:
    \begin{equation}
        \Tilde{W}(z) = z^{\Tilde{m}}e^{\Tilde{a}_0+\Tilde{a}_1z}\prod_{n=1}^{\pinf} \left( 1 - \frac{z}{W_n} \right) e^{z/W_n},
        \label{HadamardWtilde}
    \end{equation}
    where $\Tilde{a}_0,\Tilde{a}_1 \in \C$ and $\Tilde{m} \in \{0,1\}$. Then, due to \eqref{HdamardW} and \eqref{HadamardWtilde},
    \[
    \begin{split}
    \frac{W(z)}{\Tilde{W}(z)} & = \frac{z^me^{a_0+a_1z}\prod_{n=1}^{\pinf} \left( 1 - \frac{z}{W_n} \right) e^{z/W_n}}{z^{\Tilde{m}}e^{\Tilde{a}_0+\Tilde{a}_1z}\prod_{n=1}^{\pinf} \left( 1 - \frac{z}{W_n} \right) e^{z/W_n}} \\
    & = z^{m-\Tilde{m}} e^{a_0-\Tilde{a}_0}e^{(a_1-\Tilde{a}_1)z}.
    \end{split}
    \]
    Since \eqref{EquivW} entails that $\frac{W(z)}{\Tilde{W}(z)} \to 1$ as $z \to \infty$ with $\Im(z) \geq 0$, we necessarily have $m=\Tilde{m}$, $a_0=\Tilde{a}_0$, $a_1=\Tilde{a}_1$, and so $W = \Tilde{W}$.
\end{proof}
As we can see, no further assumptions on the potential than $q$ being integrable, real and compactly supported are required to conclude on the uniqueness of the normalised Jost function $W$. We now consider $S^\pm$. As we did with $W$ in \eqref{HdamardW}, since they are also entire functions of growth order at most one, we can use the Hadamard's factorisation theorem and so write
\begin{equation}
    S^{\pm}(z) = (\mp1)^lz^le^{b_0\mp b_1z}\prod_{n=1}^{\pinf} \left( 1 \pm \frac{z}{S_n} \right) e^{\mp z/S_n},
    \label{S+-}
\end{equation}
where $b_0, b_1 \in \C$, $l \geq m$. The $S_n$, $n \in \N$, are the non-vanishing zeros of $S^-$ such that $0<|S_1| \leq |S_2| \leq \cdots$, listed according to multiplicity. The situation is the same as before with $W$, and we want now to characterize the elements $S_n$.
\begin{lemme}
    The set of zeros of $S^-$ is composed of the zeros of $s^-$ and the complex numbers $\pm i k_0$ for $k_0 \in \N$ and such that $(ik_0)^2$ is an eigenvalue.
\end{lemme}
\begin{proof}
    The proof is similar to the proof of Lemma \ref{ZeroW}.
    Firstly from \eqref{WS} we know that, for $z \in \C \backslash i\Z^*$, 
    \[
    S^-(z) = \frac{s^-(z)}{\Gamma(1-iz)\Gamma(1+iz)}.
    \]
    We deduce that on $\C \backslash (i\Z^*)$, the zeros of $S^-$ are the zeros of $s^-$. Assume now that $z_0 \in i\Z^*$ and $S^-(z_0) = 0$. As we did in the proof of Lemma $\ref{ZeroW}$, the two normalised Jost solutions are colinear for $z = z_0$. This implies the existence of two complex constants $C_1$ and $C_2$ such that $S^-(z_0) = C_1 W(z_0) = C_2 W(-z_0)$. Then, $\pm z_0$ are also zeros of $W$. Furthermore, one of these two zeros is in $i\N$, and so the square of this zero is an eigenvalue. Conversely, if $(ik_0)^2$ is an eigenvalue, with $k_0 \in \N$, the two previous relations show that $\pm ik_0$ is also a zero of $S^-$, which ends the proof.
\end{proof}
Similarly to $W$, we have that the knowledge of eigenvalues and zeros of $s^-$ is equivalent to the knowledge of the zeros of $S^-$.
Now, using Lemma \ref{Lemme}, we have
\[
W(z)W(-z) - \frac{4z^2}{\Gamma(1-iz)^2\Gamma(1+iz)^2} = S^\pm(z)S^\pm(-z).
\]
Evaluating this identity at $z=0$, we obtain that the right-hand side has a zero of order $2l$. The left-hand side is perfectly known, so $l$ is uniquely determined. To obtain the quantities $e^{b_0}$ and $b_1$, we have, as we previously did for the Jost function $w$, to give another expression of $s^+$. 
Let $\Tilde{c}:= 1+iz$, and similarly $\Tilde{a} := \frac{1}{2}+iz+\sqrt{\frac{1}{4}-\lambda}$ and $\Tilde{b}:=\frac{1}{2}+iz-\sqrt{\frac{1}{4}-\lambda}$. If $x \geq \beta$, then using \eqref{Eq:Jost+} and \eqref{Eq:Jost-}:
\begin{align}
s^+(z)
& = 2\frac{\Gamma(1-iz)\Gamma(1+iz)}{\Gamma(c-a)\Gamma(c-b)} \nonumber \\
& +\Bigg[\int_{2 \alpha - x }^x e^{-iz(t+x)} \left( \partial_x + \partial_t \right) K^-(x,t) F\left(c-a,c-b,c;\frac{e^{2t}}{1+e^{2t}}\right) \, \mathrm{d}t \nonumber \\
& + \int_{2 \alpha - x }^x e^{-iz(t+x)} K^-(x,t) \partial_t F\left(c-a,c-b,c;\frac{e^{2t}}{1+e^{2t}}\right) \, \mathrm{d}t \Bigg] \nonumber \\
& \times F(\Tilde{c}-\Tilde{a},\Tilde{c}-\Tilde{b},\Tilde{c};\frac{1}{1+e^{2x}}) \nonumber \\
 & + \frac{2e^{2x}}{(1+e^{2x})^2} \frac{(\Tilde{c}-\Tilde{a})(\Tilde{c}-\Tilde{b})}{\Tilde{c}} \times F(\Tilde{c}-\Tilde{a}+1,\Tilde{c}-\Tilde{b}+1,\Tilde{c}+1;\frac{1}{1+e^{2x}}) \nonumber \\
 & \times \int_{2 \alpha - x }^x e^{-iz(t+x)} K^-(x,t) F\left(c-a+1,c-b+1,c+1;\frac{e^{2t}}{1+e^{2t}}\right) \, \mathrm{d}t.
   \label{s}
\end{align}
This new writing of $s^+$ allows us to find its asymptotics under an additional assumption on $q$.
\begin{lemme}
    Let $V$ be a potential such that $V(x) = \frac{\lambda}{\cosh^2(x)}+q(x)$. Assume that $\supp{q} \subset [\alpha,\beta] \subset \R_-$. Let $z_n = i(\frac{4n+1}{2})$, for every integer n. Then:
    \[
    S^+(z_n) \equ{n \to \infty} \frac{2}{\Gamma(c-a)\Gamma(c-b)}.
    \]
    \label{Lemme3}
\end{lemme}
\begin{proof}
Since $z_n$ is not in $i\N$ for $n \in \N$, it is enough to prove that 
\[
 s^{+}(z_n) \equ{n \to \infty} 2\frac{\Gamma(1-iz_n)\Gamma(1+iz_n)}{\Gamma(c-a)\Gamma(c-b)}
\]
in order to conclude. To obtain this asymptotic, we use the representation of $s^+$ given in \eqref{s} with $x = \beta$. Note that $x \leq 0$ and Euler's reflection formula gives
\[
\Gamma(1-iz_n)\Gamma(1+iz_n) = \Gamma\left(1+\frac{4n+1}{2}\right)\Gamma\left(1-\frac{4n+1}{2}\right) = \frac{(4n+1)\pi}{2}.
\]
The rest of the proof is similar to the proof of Lemma \ref{LemmeEquivW}, using Proposition \ref{Prop1} to prove that the integral terms are bounded.
.
\end{proof}
This asymptotic allows us to obtain the following proposition.
\begin{prop}
    The map between the set of the normalised Jost functions $S^-$ associated with a potential $V(x) = \frac{\lambda}{\cosh^2(x)}+q(x)$ such that $q$ is a real-valued, integrable and compactly supported function whose support satisfies $\supp{q} \subset [\alpha,\beta] \subset \R_-$ and the set of zeros of one of the reflection coefficients in one-to-one.
    \label{PropS}
\end{prop}
\begin{proof}
The proof is identical to the one of Proposition \ref{PropW}, using Lemma \ref{Lemme3} and \eqref{S+-}. 
\end{proof}
We can now state an inverse-problem result.
\begin{theo}
    The map between the real-valued, integrable and compactly supported perturbation of a Pöschl-Teller potential $q$, with a support satisfying $\supp{q} \subset [\alpha,\beta] \subset \R_-$ (respectively in $\R_+$) and the set of zeros and poles of one of its reflection coefficients is one-to-one.
    \label{Theorem}
\end{theo}
\begin{rem}
Due to the second item of Lemma \ref{Lemme}, the zeros of $S^-$ are exactly the opposite of the ones of $S^+$. This means that if we know the zeros of the right reflection coefficient, we also know the zeros of the left reflection coefficient, and vice versa.
\end{rem}
\begin{proof}
Using symmetry with respect to the origin, it is enough to do the proof in the case $\supp{q} \subset \R_-$.
Therefore, under this assumption, Propositions \ref{PropW} and \ref{PropS} prove that Jost functions are 
fully determined by their zeros. So reflection coefficients $R^\pm$ are determined by their poles and zeros. 
According to Gelfand-Levitan-Marchenko theory, 
the knowledge of the zeros, poles of the reflection 
coefficients and the norming constant determines uniquely the potential. 
We refer to the book of Marchenko \cite[Chapter 3, Section 5]{Marchen} for a proof of the inverse problem on the line.
Following the idea of Hitrik \cite[Prop 3.3]{Hitrik} in the case of compactly supported perturbations, we introduce here some notations: let $K^+_0$ (respectively $K^+_q$) be the kernel of the transformation operator between $-\dxx$ and $H_0$ (resp. between $-\dxx$ and $H$) defined similarly as $K^+$. We recall that the operator $H_0$ (resp. $H$) are defined at \eqref{H0} (resp. \eqref{H}). Then, we can write the two Gelfand-Levitan-Marchenko equations:
\begin{equation}
K^+_0(x,y) + Q_0(x+y) + \int_{x}^{\pinf} K^+_0(x,t)Q_0(y+t) \, \mathrm{d}t = 0, ~ x < y,
\label{GLF0}
\end{equation}
and
\begin{equation}
K^+_q(x,y) + Q_q(x+y) + \int_{x}^{\pinf} K^+_q(x,t)Q_q(y+t) \, \mathrm{d}t = 0, ~ x < y,
\label{GLFq}
\end{equation}
where $Q_0$ is defined by 
\[
Q_0(x) = \frac{1}{2\pi} \int_{\minf}^{\pinf} e^{ikx} R^+_0(k) \, \mathrm{d}k+ \sum_{j=1}^{N_0} c_j^0 e^{-\beta_j^0 x},
\]
where the $c_j^0$ for $j \in \{1,\dots,N_0\}$ are the norming constants and the $\beta_j^0 > 0$ for $j \in \{1,\dots,N_0\}$ are such that $-(\beta_j^0)^2$ are the eigenvalues of the operator $H_0$.
Similarly
\[
Q_q(x) = \frac{1}{2\pi} \int_{\minf}^{\pinf} e^{ikx} R^+(k) \, \mathrm{d}k + \sum_{j=1}^N c_j e^{-\beta_j x},
\]
where the $c_j$ for $j \in \{1,\dots,N\}$ are the norming constants and the $\beta_j > 0$ for $j \in \{1,\dots,N\}$ are such that $-\beta_j^2$ are the eigenvalues of the operator $H$.
Due to the transformation operator theory (see \cite[Chapter 1]{Levitan}), we have:
\[
K_q^+(x,y) = K_0^+(x,y) + K^+(x,y) + \int_x^y K^+(x,s) K_0^+(s,y) \, \mathrm{d}s.
\]
As, for $x>\beta$, $K^+(x,y) = 0$, after subtracting \eqref{GLF0} from \eqref{GLFq}, we have
\begin{equation}
D_q(x+y) + \int_x^{\pinf} K_0^+(x,t)D_q(y+t) \, \mathrm{d}t = 0,~ \beta < x<y,
\label{Volt}
\end{equation}
where $D_q(x) := Q_q(x) - Q_0(x)$.
Equation \eqref{Volt} is a Volterra type equation, due to the fact that the kernel $K_0^+$ is integrable at $\pinf$. We refer to \cite[Chapter 1]{Levitan} for a proof of this result.
Then, we straightly obtain from \eqref{Volt} that $D_q(x) = 0$ for $x>2\beta$. 
The quantity $Q_0$ is known, then using the fact that $D_q(x)=0$ for $x>2\beta$, we have
\[
\sum_{j=1}^N c_j e^{-\beta_jx} = Q_0(x) - \frac{1}{2\pi} \int_{\minf}^{\pinf} e^{ikx} R^+(k) \, \mathrm{d}k,
\]
and so the norming constant $c_j$ for $j \in \{1,\dots,N\}$ are uniquely determined.
In conclusion, the knowledge of zeros and poles of one of the reflection coefficients uniquely determines the potential $q$.
\end{proof}
In the case of a potential $V$ such that $q$ has a support that contains zero, we are going to use other assumptions: we assume now that there is no half-bound state, \textit{i.e.} $W(0)\neq 0$ or there exists an eigenvalue in $\R^- \backslash (-\N)$. 
According to the work of Bledsoe \cite[Proposition 2.3 and Lemma 2.4]{Bledsoe},
we have the two following lemmas.
\begin{lemme}
    The complex number $b_1$ in \eqref{S+-} is an imaginary number.
    \label{5}
\end{lemme}
\begin{lemme}
    Let $V$ and $\Tilde{V}$ be two potentials defined as before, \textit{i.e.} $V(x) = \frac{\lambda}{\cosh^2(x)}+q(x)$ and $\tilde{V}(x) = \frac{\lambda}{\cosh^2(x)}+\tilde{q}(x)$. Let $\Tilde{W}$ and $\Tilde{S}^\pm$ be the Jost functions associated to the potential $\tilde{q}$. If $W = \Tilde{W}$ and $S^\pm(z) = e^{\mp 2i\alpha z}\Tilde{S}^\pm(z)$, then $\Tilde{V}(x) = V(x-\alpha)$ for every real number $x$ and for $\alpha$ a real number.
    \label{lemmeshift}
\end{lemme}
\begin{rem}
There is a misprint in the statement of Lemma \ref{lemmeshift} in \cite{Bledsoe}. It is written $S^\pm(z) = e^{\pm 2i\alpha z}\Tilde{S}^\pm(z)$. 
\end{rem}
Now, we can state another inverse problem theorem.
\begin{theo}
   The map between the real-valued, integrable and compactly perturbations of a Pöschl-Teller potential, without half-bound state or with an eigenvalue in $\R^- \backslash (-\N)$ and the set of zeros and poles of one of its reflection coefficients is one-to-one.
    \label{Theorem2}
\end{theo}
\begin{proof}
First, assume that there is no half-bound state. In this case, recalling Lemma \ref{Lemme}, we have $W(0) = -S^\pm(0) \neq 0$ and then $l=0$ in \eqref{S+-}, so $e^{b_0}$ defined in \eqref{S+-} is fully determined.
Now, assume that there exists an eigenvalue $-k_0^2$ in $\R^- \backslash (-\N)$: we use the knowledge of both $W$ and
the link between norming constants and residues of reflection coefficients at $z = i k_0$, where $k_0 \in \R_+^* \backslash \N$ and $-k_0^2$ is this eigenvalue
, to establish that
\[
\int_{\minf}^{\pinf} |\mathbf{f}^\pm(x,i k_0)|^2 \, \mathrm{d}x = i \frac{S^\pm(i k_0) \dot W(i k_0)}{\frac{4 k_0^2}{\Gamma(1+k_0)^2\Gamma(1-k_0)^2}} > 0,
\]
where $\dot W$ denotes the derivative of $W$ with respect to the complex variable $z$. Then $e^{b_0}$ is also completely determined in this case. See \cite[Theorem 2.2]{Bledsoe} for further information.
    Using a proof by contradiction, we are going to prove that $b_1$ is also fully determined by the two previous assumptions. Assume that $b_1 \neq \Tilde{b}_1$. Due to Lemma \ref{5}, they are both imaginary numbers, so one can write them $b_1 = i\beta_1$ and $\Tilde{b}_1 = i \Tilde{\beta}_1$ with $\beta_1,\Tilde{\beta}_1 \in \R$. Let $\alpha:= \beta_1 - \Tilde{\beta}_1$ be a non-zero real number. Then we have
\begin{equation}
    S^\pm(z) = e^{\mp i\alpha z}\Tilde{S}^\pm(z).
    \label{shift}
\end{equation}
Using Lemma \ref{lemmeshift}, Equation \eqref{shift} and $W=\Tilde{W}$, we have for every real number $x$
\begin{equation}
    \Tilde{V}(x) = V\left(x-\frac{\alpha}{2}\right).
    \label{egalite}
\end{equation}
This result holds for every real $x$, so one can choose $x$ large enough, such that $x \notin \supp{q}$ and $x \notin \supp{\Tilde{q}}$. For such a real number $x$, \eqref{egalite} becomes
\[
    \frac{\lambda}{\cosh^2(x)} = \frac{\lambda}{\cosh^2(x-\frac{\alpha}{2})}.
\]
This is absurd, so we have a contradiction with the assumption that $\alpha$ is a non-zero real number.
Then, $S^+$ is uniquely determined, and we have the result of the theorem.
\end{proof}
Actually, what we have shown in the proof above is that if ${b_0} = {\tilde{b}_0}$,
then necessarily $b_1 = \tilde{b}_1$.
Finally, in the case there is a half-bound state and all the eigenvalues, if any, are in $-\N$, in order to be able to conclude, we must add a quantity to our scattering data. This point is mentioned in the work of Korotyaev \cite{Koro} for compactly supported potential. We now state the last uniqueness theorem of this work.
\begin{theo}
    The map between the real-valued, integrable and compactly supported perturbations of a Pöschl-Teller potential and the set of zeros, poles of one of its reflection coefficients and the quantity $p = sign(i^l e^{b_0})$ (defined in \eqref{S+-}) is one-to-one.
    \label{Theorem3}
\end{theo}
\begin{proof}
We are able to determine, with Lemma \ref{Lemme}, as already mentioned, the quantity $e^{2b_0}$. Let $p:= sign(i^l e^{b_0})$, where $l$ and $b_0$ were defined in \eqref{S+-}. $p$ is well-defined because $i^l e^{b_0}$ is a real number. Suppose that $p$ is given, $i.e$ $p$ is in the scattering data set with the poles and zeros of the reflection coefficients. Then we can determine ${b_0}$. Following the proof of Theorem \ref{Theorem2}, using Lemma \ref{lemmeshift}, we then obtain ${b_1}$.
\end{proof}

One can see these three theorems as adjustments of results stated in the works of Bledsoe \cite{Bledsoe} and Korotyaev \cite{Koro}. 
They have shown similar uniqueness theorems considering a compact supported perturbation of the zero potential. 
The main difference with their works is that, due to the presence of a Pöschl-Teller potential, 
we 
have to
deal with more different cases. 
Due to the translation invariance of the support,
the conclusion of Theorem \ref{Theorem} cannot hold in their cases. These three theorems show that the normalised reflection coefficients of a potential $V$ such that $V(x) = \frac{\lambda}{\cosh^2(x)}+q(x)$ are completely determined under different assumptions. Then, due to the relation \eqref{R+} between $T$, $R^\pm$ and $\mathcal{T}$, $\mathcal{R}^\pm$, these quantities will also be fully determined.

\section{Location of the resonances at infinity for a compactly supported perturbation}

We know the resonances when $q=0$ (see Figure \ref{Dessin}). We would like to determine how these "initial" resonances behave when we perturb a Pöschl-Teller potential with a real-valued and integrable potential $q$ with compact support in $[\alpha,\beta]$. 

We mention here the works of Zworski \cite{Zworski,ZWORSKI1987277}, Bledsoe \cite{Bledsoe} and Stepin-Tarasov \cite{StepTara07, StepTara09} in which they gave asymptotics of resonances for compactly supported or super exponentially decreasing perturbations of the free operator.
Borthwick, Boussaïd, and Daudé \cite{BBD} consider the case of a compactly perturbed exponentially decreasing potential on the half-line.

We assume in this section that $q$ satisfies Assumption \ref{Asum}. We recall here that this assumption implies for $q$ to have a jump at $x=\beta$ for one of its derivative, and also a jump at $x=\alpha$ for another derivative.
Using the definition of the Wronskian $w$, we have for $z \in \C^- \backslash (-i\N)$.
\begin{align}
    w(z) & = [f^-((\alpha+\beta)/{2},z),f^+((\alpha+\beta)/{2},z)] \notag \\
    & = f^-((\alpha+\beta)/{2},z)(f^+)'((\alpha+\beta)/{2},z) \notag \\ 
    & -(f^-)'((\alpha+\beta)/{2},z)f^+((\alpha+\beta)/{2},z).
    \label{wsomme}
\end{align}
To study this function, we are going to split the domain $\C^-$ into two subdomains $S_1$ and $S_2$:
\begin{defi}
 We define a set $\mathcal{B}$ by
    \[
    \mathcal{B}:= \bigcup_{n=0}^{\pinf}B(-i(n+1),\rho),
    \]
    for $\rho>0$, and where $B(x,r)$ stands for the open ball of center $x$ and radius $r$.
    Then, let $\C^-:= \{ z \in \C ~|~  \Im(z) <0 \}$. We can define finally $S_1$ and $S_2$, two domains of the lower half-plane defined by
    \[
    S_1:= \left\{ z \in \C^- \backslash \mathcal{B} ~|~ - |\Re(z)| \geq \delta^{-1} \Im(z) \right\},
    \]
    and
    \[
    S_2:= \left\{ z \in \C^- ~|~ - |\Re(z)| < \delta^{-1} \Im(z) \right\},
    \]
    for $\delta >0$.
    \label{Def}
\end{defi}
\subsection{Location of resonances in $S_2$}
We start with the location of resonances on $S_2$.
We are going to give asymptotic expressions for each terms in the previous formula \eqref{wsomme}:
\begin{lemme} 
    Let $q$ be a real-valued, integrable and compactly supported potential such that $\supp{q} \subset [\alpha,\beta]$. We assume that $q$ satisfies Assumption \ref{Asum}.
    We have, for $z \in S_2$:
\begin{align}
& f^+((\alpha+\beta)/{2},z) = f^+_0((\alpha+\beta)/{2},z) +F\left(c-a,c-b,c;\frac{1}{1+e^{3\beta-\alpha}}\right) \times \nonumber \\
& \times(-1)^p \frac{\partial_t^pK^+(({\alpha+\beta})/{2},(3\beta-\alpha)/2)}{(iz)^{p+1}} e^{iz(3\beta-\alpha)/2} \left[ 1 + o(1) \right],
\label{f^+asymptotic}
\end{align}
and
\begin{align}
& (f^+)'((\alpha+\beta)/{2},z) \nonumber \\
& = iz f^+_0((\alpha+\beta)/{2},z) (1+ o(1)) + F\left(c-a,c-b,c;\frac{1}{1+e^{3\beta-\alpha}}\right) \times \nonumber \\
& \times (-1)^{p-1} \frac{\partial_t^{p-1}\partial_xK^+((\alpha+\beta)/{2},(3\beta-\alpha)/{2})}{(iz)^{p}}e^{iz(3\beta-\alpha)/2} \left[ 1+o(1) \right],
\label{f^+'asymptotic}
\end{align}
as $|z| \to \pinf$.
\label{Lemmef^+}
\end{lemme}

\begin{proof}
    The notations $o$ and $\sim$ hold for $|z|$ tending to $\pinf$ and $z \in S_2$. 
    We give the proof only for the first formula, due to the fact that the proof of the second asymptotic formula is similar.
    Since 
    \[
    f^+((\alpha+\beta)/2,z) = f^+_0((\alpha+\beta)/2,z) + \int_{\frac{\alpha+\beta}{2}}^{\frac{3\beta-\alpha}{2}} K^+((\alpha+\beta)/2,t) f^+_0(t,z) \, \mathrm{d}t,
    \]
    we have to modify the integral term to obtain the wanted formula.
    After $(p+1)$ integrations by parts, using item 4 of Proposition \ref{PropNoyau} to evaluate the kernel $K^+$ and its $k^{th}$ derivative at $t = 2\beta - x$ for $k \in \{0,1,\dots,p\}$, we have
\begin{align}
& \int_{\frac{\alpha+\beta}{2}}^{\frac{3\beta-\alpha}{2}} K^+((\alpha+\beta)/2,t) f^+_0(t,z) \, \mathrm{d}t \notag \\
& = \int_{\frac{\alpha+\beta}{2}}^{\frac{3\beta-\alpha}{2}} K^+((\alpha+\beta)/2,t) F\left(c-a,c-b,c; \frac{1}{1+e^{2t}} \right) e^{itz} \, \mathrm{d}t\notag \\
& = \sum_{k=0}^{p} \frac{(-1)^{k+1}}{(iz)^{k+1}}\partial_t^k \left.\left( K^+((\alpha+\beta)/2,t) F\left(c-a,c-b,c; \frac{1}{1+e^{2t}} \right)\right) \right|_{t=\frac{\alpha+\beta}{2}} \\
& + \left[ (-1)^p \frac{e^{itz}}{(iz)^{p+1}}\partial_t^p\left( K^+((\alpha+\beta)/2,t) F\left(c-a,c-b,c; \frac{1}{1+e^{2t}} \right)\right) \right]^\frac{3\beta-\alpha}{2}_{\frac{\alpha+\beta}{2}} \notag \\
& - \int_{\frac{\alpha+\beta}{2}}^{\frac{3\beta-\alpha}{2}} (-1)^p \partial_t^{p+1}\left[ K^+((\alpha+\beta)/2,t) F\left(c-a,c-b,c; \frac{1}{1+e^{2t}} \right)\right] \frac{e^{itz}}{(iz)^{p+1}} \, \mathrm{d}t \notag \\
& = (-1)^p \frac{\partial_t^pK^+((\alpha+\beta)/2,2\beta)}{(iz)^{p+1}} \notag \\
& \times F\left(c-a,c-b,c;\frac{1}{1+e^{3\beta-\alpha}}\right)e^{iz(3\beta-\alpha)/2} \left[ 1+o(1) \right].
\label{IPP}
\end{align}
Since we assume that the $(p-1)^{th}$ derivative of $q$ has a jump at $x = \beta$, Equation \eqref{2beta} tells us that $\partial_t^pK^+((\alpha+\beta)/2,((3\beta-\alpha)/2)^-) = -\frac{1}{4} q^{(p-1)}(\beta^-) \neq 0$. 
The $o(1)$ corresponds, firstly, to the terms between bracket evaluated at $t=(\alpha+\beta)/2$, which are negligible with respect to the one with the quantity $e^{iz(3\beta-\alpha)/2}$ and secondly to the remaining integral term after all the integrations by parts.

For this result, we need to control uniformly the hypergeometric function, see as a function of the complex parameter $z$ with respect to $\frac{1}{1+e^{2t}}$. To obtain such a result, we apply Proposition \ref{Prop1} of the Appendix. So, for every $t \in \left(\frac{\alpha+\beta}{2},\frac{3\beta-\alpha}{2}\right)$, we have:
\[
F\left(c-a,c-b,c;\frac{1}{1+e^{2t}} \right) = 1 +o(1),
\]
as $|c|$ tends to infinity.
So, using Leibniz formula, we can write the first term of the sum of the derivatives $$\partial_t^{p+1}\left[K^+((\alpha+\beta)/2,t) F\left(c-a,c-b,c; \frac{1}{1+e^{2t}} \right)\right]$$ as $$\partial_t^{p+1}K^+((\alpha+\beta)/2,t) F\left(c-a,c-b,c; \frac{1}{1+e^{2t}} \right),$$ and then we have that
\[
\begin{split}
    & e^{-iz\frac{3\beta-\alpha}{2}}\int_{\frac{\alpha+\beta}{2}}^{\frac{3\beta-\alpha}{2}} \partial_t^{p+1}K^+((\alpha+\beta)/2,t) F\left(c-a,c-b,c; \frac{1}{1+e^{2t}} \right) e^{itz} \, \mathrm{d}t \\
    & = e^{-iz\frac{3\beta-\alpha}{2}}\left(\int_{\frac{\alpha+\beta}{2}}^{\frac{3\beta-\alpha}{2}} \partial_t^{p+1}K^+((\alpha+\beta)/2,t) e^{itz} \, \mathrm{d}t \right. \\
    & \left. + \int_{\frac{\alpha+\beta}{2}}^{\frac{3\beta-\alpha}{2}} \partial_t^{p+1}K^+((\alpha+\beta)/2,t) e^{itz} o(1) \, \mathrm{d}t \right).
\end{split}
\]
Then, according to Lemma 2.2 in \cite{Titch}, we have that
\[
\int_{\frac{\alpha+\beta}{2}}^{\frac{3\beta-\alpha}{2}} \Phi(t) {e^{itz}} \, \mathrm{d}t = o(e^{-\frac{3\beta-\alpha}{2}\Im(z)}),
\]
for $\Phi$ an integrable function on $[\frac{\alpha+\beta}{2},\frac{3\beta-\alpha}{2}]$. So the integral term
\[
e^{-iz\frac{3\beta-\alpha}{2}}\int_{\frac{\alpha+\beta}{2}}^{\frac{3\beta-\alpha}{2}} \partial_t^{p+1}K^+((\alpha+\beta)/2,t) e^{itz} \, \mathrm{d}t
\]
tends to 0. For the second integral term, we apply the triangle inequality, use Proposition \ref{PropNoyau} which implies that the derivatives $\partial_t K^+((\alpha+\beta)/2,t)$ are bounded, and obtain the same conclusion as for the first one. Since the other terms of the Leibniz formula are made of derivatives of hypergeometric functions, we apply the same reasoning to these derivatives, using again Proposition \ref{Prop1}.
This concludes the proof.

Concerning the second formula, the only difference in the proof is that we make $p$ integrations by parts this time.
\end{proof}
We can state a similar lemma for the functions $f^-(0,z)$ and $(f^-)'(0,z)$.
\begin{lemme} 
 Let $q$ be a real-valued, integrable and compactly supported potential such that $\supp{q} \subset [\alpha,\beta]$. We assume that $q$ satisfies Assumption \ref{Asum}.
    We have, for $z \in S_2$:
    \begin{align}
   & f^-((\alpha+\beta)/2,z) = f^-_0((\alpha+\beta)/2,z) +F\left(c-a,c-b,c;\frac{e^{3\alpha-\beta}}{1+e^{3\alpha-\beta}}\right) \nonumber \\
   & \times (-1)^{r+1} \frac{\partial_t^rK^-((\alpha+\beta)/2,(3\alpha-\beta)/2)}{(-iz)^{r+1}}e^{-iz(3\alpha-\beta)/2} \left[ 1+o(1) \right],
   \label{f^-asymptotic}
    \end{align}
    and
    \begin{align}
    & (f^-)'((\alpha+\beta)/2,z) \nonumber \\
    & = -izf^-_0((\alpha+\beta)/2,z) (1+ o(1))+F\left(c-a,c-b,c;\frac{e^{3\alpha-\beta}}{1+e^{3\alpha-\beta}}\right) \nonumber \\
    & \times (-1)^r \frac{\partial_t^{r-1}\partial_xK^-((\alpha+\beta)/2,(3\alpha-\beta)/2)}{(-iz)^{r}}e^{-iz(3\alpha-\beta)/2} \left[ 1+o(1) \right],
    \label{f^-'asymptotic}
    \end{align}
    as $|z| \to \pinf$.
    \label{Lemmef^-}
\end{lemme}
We can deduce the corresponding result for $w$.
\begin{prop}
    Let $q$ be a real-valued, integrable and compactly supported with $\supp{q} \subset [\alpha,\beta]$. We assume that $q$ satisfies Assumption \ref{Asum}. Then, for $z \in S_2$,
    \begin{align}
    & w(z) = 2iz \left[ 1 + {o(1)} - \frac{Ae^{2iz(\beta - \alpha)}}{(iz)^{p+r+2}} \left( 1 +  {o(1)} \right)\right],
        \label{NewExW}
    \end{align}
    as $|z| \to \pinf$ and
    where
    \[
    \begin{split}
    A = (-1)^p q^{(r-1)}(\alpha^+) \times q^{(p-1)}(\beta^-)
    \end{split}
    \]
    is a real constant.
    \label{NewExpressionW}
\end{prop}

\begin{proof}
The notations $o$ and $\sim$ hold for $|z|$ tending to $\pinf$ in $S_2$.
    We can now use the asymptotic \eqref{HypGeoLim} for the hypergeometric functions (given in Proposition \ref{Asymptot}), \eqref{f^+'asymptotic} and \eqref{f^-asymptotic} in order to have, for $z \in S_2$:
   \[
    \begin{split}
        & f^-((\alpha+\beta)/2,z)(f^+)'((\alpha+\beta)/2,z) \\
        & = iz \left[ 1 +o(1) - \frac{Ae^{2iz(\beta - \alpha)}}{(iz)^{p+r+2}} \left( 1 + {o(1)} \right)\right], \\
    \end{split}
    \]
    where $A = (-1)^p q^{(r-1)}(\alpha^+) \times q^{(p-1)}(\beta^-)$.
    The same reasoning applied to \eqref{f^+asymptotic} and \eqref{f^-'asymptotic} gives:
        \[
    \begin{split}
        & f^+((\alpha+\beta)/2,z)(f^-)'((\alpha+\beta)/2,z) \\
        & = -iz \left[ 1 +o(1) - \frac{Ae^{2iz(\beta - \alpha)}}{(iz)^{p+r+2}} \left( 1 + {o(1)} \right)\right], \\
    \end{split}
    \]
    with the same $A$.
    We used above to define the constant $A$ the fact that $\partial_t^{r-1} \partial_x K^-((\alpha+\beta)/2,(3\alpha-\beta)/2) = \partial_t^r K^-((\alpha+\beta)/2,(3\alpha-\beta)/2)$, and similarly $\partial_t^{p-1} \partial_x K^+((\alpha+\beta)/2,(3\beta-\alpha)/2) = \partial_t^p K^+((\alpha+\beta)/2,(3\beta-\alpha)/2)$. This is a direct consequence of the fact that $K^+(x,2\beta-x)=K^-(x,2\alpha - x)=0$ for $x$ a real number.
    Finally, we can subtract the expression $f^+((\alpha+\beta)/2,z)(f^-)'((\alpha+\beta)/2,z)$ from the one of $f^-((\alpha+\beta)/2,z)(f^+)'((\alpha+\beta)/2,z)$ to obtain the wanted expression of $w$.
\end{proof}    
From \eqref{NewExW}, we can determine an asymptotic for the resonances belonging to the domain $S_2$.
\begin{theo}
    Let $V$ be a potential defined by $V(x) = \frac{\lambda}{\cosh^2(x)}+q(x)$, with $q$ real-valued, integrable and its support included in $[\alpha,\beta]$. Furthermore, we assume $q$ satisfies Assumption \ref{Asum}. The set of resonances in $S_2$ is given by a sequence $(\beta_j)_{j\in\Z^*}$ such that:
    \[
    \begin{split}
    \beta_{\pm j} & = \pm \frac{\pi}{2(\beta - \alpha)} \left( 2j + \frac{p+r+2}{2} \pm \frac{\left( 1-sign(A) \right)}{2} \right) \\
    & - \frac{i(p+r+2)}{2(\beta - \alpha)} \log \left( \frac{j \pi}{\beta - \alpha}\right) + \frac{i}{2(\beta - \alpha)} \log|A| +o(1),
    \end{split}
    \]
    as $j \to \pinf$ and with 
    $A = {(-1)^{p} q^{(r-1)}(\alpha^+) \times q^{(p-1)}(\beta^-)}$ a real constant.
    \label{Theo4}
\end{theo}
\begin{proof}
     The notations $o$ and $\sim$ hold for $|z|$ tending to $\pinf$ in $S_2$. 
     It follows from Hardy \cite[422-423]{Hardy}
     and Cartwright \cite{Cartwright1,Cartwright2} method that the zeros denoted $(\beta_j)_j$ of 
    \[
    1 + \underset{|z| \to \pinf}{o(1)} - \frac{Ae^{2iz(\beta - \alpha)}}{(iz)^{p+r+2}} \left( 1 +  \underset{|z| \to \pinf}{o(1)} \right)
    \]
    are located in the lower half-plane and have the following asymptotics:
    \[
    \begin{split}
    \beta_{\pm j} & = \pm \frac{\pi}{2(\beta - \alpha)} \left( 2j + \frac{p+r+2}{2} \pm \frac{\left( 1-sign(A) \right)}{2} \right) \\
    & - \frac{i(p+r+2)}{2(\beta - \alpha)} \log \left( \frac{j \pi}{\beta - \alpha}\right) + \frac{i}{2(\beta - \alpha)} \log|A| +o(1).
    \qedhere
    \end{split}
    \]
\end{proof}
\subsection{Absence of resonances in $S_1$}
Now, we want to locate resonances in $S_1$. Since \eqref{NewExW} doesn't hold in $S_1$, we propose to find another expression for the Wronskian $w$.

\underline{We first make the assumption that $0 \in (\alpha,\beta)$}, since the asymptotics of the hypergeometric functions will be easier in that case.

We have:
\begin{align}
    w(z) & = [f^-(0,z),f^+(0,z)] \notag \\
    & = f^-(0,z)(f^+)'(0,z)-(f^-)'(0,z)f^+(0,z) \notag \\
    & = f^-(0,z)(f^+)'(0,z)\left(1 - \frac{(f^-)'(0,z)f^+(0,z)}{f^-(0,z)(f^+)'(0,z)} \right) \notag \\
    & = f^-(0,z)(f^+)'(0,z)\left(1 - \frac{m_-(z)}{m_+(z)} \right),
    \label{WT}
\end{align}
with $m_+(z):= \frac{(f^+)'(0,z)}{f^+(0,z)}$, $m_-(z):= \frac{(f^-)'(0,z)}{f^-(0,z)}$, and for $z$ a complex number in the domain $S_1$ such that $f^-(0,z)(f^+)'(0,z) \neq 0$. We point out here that the two quantities $m_+$ and $m_-$ are called Weyl-Titchmarsh functions. 
See \cite{GesSim} and \cite{Teschl} for more details on Weyl-Titchmarsh functions and their use in inverse problems.
Zeros of $w$ are either zeros of $f^\pm(0,z)$ (or their derivatives) or zeros of the third factor in \eqref{WT}. We are going to study these three factors in $S_1$ in order to locate resonances in this domain.
In order to give asymptotics for $f^\pm(0,z)$ and $(f^\pm)'(0,z)$ in $S_1$, we need the following lemma. 
\begin{lemme}
     Let $\zeta$ be a complex number such that $\Re(\zeta) < \frac{1}{2}$.
     Then, there exist $\delta$ and $R$ positive real numbers such that:
     \begin{equation}
         |F(c-a,c-b,c;\zeta)| > \delta,
     \end{equation}
     for $|z| > R$ and $z \in S_1$.
     \label{HypGeoPos}
\end{lemme}
\begin{proof}
This is a direct consequence of Proposition \ref{Asymptot} of the Appendix, and more precisely asymptotics \eqref{HypGeoLim}.
\end{proof}
We are now ready to give the asymptotics for Jost solutions.
\begin{lemme}    \label{Lemmef+}
Let $q$ be a real-valued, integrable and compactly supported with $\supp{q} \subset [\alpha,\beta]$. We assume that $0 \in (\alpha,\beta)$ and $q$ satisfies Assumption \ref{Asum}.
    We have:
\begin{equation}
f^+(0,z) \sim (-1)^p \frac{\partial_t^pK^+(0,2\beta)}{(iz)^{p+1}}{F}\left(c-a,c-b,c;\frac{1}{1+e^{4\beta}}\right)e^{2iz\beta},
\label{Simf^+}
\end{equation}
and
\begin{equation}
(f^+)'(0,z) \sim (-1)^{p-1} \frac{\partial_t^{p-1}\partial_xK^+(0,2\beta)}{(iz)^{p}}{F}\left(c-a,c-b,c;\frac{1}{1+e^{4\beta}}\right)e^{2iz\beta},
\label{Simf^+'}
\end{equation}
when $|z| \to +\infty$ and $z \in S_1$.
\end{lemme}
\begin{proof}
    The notations $o$ and $\sim$ holds for $|z|$ tending to $\pinf$ and $z \in S_1$. 
    As the condition to apply the third item of Proposition \ref{Asymptot} are not satisfied in $S_1$, we need to adapt the formula \eqref{IPP} to this domain. We have:
\begin{align*}
& \int_0^{2\beta } K^+(0,t) f^+_0(t,z) \, \mathrm{d}t = \int_0^{2\beta } K^+(0,t) F\left(c-a,c-b,c; \frac{1}{1+e^{2t}} \right) e^{itz} \, \mathrm{d}t\notag \\
& = \sum_{k=0}^{p} \frac{(-1)^{k+1}}{(iz)^{k+1}}\partial_t^k \left.\left( K^+(0,t) F\left(c-a,c-b,c; \frac{1}{1+e^{2t}} \right)\right) \right|_{t=0} \notag \\
& + \left[ (-1)^p \frac{e^{itz}}{(iz)^{p+1}}\partial_t^p\left( K^+(0,t) F\left(c-a,c-b,c; \frac{1}{1+e^{2t}} \right)\right) \right]^{2 \beta}_0 \notag \\
& - \int_0^{2\beta } (-1)^p \partial_t^{p+1}\left[ K^+(0,t) F\left(c-a,c-b,c; \frac{1}{1+e^{2t}} \right)\right] \frac{e^{itz}}{(iz)^{p+1}} \, \mathrm{d}t \notag \\
& = g(z) + (-1)^p \frac{\partial_t^pK^+(0,2\beta)}{(iz)^{p+1}}F\left(c-a,c-b,c;\frac{1}{1+e^{4\beta}}\right)e^{2iz\beta} \left[ 1+o(1) \right],
\end{align*}
where the quantity $g(z)$ is defined by
\[
g(z) := \sum_{k=0}^{p} \frac{(-1)^{k+1}}{(iz)^{k+1}}\partial_t^k \left.\left( K^+(0,t) F\left(c-a,c-b,c; \frac{1}{1+e^{2t}} \right)\right) \right|_{t=0}.
\]
We refer to the proof of Lemma \ref{Lemmef^+}, and especially the use of Proposition \ref{Prop1} of the Appendix for the justification of the term $[1+o(1)]$ in the last expression of $ \int_0^{2\beta } K^+(0,t) f^+_0(t,z) \, \mathrm{d}t$.
From Lemma \ref{HypGeoPos}, we deduce that the quantity
\begin{equation}
    (-1)^p \frac{\partial_t^pK^+(0,2\beta)}{(iz)^{p+1}}F\left(c-a,c-b,c;\frac{1}{1+e^{4\beta}}\right)e^{2iz\beta}
    \label{expo}
\end{equation}
has an exponential growth.
We are now going to use Lemma \ref{abc} from the Appendix to show that $f_0^+(0,z)$ and $g$ have at most a polynomial growth in $S_1$ as $|z| \to \pinf$. Using Propositions \ref{Noyaux} and \ref{PropNoyau}, it is enough to prove that the hypergeometric functions in the definition of $f_0^+(0,z)$ and $g$ have at most a polynomial growth in $S_1$ as $|z| \to \pinf$. All the hypergeometric terms in the definitions of $f_0^+(0,z)$ and $g$ are of the form:
\[
\begin{split}
    2^{\frac{1}{2}-\frac{a'}{2}-\frac{b'}{2}}\frac{\sqrt{\pi}\Gamma\left(\frac{a'}{2}+\frac{b'}{2}+\frac{1}{2}\right)}{\Gamma\left(\frac{a'}{2}+\frac{1}{2}\right) \Gamma\left(\frac{b'}{2}+\frac{1}{2}\right)},
\end{split}
\]
with $a' = a+k$, $b' = b+k$ for $a$ (respectively $b$) defined in \eqref{a} (respectively \eqref{b}) and $k$ an integer between $0$ and $p$ due to Lemma \ref{abc}.
We use now Euler's reflection formula
\[
\Gamma(1-z)\Gamma(z) = \frac{\pi}{\sin(\pi z)},~z \in \C \backslash \Z,
\]
to modify the expression of gamma function, so we will be able to use Stirling formula. 
So we have
\[
\begin{split}
    & \frac{\sqrt{\pi}\Gamma\left(\frac{a'}{2}+\frac{b'}{2}+\frac{1}{2}\right)}{\Gamma\left(\frac{a'}{2}+\frac{1}{2}\right)\Gamma\left(\frac{b'}{2}+\frac{1}{2}\right)} = \frac{\sqrt{\pi}\Gamma\left(c'\right)}{\Gamma\left(\frac{a'}{2}+\frac{1}{2}\right)\Gamma\left(\frac{b'}{2}+\frac{1}{2}\right)} \\
    & = \sqrt{\pi} \frac{\Gamma\left(\frac{1}{2}-\frac{a'}{2}\right)\Gamma\left(\frac{1}{2}-\frac{b'}{2}\right)\sin\left(\pi\left(\frac{a'}{2}+\frac{1}{2}\right)\right)\sin\left(\pi\left(\frac{b'}{2}+\frac{1}{2}\right)\right) \pi}{\Gamma(1-c') \pi^2 \sin(\pi c')} \\
    & = \frac{1}{\sqrt{\pi}}\frac{\sin\left(\pi\left(\frac{a'}{2}+\frac{1}{2}\right)\right)\sin\left(\pi\left(\frac{b'}{2}+\frac{1}{2}\right)\right)}{\sin(\pi c')} \frac{\Gamma\left(\frac{1}{2}-\frac{a'}{2}\right)\Gamma\left(\frac{1}{2}-\frac{b'}{2}\right)}{\Gamma(1-c')},
\end{split}
\]
with $c' = \frac{a'}{2}+\frac{b'}{2}+\frac{1}{2}$.
Then, we apply Stirling formula
\[
\Gamma(z) \sim (2\pi z)^{\frac{1}{2}} e^{-z} z^z, |\text{Arg}(z)| \leq \pi - \varepsilon, \varepsilon>0,
\]
on the last term. Due to the fact that $z \in S_1$, we are in the domain of application of Stirling formula:
\[
\begin{split}
    \frac{\Gamma\left(\frac{1}{2}-\frac{a'}{2}\right)\Gamma\left(\frac{1}{2}-\frac{b'}{2}\right)}{\Gamma\left(1-c'\right)} & \sim \frac{\sqrt{2\pi\left(\frac{1}{2}-\frac{a'}{2}\right)}\sqrt{2\pi\left(\frac{1}{2}-\frac{b'}{2}\right)}}{\sqrt{2\pi\left(1-c'\right)}} \\
    & \times \frac{e^{-\left(\frac{1}{2}-\frac{a'}{2}\right)}e^{-\left(\frac{1}{2}-\frac{b'}{2}\right)}}{e^{-(1-c')}} \times \frac{\left(\frac{1}{2}-\frac{a'}{2}\right)^{\frac{1}{2}-\frac{a'}{2}}\left(\frac{1}{2}-\frac{b'}{2}\right)^{\frac{1}{2}-\frac{b'}{2}}}{(1-c')^{1-c'}}.
\end{split}
\]
The first two terms are $O(\sqrt{iz})$. Consider the last one, we have:
\[
\begin{split}
    \frac{\left(\frac{1}{2}-\frac{a'}{2}\right)^{\frac{1}{2}-\frac{a'}{2}}\left(\frac{1}{2}-\frac{b'}{2}\right)^{\frac{1}{2}-\frac{b'}{2}}}{(1-c')^{1-c'}} & = \frac{\left(\frac{1}{2}-\frac{a'}{2}\right)^{\frac{1}{2}}\left(\frac{1}{2}-\frac{b'}{2}\right)^{\frac{1}{2}}}{(1-c')} \\
    & \times\frac{\left(\frac{1}{2}-\frac{a'}{2}\right)^{-\frac{a'}{2}}\left(\frac{1}{2}-\frac{b'}{2}\right)^{-\frac{b'}{2}}}{(1-c')^{-c'}}.
\end{split}
\]
The first term is $O(1)$. Then
\begin{align*}
    & \frac{\left(\frac{1}{2}-\frac{a'}{2}\right)^{-\frac{a'}{2}}\left(\frac{1}{2}-\frac{b'}{2}\right)^{-\frac{b'}{2}}}{(1-c')^{-c'}} \\
    & = \frac{e^{(1-iz)\log(iz)}}{e^{\left(\frac{1}{4}-\frac{iz}{2}+\frac{\sqrt{\frac{1}{4}-\lambda}}{2}+\frac{k}{2}\right)\log\left(\frac{1}{4}+\frac{iz}{2}-\frac{\sqrt{\frac{1}{4}-\lambda}}{2}+\frac{k}{2}\right)}} \\
    & \times \frac{1}{e^{\left(\frac{1}{4}-\frac{iz}{2}-\frac{\sqrt{\frac{1}{4}-\lambda}}{2}+\frac{k}{2}\right)\log\left(\frac{1}{4}+\frac{iz}{2}+\frac{\sqrt{\frac{1}{4}-\lambda}}{2}+\frac{k}{2}\right)}} \\
    & = \frac{e^{\log(iz)}}{e^{\left(\frac{1}{4}+\frac{\sqrt{\frac{1}{4}-\lambda}}{2}+\frac{k}{2}\right)\log\left(\frac{1}{4}+\frac{iz}{2}-\frac{\sqrt{\frac{1}{4}-\lambda}}{2}+\frac{k}{2}\right)}e^{\left(\frac{1}{4}-\frac{\sqrt{\frac{1}{4}-\lambda}}{2}+\frac{k}{2}\right)\log\left(\frac{1}{4}+\frac{iz}{2}+\frac{\sqrt{\frac{1}{4}-\lambda}}{2}+\frac{k}{2}\right)}} \\
    & \times \frac{e^{-iz\log(iz)}}{e^{\frac{-iz}{2}\log\left(\frac{1}{4}+\frac{iz}{2}-\frac{\sqrt{\frac{1}{4}-\lambda}}{2}+\frac{k}{2}\right)}e^{\frac{-iz}{2}\log\left(\frac{1}{4}+\frac{iz}{2}-\frac{\sqrt{\frac{1}{4}+\lambda}}{2}+\frac{k}{2}\right)}}.
\end{align*}
The first term of this product is still a $O(\sqrt{iz})$. For the second one:
\[
\begin{split}
     \frac{e^{-iz\log(iz)}}{e^{\frac{-iz}{2}\log\left(\frac{1}{4}+\frac{iz}{2}-\frac{\sqrt{\frac{1}{4}-\lambda}}{2}\right)}e^{\frac{-iz}{2}\log\left(\frac{1}{4}+\frac{iz}{2}-\frac{\sqrt{\frac{1}{4}+\lambda}}{2}\right)}} \sim K e^{{iz} \log(\frac{1}{2})} \to 0, z \to \infty.
\end{split}
\]
Furthermore, we have
\[
\frac{\sin\left(\pi\left(\frac{a'}{2}+\frac{1}{2}\right)\right)\sin\left(\pi\left(\frac{b'}{2}+\frac{1}{2}\right)\right)}{\sin(\pi c')}  = \frac{\cos(\pi\sqrt{\frac{1}{4} - \lambda}) + \sin((iz+k)\pi)}{2\sin((iz+k)\pi)}.
\]
So $f^+_0(0,z)$ and $g(z)$ has at most a polynomial growth at infinity in $S_1$, and
\begin{equation}
    f^+_0(0,z)+g(z) = O(z^m),~ \text{for some} ~ m \in \N.
    \label{poly}
\end{equation}
Together, the exponential growth of \eqref{expo} and the polynomial growth of $\eqref{poly}$ imply that
\[
\begin{split}
{f}^+(0,z) & \sim \int_0^{2\beta } K^+(0,t) {f}^+_0(t,z) \, \mathrm{d}t \\
& \sim (-1)^p \frac{\partial_t^pK^+(0,2\beta)}{(iz)^{p+1}}{F}\left(c-a,c-b,c;\frac{1}{1+e^{4\beta}}\right)e^{2iz\beta}.
\qedhere
\end{split}
\]
\end{proof}
We obtain similar asymptotics for $f^-$:
\begin{lemme}
Let $q$ be a real-valued, integrable and compactly supported with $\supp{q} \subset [\alpha,\beta]$. We assume that $0 \in (\alpha,\beta)$ and $q$ satisfies Assumption \ref{Asum}.
   We have:
\begin{equation}
f^-(0,z) \sim (-1)^{r+1} \frac{\partial_t^rK^-(0,2\alpha)}{(-iz)^{r+1}}{F}\left(c-a,c-b,c;\frac{e^{4\alpha}}{1+e^{4\alpha}}\right)e^{-2iz\alpha},
\label{Simf^-}
\end{equation}
and
\begin{equation}
(f^-)'(0,z) \sim (-1)^{r} \frac{\partial_t^{r-1}\partial_xK^-(0,2\alpha)}{(-iz)^{r}}{F}\left(c-a,c-b,c;\frac{e^{4\alpha}}{1+e^{4\alpha}}\right)e^{-2iz\alpha},
\label{Simf^-'}
\end{equation}
when $|z| \to +\infty$ and $z \in S_1$.
    \label{Lemmef-}
\end{lemme}
Using these results, we now state a result on the limit of the quantity $\frac{m_-}{m_+}$ in $S_1$.
\begin{lemme}
     Let $q$ be a real-valued, integrable and compactly supported with $\supp{q} \subset [\alpha,\beta]$. We assume that $0 \in (\alpha,\beta)$ and $q$ satisfies Assumption \ref{Asum}. We have:
    \[
    \frac{m_-(z)}{m_+(z)} \to -1,~ |z| \to \infty, ~ z \in S_1 .
    \]
    \label{m-m+}
\end{lemme}
\begin{proof}
The notations $o$ and $\sim$ holds for $|z|$ tending to $\pinf$ in $S_1$. Assume that $|z| \to +\infty$. 
From Lemma \ref{HypGeoPos}, we infer that the quantity $f^-(0,z)({f}^+)'(0,z) \neq 0$ for $|z|$ large in $S_1$.
Then, the expression \eqref{WT} holds in $S_1$ for $|z|$ large. Since the quotient $\frac{m_-(z)}{m_+(z)}$ is well-defined, we can now give an asymptotic for it in $S_1$: 
Using Lemma \ref{Lemmef+}, and the two asymptotics \eqref{Simf^+} and \eqref{Simf^+'}, we have:
\[
\begin{split}
m_+(z) & \sim \frac{(-1)^{p-1} \frac{\partial_t^{p-1}\partial_xK^+(0,2\beta)}{(iz)^{p}}{F}\left(c-a,c-b,c;\frac{1}{1+e^{4\beta}}\right)e^{2iz\beta}}{(-1)^p \frac{\partial_t^pK^+(0,2\beta)}{(iz)^{p+1}}{F}\left(c-a,c-b,c;\frac{1}{1+e^{4\beta}}\right)e^{2iz\beta}} \\
& \sim - \frac{\partial_t^{p-1}\partial_xK^+(0,2\beta)\times iz}{\partial_t^pK^+(0,2\beta)} \\
& \sim -iz,~|z| \to +\infty,~ z \in S_1.
\end{split}
\]
since $\partial_t^{p-1}\partial_xK^+(0,2\beta) = \partial_t^pK^+(0,2\beta)$. 
Using the same proof, and the asymptotics \eqref{Simf^-'} and \eqref{Simf^-} of Lemma \ref{Lemmef-}, we have a similar result for $m_-$:
\[
\begin{split}
m_-(z) & \sim \frac{(-1)^{r} \frac{\partial_t^{r-1}\partial_xK^+(0,2\beta)}{(-iz)^{r}}{F}\left(c-a,c-b,c;\frac{e^{4\alpha}}{1+e^{4\alpha}}\right)e^{-2iz\alpha}}{(-1)^{r+1} \frac{\partial_t^r K^+(0,2\beta)}{(-iz)^{r+1}}{F}\left(c-a,c-b,c;\frac{e^{4\alpha}}{1+e^{4\alpha}}\right)e^{-2iz\alpha}} \\
& \sim \frac{\partial_t^{r-1}\partial_xK^-(0,2\alpha)\times iz}{\partial_t^rK^-(0,2\alpha)} \\
& \sim iz,~ |z| \to +\infty,~ z \in S_1.
\end{split}
\]
Then, combining these two equivalents:
\[
 \frac{m_-(z)}{m_+(z)} \to -1, ~|z| \to +\infty,~ z \in S_1,
\]
and so the result is proved.
\end{proof}
\begin{theo}
Let $V$ be a potential defined by $V(x)  = \frac{\lambda}{\cosh^2(x)}+q(x)$, with $q$ real-valued, integrable and its support included in $[\alpha,\beta]$. Furthermore, we assume that $0 \in (\alpha,\beta)$ and $q$ satisfies Assumption \ref{Asum}.
    There exists $R$ real positive number such that the meromorphic function $w$ has no zero $z$ in $S_1$ with $|z| > R$.
    \label{Nozero}
\end{theo}
\begin{proof}
As $\left( 1 - \frac{m_-(z)}{m_+(z)}\right)$ tends to $2$, it will not vanish for sufficiently large values of $z$. Since $f^-(0,z)(f^+)'(0,z) \neq 0$, then there is no zero $z$ for the meromorphic function $w$, and so no resonance, in $S_1$ for $|z|$ large.
\end{proof}
\begin{rem}
We mention here that in the case of an even perturbation $q$, we immediately have $\frac{m_-(z)}{m_+(z)}=-1$, and so $w(z) = 2f^+(0,z)(f^+)'(0,z)$.
\end{rem}
\underline{Finally, assume that $0 \notin (\alpha,\beta)$.} We suppose that $\supp{q} \subset \R_+$, since the proof of the other case is similar. We have that $f^-(0,z) = f^-_0(0,z)$, because the integral term vanishes. Since ${f}^-_0(0,z)$ has well-known zero in this domain, we can slightly modify the argument used in Theorem \ref{Nozero}.
\begin{theo}
Let $V$ be a potential defined by $V(x)  = \frac{\lambda}{\cosh^2(x)}+q(x)$, with $q$ real-valued, integrable and its support included in $[\alpha,\beta]$. Furthermore, we assume that $0 \leq \alpha$ and $q$ satisfies Assumption \ref{Asum}.
    There exists $R$ real positive number such that the meromorphic function $w$ has no zero $z$ in $S_1$ with $|z| > R$.
\label{NozeroR-}
\end{theo}
\begin{proof}
    The proof is similar to Theorem \ref{Nozero}, exploiting the fact that ${f}^-(0,z) = {f}^-_0(0,z)$.
    Using Lemma \ref{abc}, we know that the zeros of $f^-_0(0,z)$ are exactly the 
    \[
    z_j := -i \left(2 j - 1 \pm \sqrt{\frac{1}{4}-\lambda} - \frac{1}{2} \right),
    \]
    for $j \in \N$.
    An adaptation of the proof of Lemma \ref{Lemmef+} to the case $0 \leq \alpha$ gives us the same conclusion as the asymptotic \eqref{Simf^+} for $(f^+)(0,z)$ in $S_1$:
    \[
    f^+(0,z) \sim (-1)^p \frac{\partial_t^pK^+(0,2\beta)}{(iz)^{p+1}}{F}\left(c-a,c-b,c;\frac{1}{1+e^{4\beta}}\right)e^{2iz\beta}.
    \]
    Using Lemma \ref{HypGeoPos} applied to the quantity $F\left(c-a,c-b,c;\frac{1}{1+e^{2x}} \right)$ for $x>0$, we can prove that $(f^+)'(0,z) \neq 0$ asymptotically.
    So, for $z$ a complex number in $S_1$ which is not a $z_j$ for $j \in \N$, $f^-_0(0,z)(f^+)'(0,z) \neq 0$ for large value of $z$.
    Then, the expression 
    \[
    w(z) = f^-_0(0,z)(f^+)'(0,z) \left( 1 - \frac{m_-(z)}{m_+(z)} \right)
    \]
    for the Wronskian holds in $S_1 \backslash \{z_j ~|~ j \in \N \}$.
    Furthermore, we can prove that 
    \[
    \begin{split}
    m_-(z) & = \frac{(f_0^-)'(0,z)}{f_0^-(0,z)} \\
    & = -iz \tan(\frac{\pi}{2}a)\tan(\frac{\pi}{2}b) \left(1+o(1) \right),~ z \to \infty ,~ |z| \in S_1 \backslash \{z_j ~|~ j \in \N \},
    \end{split}
    \]
    and the proof of the asymptotic of $m_+$ in Lemma \ref{m-m+} is still valid, so 
    \[
    1- \frac{m_-(z)}{m_+(z)} = 1 - \tan(\frac{\pi}{2}a)\tan(\frac{\pi}{2}b) \left( 1+o(1) \right),~ |z| \to \infty ,~ z \in S_1 \backslash \{z_j ~|~ j \in \N \}.
    \]
    Let now consider $\rho >0$ and $z \in S_1 \backslash \bigcup_{j \in \N}B(z_j,\rho)$. Then the quantity $1- \frac{m_-(z)}{m_+(z)}$ cannot vanish, otherwise $\tan(\frac{\pi}{2}a)\tan(\frac{\pi}{2}b) = \frac{\sin(\frac{\pi}{2}a)\sin(\frac{\pi}{2}b)}{\cos(\frac{\pi}{2}a)\cos(\frac{\pi}{2}b)}$ would be equal to $1$. Hence, $\cos\left(\frac{\pi}{2}(a+b)\right)$ would vanish, and so $z$ would be in $-i\N$. This is a contradiction with the definition of $S_1$.
    So, the function $w$ has no zero $z$ for $|z|$ large in $S_1  \backslash \bigcup_{j \in \N}B(z_j,\rho)$.
    
    Let $j \in \N$, and $z \in B(z_j,\rho)$. We take $\rho$ small enough so that the zeros of
    \[
    \begin{split}
    (f^-_0)'(0,z) & = 2^{iz-1}\frac{ab}{c}\frac{\sqrt{\pi}\Gamma\left(\frac{a+1}{2}+\frac{b+1}{2}+\frac{1}{2}\right)}{\Gamma\left(\frac{a+1}{2}+\frac{1}{2}\right) \Gamma\left(\frac{b+1}{2}+\frac{1}{2}\right)},
    \end{split}
    \]
    which are defined by
    \[
    \tilde{z}_j := -i \left(2 j \pm \sqrt{\frac{1}{4}-\lambda} - \frac{1}{2} \right),
    \]
    are not in $B(z_j,\rho)$. In other words, for every $n \in \N$, we have that $|z_j - \tilde{z}_n| > \rho$.
    We can remark, once again for $\rho$ small enough, that it exists a constant $C >> 1$ such that for $z \in B(z_j,\rho)$, $\left|\tan(\frac{\pi}{2}a)\tan(\frac{\pi}{2}b)\right| \geq C$. Then we have, for $z \in \partial B(z_j,\rho)$ the border of the disk:
    \[
    \begin{split}
       & \left| 1 - \frac{m_-(z)}{m_+(z)} - \tan(\frac{\pi}{2}a)\tan(\frac{\pi}{2}b)\right| = \left| 1 + \tan(\frac{\pi}{2}a)\tan(\frac{\pi}{2}b)o(1) \right| \\
        & = \left|\tan(\frac{\pi}{2}a)\tan(\frac{\pi}{2}b)\right| \left| \frac{1}{\tan(\frac{\pi}{2}a)\tan(\frac{\pi}{2}b)} + o(1) \right| \\
        & < \left| \tan(\frac{\pi}{2}a)\tan(\frac{\pi}{2}b) \right|.
    \end{split}
    \]
    Applying Rouché's theorem, we prove that on the set $B(z_j,\rho)$ the quantity $1-\frac{m_-(z)}{m_+(z)}$ has no zero, which concludes the proof.
\end{proof}
In the case $\supp{q} \subset \R_-$, we can use the symmetries of the problem under the change of variable $\tilde{x} = -x$. A potential $q$ with its support in $\R_-$ has now its support in $\R_+$, where we know how to conclude.

\section{Acknowledgements}
I would like to thank my advisors, Nabile Boussaïd and Thierry Daudé, for their answers, comments and advice during the development of this work. I would also like to sincerely and deeply thank the anonymous referees for their feedbacks, which have helped to improve the quality of this work. This work was supported by the French 'Investissements d’Avenir’ program, project Agence Nationale de la Recherche (ISITE-BFC) (contract ANR-15-IDEX-0003). The work is part of the EUR project TACTICQ.

\appendix
\section{Appendix: some results on hypergeometric \newline functions}

We recall here some results on hypergeometric functions. See \cite{Lebedev,Temme} or \cite{NIST:DLMF,NIST} for classical references. The result on asymptotics presented in Proposition \ref{Asymptot} have been proved by Wagner \cite{Wagner}. 


\begin{defi}{(See \cite[15.2.1]{NIST:DLMF})}
Let $a$, $b$, $c$ and $\zeta$ be complex numbers, with $|\zeta| < 1$. Then, an hypergeometric function is defined by
    \[
\zeta \mapsto F(a,b,c;\zeta):= {}_2F_1(a,b,c;\zeta) = \sum_{n=0}^{\pinf} \frac{(a)_n(b)_n}{(c)_n}\frac{\zeta^n}{n!},
\]
where $(x)_n$ is the Pochhammer symbol, defined by $(x)_0:=1$ and for every positive integer $n$ 
\[
(x)_n:= x(x+1)\dots(x+n-1) = \frac{\Gamma(x+n)}{\Gamma(x)}.
\]
\label{Definition}
\end{defi}
We give here a proposition on asymptotics for the hypergeometric functions as $|c|$ tends to infinity:
\begin{prop}{(See \cite[Satz 1, p.443]{Wagner})}
Let $\delta$ denote an arbitrary small positive constant. Also let $a$, $b$, $\zeta$ be real or complex and fixed, and at least one of the following conditions be satisfied: 
\begin{enumerate}
\item $a$ and/or $b$ $\in \{0,-1,-2,\dots\}$;
\item $\Re(\zeta) < 1/2$ and $c$ such that $|c+n| \geq \delta > 0$, for every $n \in \{0,-1,-2,\cdots\}$;
\item $\Re(\zeta) = 1/2$ and $|\arg(c)| \leq \pi - \delta$;
\item $\Re(\zeta) > 1/2$ and $\alpha_{-}- \pi/2 + \delta \leq \arg(c) \leq \alpha_{+} + \pi/2 - \delta$  
\[
\alpha_{\pm}:= \arctan(\frac{\arg(\zeta)-\arg(1-\zeta)\mp\pi}{\ln|1-\zeta^{-1}|});
\]

\end{enumerate}
Then, for fixed $m \in \N$,
\begin{equation}
F(a,b,c;\zeta) = \sum_{n=0}^{m-1} \frac{(a)_n (b)_n}{(c)_n} \frac{\zeta^n}{n!} + O(c^{-m}), ~|c| \to \pinf.
\label{HypGeoLim}
\end{equation}
\label{Asymptot}
\end{prop}
\begin{prop}{(See \cite[15.5.1]{NIST:DLMF})}
     Let $a$, $b$ and $\zeta$ be complex numbers. Let $c$ be a complex number such that $c \in \C \backslash \{0,-1,\dots\}$. Then
     \[
     \frac{\mathrm{d}}{\mathrm{d}\zeta}F(a,b,c;\zeta) = \frac{ab}{c}F(a+1,b+1,c+1;\zeta).
     \]
     \label{changement}
\end{prop}
The following proposition is useful to consider the limit as $z$ tends to infinity of $F(a,b,c;\zeta)$, as we have shown earlier that, to be able to consider such a limit, the parameters $a$ and $b$ must be constant.
\begin{prop}{(See \cite[15.8.1]{NIST:DLMF})}
    Let $a$, $b$, $c$ and $\zeta$ be complex numbers such that $|Arg(1-\zeta)|<\pi$. Then we have
    \[
    F(a,b,c;\zeta) = (1-\zeta)^{c-a-b} F(c-a,c-b,c;\zeta). 
    \]
    \label{NewExpF}
\end{prop}
The next two lemmas give values of two different Wronskians. They are useful, combined with the next proposition, to calculate Wronskians such that $w_0$ and $s_0^\pm$.
\begin{lemme}{(See \cite[15.10.3, 15.10.12]{NIST:DLMF})}
     Let $a$, $b$, $c$ and $\zeta$ be complex numbers.
    Let $f_1$ and $f_2$ be two functions defined by
    \[
    \begin{split}
    f_1(\zeta) &:= F(a,b,c;\zeta) 
    \end{split}
    \]
    {and}
    \[
    \begin{split}
    f_2(\zeta) & := \zeta^{1-c} F(a-c+1,b-c+1,2-c;\zeta) \\
    & = \zeta^{1-c}\left(1-\zeta\right)^{c-a-b}F(1-a,1-b,2-c;\zeta).
    \end{split}
    \]
    Then 
    \[
    [f_1(\zeta),f_2(\zeta)] = (1-c) \zeta^{-c} (1-\zeta)^{c-a-b-1}.
    \]
    \label{LemmeAppendix1}
\end{lemme}

\begin{lemme}{(See \cite[Equation 15.10.5]{NIST:DLMF})}
    Let $a$, $b$, $c$ and $\zeta$ be complex numbers.
    Let $f_3$ and $f_4$ be two functions defined by
    \[
    f_3(\zeta):= F(a,b,a+b+c-1;1-\zeta) \]
    and
    \[
    f_4(\zeta):= (1-\zeta)^{c-a-b} F(c-a,c-b,c-a-b+1;1-\zeta).
    \]
    Then 
    \[
    [f_3(\zeta),f_4(\zeta)] = (a+b-c) \zeta^{-c} (1-\zeta)^{c-a-b-1}.
    \]
    \label{LemmeAppendix2}
\end{lemme}
This following proposition makes explicit the link between hypergeometric functions at $1-\zeta$ and hypergeometric functions at $\zeta$, and vice versa. These formulas are called ``Kummer's formulas'', and there exist 20 of them actually. We list here only the two that are used in our work.
\begin{prop}[Kummer's formulas]{(See \cite[15.10.17 and 15.10.21]{NIST:DLMF})}
    Let $a$, $b$, $c$ and $\zeta$ be complex numbers. 
    \begin{itemize}
        \item $f_3(\zeta)  = \frac{\Gamma(1-c)\Gamma(c)}{\Gamma(a-c+1)\Gamma(b-c+1)}f_1(\zeta) + \frac{\Gamma(c-1)\Gamma(c)}{\Gamma(a)\Gamma(b)}f_2(\zeta)$;
        \item $f_1(\zeta) = \frac{\Gamma(c)\Gamma(c-a-b)}{\Gamma(c-a)\Gamma(c-b)}f_3(\zeta)+ \frac{\Gamma(c)\Gamma(a+b-c)}{\Gamma(a)\Gamma(b)}f_4(\zeta)$,
    \end{itemize}
    whenever $c$ or $1-c$ are not in $-\N$.
    \label{Kummer}
\end{prop}

We may remark that, when the coefficients $a$, $b$ and $c$ satisfy the equality: $\frac{a}{2}+\frac{b}{2}+\frac{1}{2}=c$, then we have the following formula for an hypergeometric function at $\zeta = \frac{1}{2}$.
\begin{lemme}{(See \cite[15.4.28]{NIST:DLMF})}
    Let $a$ and $b$ be complex numbers. Then we have:
    \[
    F\left( a,b, \frac{a}{2}+\frac{b}{2}+\frac{1}{2}; \frac{1}{2} \right) = \frac{\sqrt{\pi}\Gamma\left(\frac{a}{2}+\frac{b}{2}+\frac{1}{2}\right)}{\Gamma\left(\frac{a}{2}+\frac{1}{2}\right)\Gamma\left(\frac{b}{2}+\frac{1}{2}\right)}.
    \]
    \label{abc}
\end{lemme}
We now mention a connection formula, between different hypergeometric functions, which is a slight modification of the connection formula for $f_3$ in Proposition \ref{Kummer}.

\begin{lemme}{(See \cite[Equation (25)]{Wagner})}
Let $a$, $b$, $c$ and $\zeta$ be complex numbers, such that $\zeta \in (0,1)$. We have, for $c \in \C \backslash \{0,-1,-2,...\}$:
\[
\begin{split}
& F(a,b,c;\zeta) \\
& = \frac{\Gamma\left(a-c+1\right)\Gamma\left(b-c+1\right)}{\Gamma(a+b-c+1)\Gamma(1-c)}F(a,b,a+b-c+1;1-\zeta) \\
& + \frac{\pi}{\sin(\pi c)} \frac{\Gamma\left(a-c+1\right)\Gamma\left(b-c+1\right)}{\Gamma(1-c)\Gamma(2-c)} \frac{\zeta^{1-c} (1-\zeta)^{c-b-a}}{\Gamma(a)\Gamma(b)} \\
& \times F(1-a,1-b,2-c;\zeta).
\end{split}
\label{ConneForm}
\]
\end{lemme}
Now, using Lemma \ref{ConneForm}, we give upper-bounds for hypergeometric functions that are uniform with respect to the complex parameter $\zeta$.
\begin{prop} \label{Prop1}
Let  $a$ and $b$ be fixed complex numbers, $\eta$ and $\rho$ be positive real numbers. Then, for $k \in \N \cup \{0\}$, we have:
\begin{enumerate}
    \item $$\sup_{\zeta \in [\eta,1-\eta]} \left| F(a,b,c;\zeta) -1 \right| \to 0,$$ as $|c| \to \pinf, ~c \in \C ~\text{and} ~ \Re(c) \geq0 ;$
    \item $$\sup_{\zeta \in [\eta,1-\eta]} \left| F(a,b,c;\zeta) -1 \right| \to 0,$$ as $|c| \to \pinf,~c \in \C~\text{and} ~ 0 > \Re(c) >-\delta |\Im(c)|;$
    \item  $\exists C_k >0, ~\forall c \in \C \backslash \bigcup_{n=0}^{\pinf} B(-n,\rho)$, $$\sup_{\zeta \in [\eta,1-\eta]} \left|\frac{\mathrm{d^k}}{\mathrm{d}\zeta^k}F(a,b,c;\zeta)\right| \leq C_k.$$
\end{enumerate}
\end{prop}
\begin{proof}
The notation $o(1)$ holds for $|c|$ tending to infinity.
We start by considering $c$ such that $\Re(c) \geq 0$.
Then by cutting the sum of the hypergeometric function in two, we obtain for $s_0 \in \N$ (to be fixed later):
\begin{align}
    F(a,b,c;\zeta) &  = 1 + \sum_{s=1}^{s_0-1} \frac{(|a|)_{s} (|b|)_{s}}{|(c)_{s} s!|} \zeta^s
    \label{Ligne1}\\
    & + \sum_{s=s_0}^{\pinf} \frac{(|a|)_{s} (|b|)_{s}}{|(c)_{s} s!|} \zeta^s, \label{Ligne2}
\end{align}
We want to prove that each of the two sums tends to zero as $z$ tends to infinity in the upper half-plane.
Then, denoting $m:= \max(|a|,|b|)$ and using that $\Re(c) \geq 0$, we have:
\[
\begin{split}
\left| \frac{|\zeta|^{s+1} \frac{(|a|)_{s+1} (|b|)_{s+1}}{|(c)_{s+1} (s+1)!|}}{|\zeta|^{s} \frac{(|a|)_{s} (|b|)_{s}}{|(c)_{s} s!|}} \right| & = |\zeta| \frac{(|a|+s)(|b|+s)}{|c+s|(s+1)} \\
&  \leq |\zeta| \frac{(m+s)^2}{|c+s|(s+1)}.\\
\end{split}
\]
Let $r>0$ such that $(1-\eta)(1+r) <1$ and $s_0$ be an integer which satisfies: if $s\geq s_0$ then $\frac{(s+m)^2}{(s+1)^2} \leq 1+r$.
Finally, we obtain
\[
|\zeta| \frac{(m+s)^2}{|c+s|(s+1)} \leq (1-\eta)(1+r) <1.
\]
Therefore, the quantity $\sum_{s=s_0}^{\pinf} \frac{(|a|)_{s} (|b|)_{s}}{|(c)_{s} s!|} \zeta^s$ is bounded by
\[
\begin{split}
& \frac{(|a|)_{s_0} (|b|)_{s_0}}{|(c)_{s_0} s_0!|} \zeta^{s_0} \sum_{s=s_0}^{\pinf} \left((1-\eta)(1+r)\right)^{s-s_0} = o(1), ~\text{as}~ |c| \to \infty \\
\end{split}
\]
Since for $s \geq 1$, we have $\frac{(|a|)_{s} (|b|)_{s}}{|(c)_{s} s!|} = o(1)$, we proved that for $\Re(c) \geq 0$ and uniformly with respect to the variable $\zeta \in [\eta,1-\eta]$,
\[
F(a,b,c;\zeta) = 1 + o(1).
\]
Now, we consider $c \in \C$ such that, for $\delta >0$, $0>\Re(c) > -\delta |\Im(c)|$. 
In this case, for $n \geq 2$, we have:
\[
\begin{split}
    |c+n|^2 & \geq (1+\delta^{-2})\Re(c)^2 + n^2 + 2n\Re(c) \\
    & = \left( \sqrt{1+\delta^{-2}}\Re(c) + \frac{n}{\sqrt{1+\delta^{-2}}}\right)^2 + \left( 1 - \frac{1}{1+\delta^{-2}} \right) n^2 \\
    & \geq \frac{\delta^{-2}}{1+\delta^{-2}} n^2 \\
    & = C_\delta n^2.
\end{split}
\]
So we directly deduce from that estimate that
\[
(c)_n \geq |c+1| C_\delta^{n-1} n!
\]
Finally, we have that 
\[
\begin{split}
\left| F(a,b,c;\zeta) - 1 \right| & = \left| \sum_{s=1}^{\pinf} \frac{(|a|)_{s} (|b|)_{s}}{|(c)_{s} s!|} \zeta^s \right| \\
& \leq \frac{m^2C_\delta}{|c+1|}\sum_{s=1}^{\pinf} \left(\frac{1-\eta}{C_\delta}\right)^s.
\end{split}
\]
So, for $\delta$ small enough, we also have that for $c \in \C$ such that, for $\delta >0$, $0>\Re(z) > \delta \Im(c)$, and uniformly with respect to the variable $\zeta \in (\eta,1-\eta)$,
\[
F(a,b,c;\zeta) = 1+o(1), ~\text{as}~ |c| \to \pinf.
\]
Now, assume that $c \in \C \backslash \bigcup_{n=0}^{\pinf} B(-n,\rho)$ such that $\Re(c)<0$ and which also satisfies, for $\delta >0$, $ \Re(c) \leq \delta \Im(c)$.
We use the connection formula given in Lemma \ref{ConneForm}:
\[
\begin{split}
& F(a,b,c;\zeta) \\
& = \frac{\Gamma\left(a-c+1\right)\Gamma\left(b-c+1\right)}{\Gamma(a+b-c+1)\Gamma(1-c)}F(a,b,a+b-c+1;1-\zeta) \\
& + \frac{\pi}{\sin(\pi c)} \frac{\Gamma\left(a-c+1\right)\Gamma\left(b-c+1\right)}{\Gamma(1-c)\Gamma(2-c)} \frac{\zeta^{1-c} (1-\zeta)^{c-b-a}}{\Gamma(a)\Gamma(b)} \\
& \times F(1-a,1-b,2-c;\zeta).
\end{split}
\]
Then, due to the assumption on $\Re(c)<0$, we can apply Stirling formula to prove that the both fractions of $\Gamma$ functions tend to 1.
Furthermore, due to the facts that $1-\zeta \in [\eta,1-\eta]$ and the location of $2-c$, we can apply our two previous results to $F(a,b,2-c;1-\zeta)$ in order to obtain that
\[
F(a,b,2-c;1-\zeta) = 1+o(1), ~\text{as}~ |c| \to \pinf.
\]
Similarly, we also have
\[
F(1-a,1-b,2-c;\zeta) = 1+o(1),~\text{as}~ |c| \to \pinf.
\]
Moreover, since $|c+n| > \delta$ for every $n$, the quantity $\frac{\pi}{\sin(\pi c)}$ is also bounded. 
The quantity $\zeta^{1-c} (1-\zeta)^{c-b-a}$ is bounded by $|1-\eta|^{-\Re(b)-\Re(a)+1}\left| \frac{1-\eta}{\eta} \right|^{\Re(c)}$, so finally we can conclude that it exists a constant $K$ such that, for every $c \in \C$ such that $\Re(c) <0$ and $c \in \C \backslash \bigcup_{n=0}^{\pinf} B(-n,\rho)$, and for every $\zeta \in [\eta,1-\eta]$, we have 
\begin{equation}
\left| F(a,b,c;\zeta) \right| \leq K, \label{Majo}
\end{equation}
and so we found an upper-bound for hypergeometric functions which is uniform in the parameter $\zeta$. Eventually, we use Proposition \ref{changement} to obtain an expression for the derivative of the hypergeometric function:
\[
\frac{\mathrm{d}}{\mathrm{d}\zeta}F(a,b,c;\zeta) = \frac{ab}{c}F(a+1,b+1,c+1;\zeta).
\]
We apply \eqref{Majo} with this new quantity, and using Stirling formula we deduce that
\[
\begin{split}
\left| \frac{ab}{c}F(a+1,b+1,c+1;\zeta) \right|
& \leq \frac{ab}{c} K\\
& \leq K'.
\end{split}
\]
We can easily generalise this result to derivatives of any order $k \in \N \cup \{0\}$:
\[
\left| \frac{\mathrm{d^k}}{\mathrm{d}\zeta^k}F(a,b,c;\zeta) \right| \leq C_k.
\qedhere
\]
\end{proof}
We can prove now that normalised hypergeometric functions (considered as function of the parameter $c$) are entire with a growth order at most one.
 \begin{lemme}
       Let $a$ and $b$ be fixed complex numbers, and $\eta >0$. Then, it exists a constant $\tilde{A}>0$ such that, for every $c \in \C$,
      \[
      \sup_{\zeta \in [\eta,1-\eta]} \left| \frac{1}{\Gamma(c)}F(a,b,c;\zeta) \right| \leq \tilde{A} e^{|c|\log|c|}.
      \]    
      \label{LemmeEntireFunction}
 \end{lemme}
 \begin{proof}
 Let $\rho>0$ such that for $c \in \C$ and for every integer $n$, we have $|c+n|>\rho$.
Using Proposition \ref{Prop1}, we have that it exists constants $A$ which are uniform with respect to the variable $\zeta$ such that
\[
\left| \frac{1}{\Gamma(c)} F(a,b,c;\zeta) \right| \leq \frac{A}{|\Gamma(c)|}.
\]
Due to the fact that $1/\Gamma$ is entire function with a growth order at most one (and maximal type) (see \cite[Chapter 1]{Levin}), it exists a constant $A'$ such that:
\[
 \frac{A}{|\Gamma(c)|} \leq A' e^{|c|\log|c|}.
\]
Then, the maximum principle on $B(-n,\rho)$ for $n$ an integer applied to the holomorphic (even entire) function $c \mapsto \frac{1}{\Gamma(c)}F(a,b,c;\zeta)$
told us that this quantity reaches its maximum on the boundary of this open set: it exists $c' \in \partial B(-n,\rho)$ such that
\[
\begin{split}
\max_{c \in B(-n,\rho)} \left|  \frac{1}{\Gamma(c)}F(a,b,c;\zeta) \right| & \leq \left|  \frac{1}{\Gamma(c')}F(a,b,c';\zeta) \right| \\
& \leq A' e^{|c'|\log|c'|}
\end{split}
\]
Since for $c \in B(-n,\rho)$, $|c-c'|<2\rho$, we have that $|c'| < |c| + 2\rho$ and so it exists a constant $\tilde{A}$ which is uniform with respect to the variable $\zeta$ such that 
\[
A' e^{|c'|\log|c'|} \leq \tilde{A}e^{|c|\log|c|}.
\]
So the quantity $c \mapsto \frac{1}{\Gamma(c)} F(a,b,c;\zeta)$ is also an entire function with a growth order at most one.
\end{proof}

\printbibliography
\end{document}

We give a new expression for the Jost function $w$, evaluating this Wronskian at $x=0$:
\begin{align}
    w(z) & = [f^-(0,z),f^+(0,z)] \notag \\
    & = f^-(0,z)(f^+)'(0,z)-(f^-)'(0,z)f^+(0,z) \notag \\
    & = f^-(0,z)(f^+)'(0,z)\left(1 - \frac{(f^-)'(0,z)f^+(0,z)}{f^-(0,z)(f^+)'(0,z)} \right) \notag \\
    & = f^-(0,z)(f^+)'(0,z)\left(1 - \frac{m_-(z)}{m_+(z)} \right),
    \label{WT}
\end{align}
with $m_+(z):= \frac{(f^+)'(0,z)}{f^+(0,z)}$ and $m_-(z):= \frac{(f^-)'(0,z)}{f^-(0,z)}$. We mention here that the two quantities $m_+$ and $m_-$ are called Weyl-Titchmarsh functions. We mention the work \cite{Kostenko} for information about the Weyl-Titchmarsh theory. Resonances are either zeros of $f^\pm(0,z)$ (or their derivatives) or zeros of the third factor in \eqref{WT}. We are going to study both of these two factors in different domains of the lower half-plane. 
\begin{defi}
    Let $S_1$ and $S_2$ be two domains of the lower half-plane defined by
    \[
    S_1:= \left\{ z \in \C^- ~|~ - |\Re(z)| \geq \delta \Im(z) \right\},
    \]
    and
    \[
    S_2:= \left\{ z \in \C^- ~|~ - |\Re(z)| < \delta \Im(z) \right\},
    \]
    for $\delta >0$ and $\C^-:= \{ z \in \C ~|~  \Im(z) <0 \}$.
    We also define a set $\mathcal{B}$ by
    \[
    \mathcal{B}:= \bigcup_{n=0}^{\pinf}B(-i(n+1),\eta),
    \]
    for $\eta>0$, and where $B(x,r)$ stands for the open ball of center $x$ and radius $r$.
    Finally, let $Z$ be the set of zeros of $f^-(0,z)(f^+)'(0,z)$ and $f^+(0,z)(f^-)'(0,z)$:
    \[
    Z:= Z_1 \cup Z_2.
    \]
    where 
    \[
    Z_1 = \left\{ z \in \C^- ~|~ f^-(0,z)(f^+)'(0,z)=0\right\}
    \]
    and
    \[
    Z_2 = \left\{ z \in \C^- ~|~ f^+(0,z)(f^-)'(0,z)=0\right\}.
    \]
\end{defi}
One can understand $Z_1$ as the domain where the quantity $m_-/m_+$ is not defined.
We start by the determination of asymptotics of the third factor in the expression \eqref{WT} of $w$.
\begin{lemme}    \label{Lemmef+}
    We have:
\[
f^+(0,z) \sim (-1)^p \frac{\partial_t^pK^+(0,2\beta)}{(iz)^{p+1}}e^{2iz\beta},~ \Im(z) \to -\infty, ~z \in \C^- \backslash \left( \mathcal{B} \cup Z \right).
\]
Let $C > 0$, we have
\[
f^+(0,z) \to 1 ,~ |\Re(z)| \to +\infty, ~ |\Im(z)| \leq  C,  ~ z \in \C^- \backslash \left( \mathcal{B} \cup Z \right).
\]
\end{lemme}

\begin{proof}
    The notations $o$, $O$ and $\sim$ holds for $z$ tending to $\infty$ and $z \in \C^- \backslash \left( \mathcal{B} \cup Z \right)$. 
    Let us start with the first statement. 
    Since 
    \[
    f^+(0,z) = f^+_0(0,z) + \int_0^{2\beta } K^+(0,t) f^+_0(t,z) \, \mathrm{d}t,
    \]
    we will show that $f^+_0(0,z) = o\left( \int_0^{2\beta } K^+(0,t) f^+_0(t,z) \, \mathrm{d}t \right)$.
    We have, after $(p+1)$ integrations by parts,
\begin{align}
& \int_0^{2\beta } K^+(0,t) f^+_0(t,z) \, \mathrm{d}t = \int_0^{2\beta } K^+(0,t) F\left(c-a,c-b,c; \frac{e^{2t}}{1+e^{2t}} \right) e^{itz} \, \mathrm{d}t\notag \\
& = \left[ \frac{e^{itz}}{iz} K^+(0,t) F\left(c-a,c-b,c; \frac{e^{2t}}{1+e^{2t}} \right) \right]^{2 \beta}_0 \notag \\
& - \int_0^{2\beta } \partial_t\left[ K^+(0,t) F\left(c-a,c-b,c; \frac{e^{2t}}{1+e^{2t}} \right)\right] \frac{e^{itz}}{iz} \, \mathrm{d}t \notag \\
& = \dots \notag \\
& = o(1) + (-1)^p \frac{\partial_t^pK^+(0,2\beta)}{(iz)^{p+1}}F\left(c-a,c-b,c;\frac{1}{1+e^{4\beta}}\right)e^{2iz\beta} \left[ 1+o(1) \right].
\label{expo}
\end{align}
Equation \eqref{2beta} tells us that $\partial_t^pK^+(0,2\beta^-) = -\frac{1}{4} q^{(p-1)}(\beta^-) \neq 0$, since we assume that the $(p-1)^{th}$ derivative of $q$ has a jump at $x = \beta$.
The second $o(1)$ corresponds to the remaining integral term after all the integrations by parts.
We are now going to use Lemma \ref{abc} to find the asymptotic of $f_0^+$ in this domain. We know that:
\[
\begin{split}
    f^+_0(0,z) & = 2^{iz}\frac{\sqrt{\pi}\Gamma\left(\frac{a}{2}+\frac{b}{2}+\frac{1}{2}\right)}{\Gamma\left(\frac{a}{2}+\frac{1}{2}\right) \Gamma\left(\frac{b}{2}+\frac{1}{2}\right)}.
\end{split}
\]
Recall that $a$, $b$ and $c$ are defined respectively in \eqref{a}, \eqref{b} and \eqref{c}.
We use now Euler's reflection formula
\[
\Gamma(1-z)\Gamma(z) = \frac{\pi}{\sin(\pi z)},~z \in \C \backslash \Z,
\]
to modify the expression of gamma function, so we will be able to use Stirling formula. To use Euler's reflection formula, we have to consider the domain with the zeros of $\frac{1}{\Gamma\left(\frac{a}{2}+\frac{1}{2}\right)}$ and $\frac{1}{\Gamma\left(\frac{b}{2}+\frac{1}{2}\right)}$ removed.
So, for such a $z$, we have
\[
\begin{split}
    & \frac{\sqrt{\pi}\Gamma\left(\frac{a}{2}+\frac{b}{2}+\frac{1}{2}\right)}{\Gamma\left(\frac{a}{2}+\frac{1}{2}\right)\Gamma\left(\frac{b}{2}+\frac{1}{2}\right)} = \frac{\sqrt{\pi}\Gamma\left(c\right)}{\Gamma\left(\frac{a}{2}+\frac{1}{2}\right)\Gamma\left(\frac{b}{2}+\frac{1}{2}\right)} \\
    & = \sqrt{\pi} \frac{\Gamma\left(\frac{1}{2}-\frac{a}{2}\right)\Gamma\left(\frac{1}{2}-\frac{b}{2}\right)\sin\left(\pi\left(\frac{a}{2}+\frac{1}{2}\right)\right)\sin\left(\pi\left(\frac{b}{2}+\frac{1}{2}\right)\right) \pi}{\Gamma(1-c) \pi^2 \sin(\pi c)} \\
    & = \frac{1}{\sqrt{\pi}}\frac{\sin\left(\pi\left(\frac{a}{2}+\frac{1}{2}\right)\right)\sin\left(\pi\left(\frac{b}{2}+\frac{1}{2}\right)\right)}{\sin(\pi c)} \frac{\Gamma\left(\frac{1}{2}-\frac{a}{2}\right)\Gamma\left(\frac{1}{2}-\frac{b}{2}\right)}{\Gamma(1-c)}.
\end{split}
\]
Then, we apply Stirling formula
\[
\Gamma(z) \sim (2\pi z)^{\frac{1}{2}} e^{-z} z^z, |\text{Arg}(z)| \leq \pi - \varepsilon, \varepsilon>0,
\]
on the last term:
\[
\begin{split}
    \frac{\Gamma\left(\frac{a}{2}+\frac{b}{2}+\frac{1}{2}\right)}{\Gamma\left(\frac{1}{2}-\frac{a}{2}\right)\Gamma\left(\frac{1}{2}-\frac{b}{2}\right)} & \sim \frac{\sqrt{2\pi\left(\frac{1}{2}-\frac{a}{2}\right)}\sqrt{2\pi\left(\frac{1}{2}-\frac{b}{2}\right)}}{\sqrt{2\pi\left(1-c\right)}} \\
    & \times \frac{e^{-\left(\frac{1}{2}-\frac{a}{2}\right)}e^{-\left(\frac{1}{2}-\frac{b}{2}\right)}}{e^{-(1-c)}} \times \frac{\left(\frac{1}{2}-\frac{a}{2}\right)^{\frac{1}{2}-\frac{a}{2}}\left(\frac{1}{2}-\frac{b}{2}\right)^{\frac{1}{2}-\frac{b}{2}}}{(1-c)^{1-c}}.
\end{split}
\]
The first two terms are $O(\sqrt{iz})$. Consider the last one, we have:
\[
\begin{split}
    \frac{\left(\frac{1}{2}-\frac{a}{2}\right)^{\frac{1}{2}-\frac{a}{2}}\left(\frac{1}{2}-\frac{b}{2}\right)^{\frac{1}{2}-\frac{b}{2}}}{(1-c)^{1-c}} & = \frac{\left(\frac{1}{2}-\frac{a}{2}\right)^{\frac{1}{2}}\left(\frac{1}{2}-\frac{b}{2}\right)^{\frac{1}{2}}}{(1-c)}\times\frac{\left(\frac{1}{2}-\frac{a}{2}\right)^{-\frac{a}{2}}\left(\frac{1}{2}-\frac{b}{2}\right)^{-\frac{b}{2}}}{(1-c)^{-c}}.
\end{split}
\]
The first term is also $O(\sqrt{iz})$. Then
\[
\begin{split}
    & \frac{\left(\frac{1}{2}-\frac{a}{2}\right)^{-\frac{a}{2}}\left(\frac{1}{2}-\frac{b}{2}\right)^{-\frac{b}{2}}}{(1-c)^{-c}} \\
    & = \frac{e^{(1-iz)\log(iz)}}{e^{\left(\frac{1}{4}-\frac{iz}{2}+\frac{\sqrt{\frac{1}{4}-\lambda}}{2}\right)\log\left(\frac{1}{4}+\frac{iz}{2}-\frac{\sqrt{\frac{1}{4}-\lambda}}{2}\right)}e^{\left(\frac{1}{4}-\frac{iz}{2}-\frac{\sqrt{\frac{1}{4}-\lambda}}{2}\right)\log\left(\frac{1}{4}+\frac{iz}{2}+\frac{\sqrt{\frac{1}{4}-\lambda}}{2}\right)}} \\
    & = \frac{e^{\log(iz)}}{e^{\left(\frac{1}{4}+\frac{\sqrt{\frac{1}{4}-\lambda}}{2}\right)\log\left(\frac{1}{4}+\frac{iz}{2}-\frac{\sqrt{\frac{1}{4}-\lambda}}{2}\right)}e^{\left(\frac{1}{4}-\frac{\sqrt{\frac{1}{4}-\lambda}}{2}\right)\log\left(\frac{1}{4}+\frac{iz}{2}+\frac{\sqrt{\frac{1}{4}-\lambda}}{2}\right)}} \\
    & \times \frac{e^{-iz\log(iz)}}{e^{\frac{-iz}{2}\log\left(\frac{1}{4}+\frac{iz}{2}-\frac{\sqrt{\frac{1}{4}-\lambda}}{2}\right)}e^{\frac{-iz}{2}\log\left(\frac{1}{4}+\frac{iz}{2}-\frac{\sqrt{\frac{1}{4}+\lambda}}{2}\right)}}.
\end{split}
\]
The first term of this product is still a $O(\sqrt{iz})$. For the second one:
\[
\begin{split}
     \frac{e^{-iz\log(iz)}}{e^{\frac{-iz}{2}\log\left(\frac{1}{4}+\frac{iz}{2}-\frac{\sqrt{\frac{1}{4}-\lambda}}{2}\right)}e^{\frac{-iz}{2}\log\left(\frac{1}{4}+\frac{iz}{2}-\frac{\sqrt{\frac{1}{4}+\lambda}}{2}\right)}} \sim K e^{{iz} \log(\frac{1}{2})} \to 0, z \to \infty.
\end{split}
\]
Furthermore, we have
\[
\frac{\sin\left(\pi\left(\frac{a}{2}+\frac{1}{2}\right)\right)\sin\left(\pi\left(\frac{b}{2}+\frac{1}{2}\right)\right)}{\sin(\pi c)}  = \frac{\cos(\pi\sqrt{\frac{1}{4} - \lambda}) + \sin(iz\pi)}{2\sin(iz\pi)}.
\]
So $f^+_0(0,z)$ has at most a polynomial growth at infinity in this domain,
\begin{equation}
    f^+_0(0,z) = O(z^m),~ m \in \N,
    \label{poly}
\end{equation}
and zeros when $z$ is such that $\frac{a+1}{2}$ or $\frac{b+1}{2}$ are equal to negative integers. And so, due to the exponential growth of \eqref{expo} and the polynomial growth of $\eqref{poly}$ we have
\[
\begin{split}
f^+(0,z) & \sim \int_0^{2\beta } K^+(0,t) f^+_0(t,z) \, \mathrm{d}t \\
& \sim (-1)^p \frac{\partial_t^pK^+(0,2\beta)}{(iz)^{p+1}}e^{2iz\beta}.
\end{split}
\]
We now consider the second statement. The integral term tends to $0$ as $|\Re(z)|$ tends to $+\infty$ using \eqref{expo} and the fact that the imaginary part is bounded. Furthermore, using Proposition \ref{Asymptot} and formula \eqref{HypGeoLim}, we have that $f^+_0(0,z) \to 1$ when $|\Re(z)|$ tends to $+\infty$. This gives the wanted results.
\end{proof}
We obtain similar asymptotics for $f^-$:
\begin{lemme}
    We have:
\[
f^-(0,z) \sim (-1)^r \frac{\partial_t^rK^-(0,2\alpha)}{(-iz)^{r+1}}e^{-2iz\alpha},~ \Im(z) \to -\infty, ~z \in \C^- \backslash \left( \mathcal{B} \cup Z \right).
\]
Let $C > 0$, we have:
\[
f^-(0,z) \to 1 ,~ |\Re(z)| \to +\infty, ~ |\Im(z)| \leq  C,  ~ z \in \C^- \backslash \left( \mathcal{B} \cup Z \right).
\]
    \label{Lemmef-}
\end{lemme}
The asymptotics introduced in Lemmas \ref{Lemmef+} (\ref{Lemmef-}, respectively) holds for $z$ which is not a zero of $f^+(0,z)$ or $(f^+)'(0,z)$ ($f^-(0,z)$ or $(f^-)'(0,z)$, respectively). We need to know where those zeros are located in order to know where these asymptotics are valid.
\begin{lemme}
    The set $Z$ is bounded.
    \label{Z}
\end{lemme}
\begin{proof}
    The set $Z$ is composed of the zeros of four functions: $f^+(0,z)$, $(f^+)'(0,z)$, $f^-(0,z)$ and $(f^-)'(0,z)$. We are going to prove that the set of the zeros of each function is bounded. We detail only the case of $f^+(0,z)$, due to the fact that they are similar. Assume that $(z_n)_n \subset \C^-$ is a sequence of zeros of $f^+(0,z)$ such that $|z_n| \to \pinf$. Firstly, we assume that $\Im(z_n) \to \minf$. Then, we have 
    \[
    f^+_0(0,z_n) = - \int_0^{2\beta} K^+(0,t)f_0^+(t,z_n) \, \mathrm{d}t.
    \]
    If $f^+_0(0,z_n) \neq 0$ for $n$ large enough, then we have
    \[
    \frac{- \int_0^{2\beta} K^+(0,t)f_0^+(t,z_n) \, \mathrm{d}t.}{f^+_0(0,z_n)} = 1.
    \]
    Taking the limit when $n$ tends to $\pinf$ shows a contradiction, due to the fact that the numerator has an exponential growth in \eqref{expo}, and the denominator a polynomial growth in \eqref{poly}.
    Now, if $f^+_0(0,z_n) = 0$ for every $n$ large enough, then necessarily the quantity in \eqref{expo} is equal to zero: it is impossible, considering the fact that 
    \[
    (-1)^p \frac{\partial_t^pK^+(0,2\beta)}{(iz_n)^{p+1}}F\left(c-a,c-b,1-iz_n;\frac{1}{1+e^{4\beta}}\right)e^{2iz_n\beta}
    \]
    tends to $\infty$ when $n$ tends to $\pinf$.
    Now, assume that there exists a constant $C$ such that $|\Im(z_n)| \leq C$ and $|\Re(z_n)| \to \pinf$.
    Then, we proved in Lemma \ref{Lemmef+} that the integral term tends to 0. Since $f_0^+(0,z_n)$ tends to $1$ (and so has no zero) in this domain of the complex plane, this is a contradiction with the existence of such a sequence of $(z_n)_n$.
    So, there is no sequence of zero $(z_n)_n$ such that $|z_n| \to \pinf$, \textit{i.e.} the set of the zeros of $f^+(0,z)$ is bounded. Similar reasoning on $(f^+)'(0,z)$, $f^-(0,z)$ and $(f^-)'(0,z)$ shows that $Z$ is bounded.
\end{proof}
Thank to Lemma \ref{Z}, we can remove the set $Z$ from the previous asymptotics of Lemmas \ref{Lemmef+} and \ref{Lemmef-}, since they are established for $z$ tending to $\infty$ in the lower half-plane and $Z$ is bounded.

Furthermore, using these two Lemmas  \eqref{Lemmef+} and \eqref{Lemmef-}, we can determine asymptotics for $\frac{m_-}{m_+}$:
\begin{lemme}
     We assume that $q$ is compactly supported, with a support satisfying $ 0 \in (\alpha,\beta)$. We have:
    \[
    \frac{m_-(z)}{m_+(z)} \to -1,~ |z| \to \infty, ~ z \in \C^- \backslash \mathcal{B} .
    \]
\end{lemme}
\begin{proof}
The notations $o$ and $\sim$ holds for $z$ tending to $\infty$ in $\C^- \backslash \mathcal{B}$. We first assume that $\Im(z) \to -\infty$. 
Using result of Lemma \ref{Lemmef+}, we can prove that
\[
f^+(0,z) \sim (-1)^p \frac{\partial_t^pK^+(0,2\beta)}{(iz)^{p+1}}e^{2iz\beta}.
\]
Similarly to what we did in Lemma \ref{Lemmef+}, we have:
\[
(f^+)'(0,z) \sim (-1)^{p-1} \frac{\partial_t^{p-1}\partial_xK^+(0,2\beta)}{(iz)^{p}}e^{2iz\beta}.
\]
So the function $m_+$ is equivalent to
\[
\begin{split}
m_+(z) & \sim \frac{(-1)^{p-1} \frac{\partial_t^{p-1}\partial_xK^+(0,2\beta)}{(iz)^{p}}e^{2iz\beta}}{(-1)^p \frac{\partial_t^pK^+(0,2\beta)}{(iz)^{p+1}}e^{2iz\beta}} \\
& \sim - \frac{\partial_t^{p-1}\partial_xK^+(0,2\beta)\times iz}{\partial_t^pK^+(0,2\beta)} \\
& \sim -iz,~\Im(z) \to \infty,~ z \in \C^- \backslash \mathcal{B}.
\end{split}
\]
since $\partial_t^{p-1}\partial_xK^+(0,2\beta) = \partial_t^pK^+(0,2\beta)$ (see the proof of \eqref{2beta}). 
Using the same proof, and Lemma \ref{Lemmef-}, we have the result for $m_-$:
\[
\begin{split}
m_-(z) & \sim \frac{(-1)^{r-1} \frac{\partial_t^{r-1}\partial_xK^+(0,2\beta)}{(-iz)^{r}}e^{-2iz\alpha}}{(-1)^{r} \frac{\partial_t^r K^+(0,2\beta)}{(-iz)^{r+1}}e^{-2iz\alpha}} \\
& \sim \frac{\partial_t^{r-1}\partial_xK^-(0,2\alpha)\times iz}{\partial_t^rK^-(0,2\alpha)} \\
& \sim iz,~ \Im(z) \to \infty,~ z \in \C^- \backslash \mathcal{B}.
\end{split}
\]
Then, combining these two equivalents:
\[
 \frac{m_-(z)}{m_+(z)} \to -1, ~\Im(z) \to \infty,~ z \in \C^- \backslash \mathcal{B},
\]
and so the result is proved. Now, if $\Re(z) \to \infty$, we have
\[
f^+(0,z) \sim 1
\]
Similarly to what we did in Lemma \ref{Lemmef+}, we have:
\[
(f^+)'(0,z) \sim iz.
\]
So the quantity $m_+$ is equivalent to
\[
m_+(z) \sim iz,~ \Re(z) \to \infty,~ z \in \C^- \backslash \mathcal{B}.
\]
Similarly, 
\[
m_-(z) \sim -iz,~ \Re(z) \to \infty,~ z \in \C^- \backslash \mathcal{B},
\]
and so 
\[
\frac{m_-(z)}{m_+(z)} \to -1, ~\Re(z) \to \infty,~ z \in \C^- \backslash \mathcal{B},
\]
and so the result is proved when $|z| \to \infty$.
\end{proof}
We mention here that in the case of an even perturbation $q$, we immediately have $\frac{m_-(z)}{m_+(z)}=-1$, and so $w(z) = 2f^+(0,z)(f^+)'(0,z)$.
Now, we can asymptotically locate resonances, firstly in $S_1$:
\begin{prop}
Let $V$ be a potential defined by $V(x)  = \frac{\lambda}{\cosh^2(x)}+q(x)$, for all $x \in \R$ with $q$ real-valued, integrable and its support included in $[\alpha,\beta]$. Furthermore, we assume that $0 \in (\alpha,\beta)$ and $q$ satisfies Assumption \ref{Asum}.
    For $z \in S_1\backslash \mathcal{B}$, asymptotically, the function $w$ has no zero in this domain.
    \label{Nozero}
\end{prop}
\begin{proof}
We consider the domain $S_1 \backslash \mathcal{B}$ instead of $S_1$ to be away from the poles of hypergeometric functions.
The term $1-\frac{m_-(z)}{m_+(z)}$ tends to 2, so it will not vanish for sufficiently large values of $z$.
Let now consider $z_0$ a resonance of $w_0$, outside a sufficiently large disc centred at 0. Let $z$ be a complex number on the boundary of $\left( S_1\backslash \mathcal{B} \right) \cap \overline{D(z_0,\varepsilon)}$ with $\varepsilon>0$, denoted $\partial\left( \left( S_1\backslash \mathcal{B} \right) \cap \overline{D(z_0,\varepsilon)}\right)$. We will now recall two results: the term $\int_{2\alpha}^{0} K^-(0,t) f^-_0(t,z) \, \mathrm{d}t$ has an exponential growth with \eqref{expo}, and the term $f^-(0,z)$ has a polynomial growth with \eqref{poly} (these two results \eqref{expo} and \eqref{poly} are stated for $f^+$, but they still hold for $f^-$, up to obvious modifications). Then, for $z$ large enough (and so for $z_0$ large enough) we have the following estimate:
\[
    \begin{split}
    \left| f^-(0,z) - \int_{2 \alpha}^{0} K^-(0,t) f^-_0(t,z) \, \mathrm{d}t \right| & = \left| f^-_0(0,z)\right| \\
    & < \left| \int_{2 \alpha}^{0} K^-(0,t) f^-_0(t,z) \, \mathrm{d}t \right|.
    \end{split}
    \]
    We can apply Rouché's theorem on the two meromorphic functions $f^-(0,z)$ and $\int_{2 \alpha}^{0} K^-(0,t) f^-_0(t,z) \, \mathrm{d}t$, seen as functions of the complex parameter $z$. The dominant function has no zero asymptotically on $\overline{D(z_0,\varepsilon)} \cap \left( S_1 \backslash \mathcal{B} \right)$ for $z_0$ a zero of $w_0$,
then $f^-(0,z)$ has no zero asymptotically on this set. The same reasoning for $(f^+)'(0,z)$ shows that this function has no zero in this domain neither. Using \eqref{WT}, we conclude that $w$ has no zero asymptotically on the set $S_1 \backslash \mathcal{B}$.
\end{proof}

We are now interested in the location of resonances in the domain $S_2$:

\begin{prop}
    Let $V$ be a potential defined by $V(x)  = \frac{\lambda}{\cosh^2(x)}+q(x)$, for all $x \in \R$ with $q$ real-valued, integrable and its support included in $[\alpha,\beta]$. Furthermore, we assume that $0 \in (\alpha,\beta)$ and $q$ satisfies Assumption \ref{Asum}.
    In the domain $S_2$, the resonances $\beta_j$ for $j \in \N$ satisfy the asymptotics:
    \[
    \begin{split}
    \beta_{\pm j} & = \pm \frac{\pi}{2(\beta - \alpha)} \left( 2j + \frac{p+r}{2} \pm \frac{\left( 1-sign(A) \right)}{2} \right) \\
    & - \frac{i(p+r)}{2(\beta - \alpha)} \log \left( \frac{j \pi}{\beta - \alpha}\right) + \frac{i}{2(\beta - \alpha)} \log|A| +o(1),
    \end{split}
    \]
    for $A = {(-1)^p \partial_t^{r-1} \partial_x K^-(0,2\alpha) \times \partial_t^{p}K^+(0,2\beta)}$ a real constant. The notation $o(1)$ holds for $j$ tending to $\infty$.
    \label{LogBranches}
\end{prop}

\begin{proof}
    The notations $o$ and $\sim$ holds for $z$ tending to $\infty$ in $S_2$ 
     Since for any $z \in S_2$, $\zeta=\frac{1}{1+e^{4\beta}}$, $\zeta = \frac{1}{2}$ or $\zeta = \frac{e^{4\alpha}}{1+e^{4\alpha}}$ satisfy assumptions $2$ or $3$ of Proposition \ref{Asymptot}, we can use the asymptotic \eqref{HypGeoLim} (see Appendix). 
     Now, after $p+1$ integrations by parts of kernel $K^+$ (as we did in the proof of Proposition \ref{Nozero}),
     we have an expression for $f^+(0,z)$for $z \in S_2$ , which is the one presented in the proof of Lemma \ref{Lemmef+}:
    \[
    \begin{split}
    & f^+(0,z) = f^+_0(0,z) + o(1) \\
    & + (-1)^p \frac{\partial_t^pK^+(0,2\beta)}{(iz)^{p+1}}F\left(c-a,c-b,c;\frac{1}{1+e^{4\beta}}\right)e^{2iz\beta} \left[ 1+o(1) \right].
    \end{split}
    \]
    Similarly to what we have done with $f^+$, we can do the same thing with $(f^+)'(0,z)$ to obtain the following expression:
    \[
    \begin{split}
    & (f^+)'(0,z) = izf^+_0(0,z) + o(1) \\
    & + (-1)^{p-1} \frac{\partial_t^{p-1}\partial_xK^+(0,2\beta)}{(iz)^{p}}F\left(c-a,c-b,c;\frac{1}{1+e^{4\beta}}\right)e^{2iz\beta} \left[ 1+o(1) \right].
    \end{split}
    \]
    We give here the two expressions of $f^-(0,z)$ and $(f^-)'(0,z)$:
    \[
    \begin{split}
   & f^-(0,z) = f^-_0(0,z) + o(1) \\
    & + (-1)^{r+1} \frac{\partial_t^pK^+(0,2\alpha)}{(-iz)^{r+1}}F\left(c-a,c-b,c;\frac{e^{4\alpha}}{1+e^{4\alpha}}\right)e^{-2iz\alpha} \left[ 1+o(1) \right],
    \end{split}
    \]
    and
    \[
    \begin{split}
    & (f^-)'(0,z) = -izf^-_0(0,z) + o(1) \\
    & + (-1)^r \frac{\partial_t^{r-1}\partial_xK^+(0,2\alpha)}{(-iz)^{r}}F\left(c-a,c-b,c;\frac{e^{4\alpha}}{1+e^{4\alpha}}\right)e^{-2iz\alpha} \left[ 1+o(1) \right].
    \end{split}
    \]
    Now, we can use two of these four different writings to obtain expressions of $f^-(0,z)(f^+)'(0,z)$. For $z \in S_2$, we have:
   \[
    \begin{split}
        & f^-(0,z)(f^+)'(0,z) = 2iz F\left(c-a,c-b,c;\frac{1}{2}\right)^2 F\left(c-a,c-b,c;\frac{1}{1+e^{4\beta}}\right)\\
        & \times F\left(c-a,c-b,c;\frac{e^{4\alpha}}{1+e^{4\alpha}}\right) \left[ 1 + \underset{|z| \to \pinf}{o(1)} - \frac{Ae^{2iz(\beta - \alpha)}}{(iz)^{p+r}} \left( 1 +  \underset{|z| \to \pinf}{o(1)} \right)\right]. \\
    \end{split}
    \]
    The same reasoning applied to $f^+(0,z)(f^-)'(0,z)$ gives
        \[
    \begin{split}
        & f^+(0,z)(f^-)'(0,z) = -2iz F\left(c-a,c-b,c;\frac{1}{2}\right)^2 F\left(c-a,c-b,c;\frac{1}{1+e^{4\beta}}\right)\\
        & \times F\left(c-a,c-b,c;\frac{e^{4\alpha}}{1+e^{4\alpha}}\right) \left[ 1 + \underset{|z| \to \pinf}{o(1)} - \frac{Ae^{2iz(\beta - \alpha)}}{(iz)^{p+r}} \left( 1 +  \underset{|z| \to \pinf}{o(1)} \right)\right] \\
    \end{split}
    \]
    where
    \[
    \begin{split}
    A & = (-1)^p \partial_t^{r-1} \partial_x K^-(0,2\alpha) \times \partial_t^{p}K^+(0,2\beta) \\
    & = (-1)^p \partial_t^{r}K^-(0,2\alpha) \times \partial_t^{p-1}\partial_x K^+(0,2\beta),
    \end{split}
    \]
    since $\partial_t^{p-1}\partial_xK^+(0,2\beta) =\partial_t^pK^+(0,2\beta)$ and $\partial_t^{r-1}\partial_xK^-(0,2\alpha) =\partial_t^rK^+(0,2\alpha)$.
    From these two expressions, we can conclude that the two quantities $f^+(0,z)(f^-)'(0,z)$ and $f^-(0,z)(f^+)'(0,z)$ have the same zeros, \textit{i.e.} $Z_1 = Z_2$.
    Furthermore, we have an expression of $w$ on the complex domain $\C^-$, and so on $S_2$:
    \[
    \begin{split}
    w(z)& = 4iz F\left(c-a,c-b,c;\frac{1}{2}\right)^2 F\left(c-a,c-b,c;\frac{1}{1+e^{4\beta}}\right)\\
        & \times F\left(c-a,c-b,c;\frac{e^{4\alpha}}{1+e^{4\alpha}}\right) \left[ 1 + \underset{|z| \to \pinf}{o(1)} - \frac{Ae^{2iz(\beta - \alpha)}}{(iz)^{p+r}} \left( 1 +  \underset{|z| \to \pinf}{o(1)} \right)\right]. \\
    \end{split}
    \]
    Then, using Hardy \cite[422-423]{Hardy} and Cartwright \cite{Cartwright1,Cartwright2} method presented in \cite{BBD,ZWORSKI1987277},
    the zeros of the fourth factor in the previous expression of $w$
    \[
    1 + \underset{|z| \to \pinf}{o(1)} - \frac{Ae^{2iz(\beta - \alpha)}}{(iz)^{p+r}} \left( 1 +  \underset{|z| \to \pinf}{o(1)} \right)
    \]
    are located in the lower half-plane and have the following asymptotics:
    \[
    \begin{split}
    \beta_{\pm j} & = \pm \frac{\pi}{2(\beta - \alpha)} \left( 2j + \frac{p+r}{2} \pm \frac{\left( 1-sign(A) \right)}{2} \right) \\
    & - \frac{i(p+r)}{2(\beta - \alpha)} \log \left( \frac{j \pi}{\beta - \alpha}\right) + \frac{i}{2(\beta - \alpha)} \log|A| +o(1).
    \end{split}
    \]
    We are now going to use Rouché's theorem. Indeed, we know the location of zeros of $f^-(0,z)(f^+)'(0,z)$. But we also know that for such points $\beta_{\pm j}$, the function $w$ cannot be written $w(z) = f^-(0,z)(f^+)'(0,z)\left( 1 - \frac{m_-(z)}{m_+(z)} \right)$. So we are going to prove that, the zeros of $w$ are arbitrarily close to the zeros of $f^-(0,z)(f^+)'(0,z)$, the $\beta_{\pm j}$, for sufficiently large $j$. 
    Considering the term $\left( 1 - \frac{m_-(z)}{m_+(z)} \right)$: in the proof of Lemma \ref{Lemmef+}, we have shown that
    \[
    \begin{split}
    f^+(0,z) & = f^+_0(0,z) + o(1) \\
    & + (-1)^p \frac{\partial_t^pK^+(0,2\beta)}{(iz)^{p+1}}F\left(c-a,c-b,c;\frac{1}{1+e^{4\beta}}\right)e^{2iz\beta} \left[ 1+o(1) \right].
    \end{split}
    \]
    Due to the asymptotic presented in Proposition \ref{Asymptot}, we have that $f_0^+(0,z) \to 1$ in when $|z| \to \pinf$ in $S_2$.
    Similarly to what we have done with $f^+$, we can do the same thing with $(f^+)'(0,z)$ to obtain the following expression:
    \[
    \begin{split}
    (f^+)'(0,z) & = izf^+_0(0,z) + o(1) \\
    & + (-1)^p \frac{\partial_t^{p-1}\partial_xK^+(0,2\beta)}{(iz)^{p}}F\left(c-a,c-b,c;\frac{1}{1+e^{4\beta}}\right)e^{2iz\beta} \left[ 1+o(1) \right].
    \end{split}
    \]
    We give here the two expressions of $f^-(0,z)$ and $(f^-)'(0,z)$:
    \[
    \begin{split}
    f^-(0,z) & = f^-_0(0,z) + o(1) \\
    & + (-1)^{r+1} \frac{\partial_t^pK^+(0,2\alpha)}{(-iz)^{r+1}}F\left(c-a,c-b,c;\frac{e^{4\alpha}}{1+e^{4\alpha}}\right)e^{-2iz\alpha} \left[ 1+o(1) \right].
    \end{split}
    \]
    and
    \[
    \begin{split}
    (f^-)'(0,z) & = -izf^-_0(0,z) + o(1) \\
    & + (-1)^r \frac{\partial_t^{r-1}\partial_xK^+(0,2\alpha)}{(-iz)^{r}}F\left(c-a,c-b,c;\frac{e^{4\alpha}}{1+e^{4\alpha}}\right)e^{-2iz\alpha} \left[ 1+o(1) \right].
    \end{split}
    \]
    It comes from these four writings:
    \[
    \begin{split}
    & m_+(z) = \frac{(f^+)'(0,z)}{f^+(0,z)}\\
    & \sim \frac{izf^+_0(0,z)+ (-1)^p \frac{\partial_t^{p-1}\partial_xK^+(0,2\beta)}{(iz)^{p}}F\left(c-a,c-b,c;\frac{1}{1+e^{4\beta}}\right)e^{2iz\beta}}{f^+_0(0,z) + (-1)^p \frac{\partial_t^pK^+(0,2\beta)}{(iz)^{p+1}}F\left(c-a,c-b,c;\frac{1}{1+e^{4\beta}}\right)e^{2iz\beta}} \\
    & = iz \frac{f^+_0(0,z) + (-1)^{p-1} \frac{\partial_t^{p-1}\partial_xK^+(0,2\beta)}{(iz)^{p+1}}F\left(c-a,c-b,c;\frac{1}{1+e^{4\beta}}\right)e^{2iz\beta}}{f^+_0(0,z) + (-1)^p \frac{\partial_t^pK^+(0,2\beta)}{(iz)^{p+1}}F\left(c-a,c-b,c;\frac{1}{1+e^{4\beta}}\right)e^{2iz\beta}} \\
    & = iz \frac{f^+_0(0,z) - (-1)^{p} \frac{\partial_t^{p-1}\partial_xK^+(0,2\beta)}{(iz)^{p+1}}F\left(c-a,c-b,c;\frac{1}{1+e^{4\beta}}\right)e^{2iz\beta}}{f^+_0(0,z) + (-1)^p \frac{\partial_t^pK^+(0,2\beta)}{(iz)^{p+1}}F\left(c-a,c-b,c;\frac{1}{1+e^{4\beta}}\right)e^{2iz\beta}} \\
    & \sim iz \frac{1 - (-1)^{p} \frac{\partial_t^{p-1}\partial_xK^+(0,2\beta)}{(iz)^{p+1}}e^{2iz\beta}}{1 + (-1)^p \frac{\partial_t^pK^+(0,2\beta)}{(iz)^{p+1}}e^{2iz\beta}}. \\
    \end{split}
    \]
    The same calculations give an equivalent for $m_-$:
    \[
    \begin{split}
        & f^-(0,z)(f^+)'(0,z) = 2iz F\left(c-a,c-b,c;\frac{1}{2}\right)^2 F\left(c-a,c-b,c;\frac{1}{1+e^{4\beta}}\right)\\
        & \times F\left(c-a,c-b,c;\frac{e^{4\alpha}}{1+e^{4\alpha}}\right) \left[ 1 + \underset{|z| \to \pinf}{o(1)} - \frac{Ae^{2iz(\beta - \alpha)}}{(iz)^{p+r}} \left( 1 +  \underset{|z| \to \pinf}{o(1)} \right)\right] \\
    \end{split}
    \]
    and 
        \[
    \begin{split}
        & f^+(0,z)(f^-)'(0,z) = -2iz F\left(c-a,c-b,c;\frac{1}{2}\right)^2 F\left(c-a,c-b,c;\frac{1}{1+e^{4\beta}}\right)\\
        & \times F\left(c-a,c-b,c;\frac{e^{4\alpha}}{1+e^{4\alpha}}\right) \left[ 1 + \underset{|z| \to \pinf}{o(1)} - \frac{Ae^{2iz(\beta - \alpha)}}{(iz)^{p+r}} \left( 1 +  \underset{|z| \to \pinf}{o(1)} \right)\right] \\
    \end{split}
    \]
    From these two expressions, we can conclude that $f^+(0,z)(f^-)'(0,z)$ and $f^-(0,z)(f^+)'(0,z)$ have the same zeros, \textit{i.e.} $Z_1 = Z_2$.
    Furthermore, we have an expression of $w$ on the complex domain $\C^-$:
    \[
    \begin{split}
    w(z)& = 4iz F\left(c-a,c-b,c;\frac{1}{2}\right)^2 F\left(c-a,c-b,c;\frac{1}{1+e^{4\beta}}\right)\\
        & \times F\left(c-a,c-b,c;\frac{e^{4\alpha}}{1+e^{4\alpha}}\right) \left[ 1 + \underset{|z| \to \pinf}{o(1)} - \frac{Ae^{2iz(\beta - \alpha)}}{(iz)^{p+r}} \left( 1 +  \underset{|z| \to \pinf}{o(1)} \right)\right] \\
    \end{split}
    \]
    Since $\partial_t^{p-1}\partial_xK^+(0,2\beta) =\partial_t^pK^+(0,2\beta)$, we have

    Let $r$ be a positive real number and $K:= \overline{B(\beta_{\pm j},r)}$, $j \in \N$. If $j$ is sufficiently large, then for $z \in \partial K$:
    \[
    \begin{split}
        & \left| w(z) - 2f^-(0,z)(f^+)'(0,z) \right| \\
        & = |f^-(0,z)(f^+)'(0,z)| \left| 1 - \frac{m_-(z)}{m_+(z)} - 2 \right| < |2f^-(0,z)(f^+)'(0,z)|.
    \end{split}
    \]
    We have used here the fact that $2$ is the limit of $ 1 - \frac{m_-(z)}{m_+(z)}$ when $|z| \to \infty$ and $z \in S_2$.
\end{proof}

\begin{figure}[H]
\centering
\begin{tikzpicture}[scale=0.85]
    \filldraw[draw=black,fill=green!20]
    (-1.5,-5) -- (0,3) -- (1.5,-5) -- cycle;
    \filldraw[draw=black,fill=blue!20]
    (-5,3) -- (0,3) -- (-1.5,-5) -- (-5,-5) -- cycle;
    \filldraw[draw=black,fill=blue!20]
    (5,3) -- (0,3) -- (1.5,-5) -- (5,-5) -- cycle;
    \draw[step=1cm, gray, very thin] (-5, -5) grid (5, 5);
    \draw[very thick, ->] (-4.5, 3) -- (4.5, 3) node[below]{$\Re(z)$};
    \draw[very thick, ->] (0, -4.5) -- (0, 4.5) node[left]
    {$\Im(z)$} ;
    \filldraw[red] (0.866,-1.5) circle (2pt);
    \filldraw[red] (-0.866,-1.5) circle (2pt);
    \filldraw[red] (0.866,-2.5) circle (2pt);
    \filldraw[red] (-0.866,-2.5) circle (2pt);
    \filldraw[red] (0.866,-3.5) circle (2pt);
    \filldraw[red] (-0.866,-3.5) circle (2pt);
    \filldraw[red] (0.866,-4.5) circle (2pt);
    \filldraw[red] (-0.866,-4.5) circle (2pt);
    \filldraw[black] (3.5,0) circle (0pt) node[above]{$S_2$};
    \filldraw[black] (-3.5,0) circle (0pt) node[above]{$S_2$};
    \filldraw[black] (0,-4.5) circle (0pt) node[right]{$S_1$};
    \filldraw[black] (1.2,-1.98060) circle (2pt);
    \filldraw[black] (-1.2,-1.98060) circle (2pt);
    \filldraw[black] (2.2,-2.19612) circle (2pt);
    \filldraw[black] (-2.2,-2.19612) circle (2pt);
    \filldraw[black] (3.2,-2.49715) circle (2pt);
    \filldraw[black] (-3.2,-2.49715) circle (2pt);
    \filldraw[black] (4.2,-2.67324) circle (2pt);
    \filldraw[black] (-4.2,-2.67324) circle (2pt);
    \draw[black] (0,0) circle (0.4) node[right]{$\mathcal{B}$};
    \draw[black] (0,-1) circle (0.4) node[right]{$\mathcal{B}$};
    \draw[black] (0,-2) circle (0.4) node[right]{$\mathcal{B}$};
    \draw[black] (0,-3) circle (0.4) node[right]{$\mathcal{B}$};
    \draw[black] (0,-4) circle (0.4) node[right]{$\mathcal{B}$};
\end{tikzpicture}
\captionsetup{justification=centering}
\caption{Sketch of location of resonances:
\\ in the Pöschl-Teller case in red, and with a perturbation in black (asymptotically) such that $0 \in (\alpha,\beta)$ and $\lambda = 1$.}
\label{Dessin2}
\end{figure}
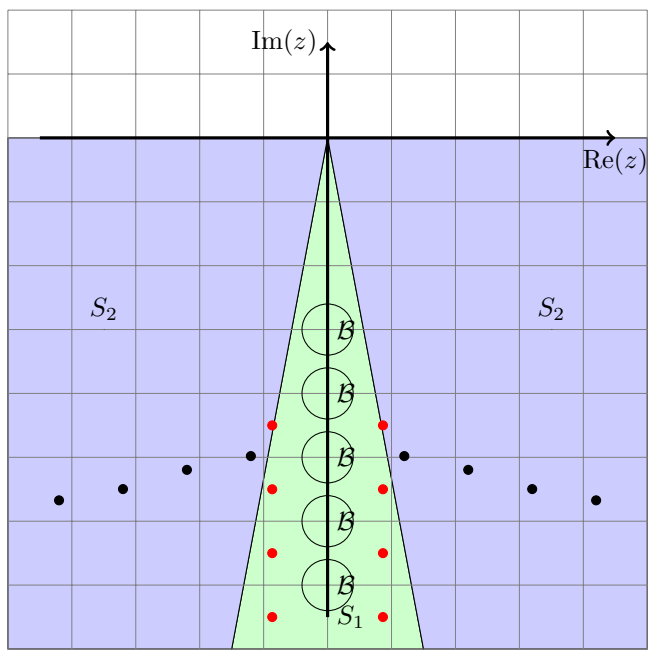

Finally, assume that $0 \notin (\alpha,\beta)$. Then we suppose that $\supp{q} \subset \R_+$. We have that $f^-(0,z) = f^-_0(0,z)$, because the integral term vanishes.
So
\[
m_-(z) \sim {-iz} \tan(\frac{\pi}{2}a)\tan(\frac{\pi}{2}b),~ z \to \infty ,~ \Im(z) < 0.
\]
Then, we have to adapt the results of Propositions \ref{Nozero} and \ref{LogBranches} to this new asymptotic. We reformulate here Proposition \ref{Nozero}:
\begin{prop}
Let $V$ be a potential defined by $V(x)  = \frac{\lambda}{\cosh^2(x)}+q(x)$, for all $x \in \R$ with $q$ real-valued, integrable and its support included in $[\alpha,\beta]$. Furthermore, we assume that $0 \notin (\alpha,\beta)$ and $q$ satisfies Assumption \ref{Asum}.
    For $z \in S_1\backslash \mathcal{B}$, the resonances $\alpha_j$ for $j \in \N$ satisfy the asymptotics:
    \[
    \alpha_{\pm j } \sim -i \left(2 j - 1 \pm \sqrt{\frac{1}{4}-\lambda} - \frac{1}{2} \right).
    \]
    The notation $\sim$ holds for $j$ tending to $\infty$.
\end{prop}
\begin{proof}
    The proof of this proposition is quite similar to the one of Proposition \ref{LogBranches}. Since $\alpha \geq 0$, $f^-(0,z) = f^-_0(0,z)$. The zeros of $f^-(0,z)$ are some of the previous resonances of $w_0$, the $\alpha_{\pm j} := -i \left(2 j - 1 \pm \sqrt{\frac{1}{4}-\lambda} - \frac{1}{2} \right)$ for $j \in \N$. However, the quantity $\frac{\tan(\frac{\pi}{2}a)}{\Gamma(\frac{a+1}{2})}$
    tends to $a$ when $z$ tends to an $\alpha_j$. This means that the $\alpha_j$ are not zeros of the function $\frac{m_-(z)}{m_+(z)}f^-(0,z)(f^+)'(0,z)$. The same reasoning holds concerning the $\alpha_{-j}$ using $b$ instead of $a$ in the previous limit. Using the fact that $Z$ is a bounded set, and so the zeros of $(f^+)'(0,z)$ are bounded, then \eqref{WT} implies that $w$ has no zero asymptotically in $S_1 \backslash \mathcal{B}$.
    Now, let consider $\alpha_{\pm j}$ for $j \in \N$, outside a sufficiently large disc centred at 0. Let $z$ be a complex number on the boundary of $\left( S_1\backslash \mathcal{B} \right) \cap \overline{D(\alpha_{\pm j},\varepsilon)}$ with $\varepsilon>0$, denoted $\partial\left( \left( S_1\backslash \mathcal{B} \right) \cap \overline{D(\alpha_{\pm j},\varepsilon)}\right)$. Then, we apply Rouché's theorem for such a $z$:
    \[
    \begin{split}
        & \left| \frac{m_-(z)}{m_+(z)}f^-(0,z) - Mf^-(0,z) \right| \\
        & = |f^-(0,z)| \left| \frac{m_-(z)}{m_+(z)}- M \right|< \left| M f^-(0,z)(f^+)'(0,z)\right|,
    \end{split}
    \]
    \Va{Préciser pourquoi cette inégalité est vraie}
    since $\frac{m_-(z)}{m_+(z)} \to M$ when $z \to \infty$ in $S_1 \backslash \mathcal{B}$.
    So, using Rouché's theorem, holomorphic functions $\frac{m_-(z)}{m_+(z)}f^-(0,z)$ and $Mf^-(0,z)$ have their zeros as close as we want in the compact $\left( S_1\backslash \mathcal{B} \right) \cap \overline{D(\alpha_{\pm j},\varepsilon)}$. Since $\frac{m_-(z)}{m_+(z)}f^-(0,z)(f^+)'(0,z)$ has no zero in this compact set, asymptotically, this is also the case for $w$.
    
\end{proof}
In $S_2$, we observe the same global behaviour for the resonances as in Proposition the case of a potential $q$ such that $0 \notin (\alpha,\beta)$: they are still located on logarithmic branches. 

In the case $\supp{q} \subset \R_-$, we can use the symmetries of the problem under the change of variable $\tilde{x} = -x$. A potential $q$ with its support in $\R_-$ has now its support in $\R_+$, where we know how to conclude.

We observe that the resonances initially created by a Pöschl-Teller potential and located on parallel lines with respect to the imaginary axis are asymptotically on two logarithmic branches when we consider a compactly supported perturbation of this potential. See Figure \ref{Dessin2} for a representation of these results.

Since $\Re(c) >0$ and $\Re(b) >0$, we can apply Stirling formula to prove that it exists a constant $C_2$ such that
\[
\left|\frac{\Gamma(c)}{\Gamma(b)}\right| \leq  |z|^{1/2+\mu}.
\]
Finally, we obtain that 
\begin{equation}
\left|F(c-a,c-b,c;\zeta)\right| \leq 1+o(|z|^{-1}).
\label{EqIm12}
\end{equation}
We use Proposition \ref{changement} to obtain an expression for the derivative of the hypergeometric function:
\[
\frac{\mathrm{d}}{\mathrm{d}\zeta}F(a,b,c;\zeta) = \frac{(c-a)(c-b)}{c}F(c-a+1,c-b+1,c+1;\zeta).
\]
We apply \eqref{Majo} with this new quantity, and using Stirling formula we deduce that
\[
\begin{split}
& \left| \frac{(c-a)(c-b)}{c}F(c-a+1,c-b+1,c+1;\zeta) \right| \\
& \leq  \left|  \frac{(c-a)(c-b)}{c} \frac{\Gamma(c+1)}{\Gamma(c-b)} \right||z|^{-\frac{1}{2}-\mu} 1+o(|z|^{-1})\\
& \leq \left|  \frac{(c-a)(c-b)}{c} \right| \left| \frac{\Gamma(c+1)}{\Gamma(b+1)} \right||z|^{-\frac{1}{2}-\mu} 1+o(|z|^{-1})\\
& \leq \frac{C}{|z|}1+o(|z|^{-1}).
\end{split}
\]
We can easily generalise this result for every order of derivative:
\[
\left|\frac{\mathrm{d^k}}{\mathrm{d}\zeta^k}F(a,b,c;\zeta)\right| \leq \frac{C}{|z|^{k-1}}1+o(|z|^{-1}).
\]
\end{comment}

We can now, similarlyapply Lebesgue's dominated convergence theorem, using the last estimate to conclude about the fact that the integral term
\begin{equation*}
e^{-iz\frac{3\beta-\alpha}{2}}\int_{\frac{\alpha+\beta}{2}}^{\frac{3\beta-\alpha}{2}} \partial_t^{p+1}\left[ K^+((\alpha+\beta)/2,t) F\left(c-a,c-b,c; \frac{1}{1+e^{2t}} \right)\right] e^{itz} \, \mathrm{d}t
\label{Integral}
\end{equation*}
tends to zero as $|z|$ tends to infinity.
\begin{comment}
Indeed, Proposition \ref{changement} gives us that the derivatives of an hypergeometric function is still an hypergeometric function, with a coefficient $(iz)^{-k}$ for the $k^{th}$ derivative. Since the hypergeometric tends to $1$ in $S_2$, it is bounded on $\C^-$. Lebesgue's dominated convergence theorem finally proves that the quantity 
\[
e^{-iz\frac{3\beta-\alpha}{2}}\int_{\frac{\alpha+\beta}{2}}^{\frac{3\beta-\alpha}{2}} \partial_t^{p+1}\left[ K^+((\alpha+\beta)/2,t) F\left(c-a,c-b,c; \frac{1}{1+e^{2t}} \right)\right] {e^{itz}} \, \mathrm{d}t
\]
tends to 0.